\documentclass[12pt]{article}
\usepackage{amssymb,amsmath,amsthm,tikz,multirow,nccrules}
\usetikzlibrary{arrows,calc}

\title{Tilings of the Sphere by Congruent Pentagons I: Edge Combinations $a^2b^2c$ and $a^3bc$}
\author{Erxiao Wang\thanks{Research was supported by ZJNU Shuang-Long Distinguished Professorship Fund No. YS304319159.}, 
Zhejiang Normal University \\
Min Yan\thanks{Research was supported by Hong Kong RGC General Research Fund 16303515 and 16305920.}, 
Hong Kong University of Science and Technology}

\usepackage[hidelinks, hyperindex]{hyperref}

\hyperref[sec:function]{}

\newcommand{\dash}{\hspace{0.1em}\dashrule{0.7}{2.4 1 2.4 1 2.4}\hspace{0.1em}} 
\newcommand{\thin}{\hspace{0.1em}\rule{0.7pt}{0.8em}\hspace{0.1em}}
\newcommand{\thick}{\hspace{0.1em}\rule{1.5pt}{0.8em}\hspace{0.1em}}

\newcommand{\pentagon}{\tikz \foreach \a in {0,...,4} \draw[rotate=72*\a] (18:0.17) -- (90:0.17);}

\newtheorem{theorem}{Theorem}
\newtheorem{lemma}[theorem]{Lemma}

\newtheorem*{theorem*}{Theorem}

\theoremstyle{definition}
\newtheorem*{definition*}{Definition}
\newtheorem*{case*}{Case}
\newtheorem*{subcase*}{Subcase}

\theoremstyle{remark}

\numberwithin{equation}{section}

\begin{document}

\maketitle

\begin{abstract}
We develop the basic tools for classifying edge-to-edge tilings of the sphere by congruent pentagons. Then we prove that, for the edge combination $a^2b^2c$, such tilings are three two-parameter families of pentagonal subdivisions of the Platonic solids, with $12$, $24$ and $60$ tiles. We also prove that, for the edge combination $a^3bc$, such tilings are two unique double pentagonal subdivisions of the Platonic solids, with $48$ and $120$ tiles.

\end{abstract}

\section{Introduction}

Mathematicians have studied tilings for more than 100 years. A lot is known about tilings of the plane and the Euclidean space \cite{rao,zong}. However, results about tilings of the sphere are relatively rare. A major achievement in this regard is the complete classification of edge-to-edge tilings of the sphere by congruent triangles. The classification was started by Sommerville \cite{so} in 1923 and completed by Ueno and Agaoka \cite{ua} in 2002. For tilings of the sphere by congruent pentagons, we know the classification for the minimal case of $12$ tiles \cite{ay1,gsy}.

Spherical tilings are relatively easier to study than planar tilings, because the former involve only finitely many tiles. The classifications in \cite{gsy,ua} not only give the complete list of tiles, but also the ways the tiles are fit together. It is not surprising that such kind of classifications for the planer tilings are only possible under various symmetry conditions, because the quotients of the plane by the symmetries often become compact.

Like the earlier works, we restrict ourselves to edge-to-edge tilings of the sphere by congruent polygons, such that all vertices have degree $\ge 3$. These are mild and natural assumptions that simplify the discussion. The polygon in such a tiling must be triangle, quadrilateral, or pentagon (see \cite{ua2}, for example). We believe that pentagonal tilings should be relatively easier to study than quadrilateral tilings because $5$ is an ``extreme'' among $3$, $4$, $5$. Indeed, many properties on pentagonal tilings in Section \ref{basic_facts} have no or less restrictive counterparts for quadrilateral tilings. Moreover, a preliminary exploration on quadrilateral tilings \cite{ua2} also showed the difficulty of the problem.

The lengths of five edges of the pentagon in our tiling may have five possible combinations (Lemma \ref{edge_combo}): $a^2b^2c,a^3bc,a^3b^2,a^4b,a^5$. Here $a^2b^2c$ means the five edge lengths are $a,a,b,b,c$, with $a,b,c$ distinct. In this paper, we classify for the first two edge combinations. The pentagons are given by Figure \ref{pentagon}, where $a,b,c$ are the normal, thick, and dashed lines. 

\begin{figure}[htp]
\centering
\begin{tikzpicture}[>=latex,scale=1]


\draw
	(234:1) -- (162:1) -- (90:1);

\draw[line width=1.5]
	(-54:1) -- (18:1) -- (90:1);

\draw[densely dashed]
	(234:1) -- (-54:1);

\node at (90:0.75) {$\alpha$};
\node at (162:0.75) {$\beta$};
\node at (15:0.75) {$\gamma$};
\node at (234:0.75) {$\delta$};
\node at (-54:0.75) {$\epsilon$};


\begin{scope}[xshift=3cm]

\draw
	(162:1) -- (234:1) -- (-54:1) -- (18:1);

\draw[line width=1.5]
	(162:1) -- (90:1);

\draw[densely dashed]
	(18:1) -- (90:1);

\node at (90:0.75) {$\alpha$};
\node at (162:0.75) {$\beta$};
\node at (15:0.75) {$\gamma$};
\node at (234:0.75) {$\delta$};
\node at (-54:0.75) {$\epsilon$};

\end{scope}

\begin{scope}[xshift=5cm]

\draw
	(0,0.4) -- node[above=-2] {\small $a$} ++(1,0);

\draw[line width=1.5]
	(0,-0.1) -- node[above=-2] {\small $b$} ++(1,0);

\draw[dashed]
	(0,-0.6) -- node[above=-2] {\small $c$} ++(1,0);

\end{scope}

\end{tikzpicture}
\caption{Pentagons with the edge combinations $a^2b^2c$ and $a^3bc$.}
\label{pentagon}
\end{figure}

\begin{theorem*}
Edge-to-edge tilings of the sphere by congruent pentagons with the edge combination $a^2b^2c$ ($a,b,c$ distinct) are the following:
\begin{enumerate}
\item Pentagonal subdivision of the tetrahedron, with $12$ tiles.
\item Pentagonal subdivision of the octahedron (or cube), with $24$ tiles.
\item Pentagonal subdivision of the icosahedron (or dodecahedron), with $60$ tiles.
\end{enumerate}
\end{theorem*}

\begin{figure}[htp]
\centering
\begin{tikzpicture}[>=latex,scale=1]


\draw
	(234:0.33) -- (162:0.33) -- (90:0.33)
	(162:0.33) -- (162:0.65)
	(18:0.65) -- (54:0.8) -- (90:0.65)
	(54:0.8) -- (54:1.2)
	(-18:0.8) -- (-54:0.65) -- (-90:0.8)
	(-54:0.65) -- (-54:0.33)
	(126:1.2) -- (198:1.2) -- (270:1.2)
	(198:1.2) -- (198:0.8);

\draw[line width=1.5]
	(-54:0.33) -- (18:0.33) -- (90:0.33)
	(18:0.33) -- (18:0.65)
	(198:0.8) -- (234:0.65) -- (270:0.8)
	(234:0.33) -- (234:0.65)
	(-90:1.2) -- (-18:1.2) -- (54:1.2)
	(-18:1.2) -- (-18:0.8)
	(90:0.65) -- (126:0.8) -- (162:0.65)
	(126:0.8) -- (126:1.2);
	
\draw[dash pattern=on 1pt off 1pt]
	(-54:0.33) -- (234:0.33)
	(90:0.33) -- (90:0.65) 
	(18:0.65) -- (-18:0.8)
	(162:0.65) -- (198:0.8)
	(-90:1.2) -- (-90:0.8)
	(54:1.2) -- (126:1.2);


\foreach \a in {0,1,2,3}
{

\begin{scope}[xshift=3.5cm,rotate=90*\a]

\draw
	(0.4,0) -- (0.6,0.25) -- (0.25,0.6)
	(0.6,0.25) -- (15:1)
	(0:1.3) -- (30:1.3) -- (60:1.3)
	(45:1) -- (30:1.3);

\draw[line width=1.5]
	(0,0) -- (0.4,0)
	(-45:1) -- (-15:1) -- (15:1)
	(0.6,-0.25) -- (-15:1) -- (0:1.3)
	(60:1.3) -- (60:1.6);

\draw[dash pattern=on 1pt off 1pt]
	(0.25,0.6) -- (0,0.4)
	(15:1) -- (45:1)
	(60:1.3) -- (90:1.3);

\end{scope}
}


\foreach \a in {0,...,4}
{
\begin{scope}[xshift=8.2cm, rotate=72*\a]

\draw
	(0,0) -- (18:0.45) -- (36:0.7) -- (72:0.7)
	(0:0.7) -- (6:1.1)
	(30:1.1) -- (36:0.7)
	(-18:1.1) -- (6:1.1) -- (30:1.1)
	(-9:1.4) -- (6:1.1) -- (21:1.4) -- (39:1.4) -- (54:1.1)
	(-9:1.4) -- (0:1.6) -- (12:1.6)
	(39:1.4) -- (48:1.6) 
	(12:1.6) -- (30:1.9) -- (48:1.6)
	(54:1.9) -- (30:1.9) -- (6:1.9)
	(-18:1.9) -- (0:1.6)
	(6:1.9) -- (-6:2.2)
	(30:1.9) -- (42:2.2)
	(-30:2.2) -- (-6:2.2) -- (18:2.2) -- (18:2.5);

\draw[dash pattern=on 1pt off 1pt]
	(18:0.45) -- (0:0.7)
	(30:1.1) -- (54:1.1)
	(-9:1.4) -- (-24:1.6)
	(12:1.6) -- (21:1.4)
	(6:1.9) -- (-18:1.9)
	(-54:2.2) -- (-30:2.2);

\draw[line width=1.5]
	(0,0) -- (18:0.45)
	(0:0.7) -- (6:1.1)
	(-18:1.1) -- (6:1.1) -- (30:1.1)
	(-9:1.4) -- (6:1.1) -- (21:1.4)
	(30:1.9) -- (42:2.2)
	(12:1.6) -- (30:1.9) -- (48:1.6)
	(54:1.9) -- (30:1.9) -- (6:1.9)
	(18:2.2) -- (18:2.5);

\end{scope}
}
	
\end{tikzpicture}
\caption{Pentagonal subdivision tilings for $a^2b^2c$.}
\label{subdivision_tiling}
\end{figure}

Pentagonal subdivision is introduced in Section \ref{pdiv}. The operation can be applied to any tiling on an oriented surface. Since the dual tiling has the same pentagonal subdivision tiling, the five Platonic solids give three pentagonal subdivision tilings, in Figure \ref{subdivision_tiling}. We already proved in \cite{ay1,gsy} that edge-to-edge tilings of the sphere by (the minimal number of) $12$ congruent pentagons is the deformed dodecahedron given by the first of Figure \ref{subdivision_tiling}, and is exactly the first family in the theorem.

The pentagonal subdivision tilings allow two free parameters. Therefore tilings for the edge combination $a^2b^2c$ form a two dimensional moduli. We have detailed description of the moduli in \cite{wy3}.

\begin{theorem*}
Edge-to-edge tilings of the sphere by congruent pentagons with the edge combination $a^3bc$ ($a,b,c$ distinct) are the following:
\begin{enumerate}
\item Double pentagonal subdivision of the octahedron (or cube), with $48$ tiles.
\item Double pentagonal subdivision of the icosahedron (or dodecahedron), with $120$ tiles.
\end{enumerate}
\end{theorem*} 

Double pentagonal subdivision is introduced in Section \ref{dpdiv}. Unlike pentagonal subdivision, each tiling allows only one specific pentagon, and we provide the exact values in Section \ref{dpdiv}. 

We also have the double pentagonal subdivision of the tetrahedron, with $24$ tiles. However, the tiling is not included in the theorem because the pentagon has $b=c$, which means the tiling has the edge combination $a^3b^2$. We still give the exact values for this tiling in Section \ref{dpdiv}. We also remark that the tiling is a degenerate case ($a=c$ in $a^2b^2c$) of the pentagonal subdivision of the octahedron. The discussion of the tiling from this viewpoint can be found in \cite[Section 7.1]{wy2}.

\begin{figure}[htp]
\centering
\begin{tikzpicture}[>=latex,scale=1]


\foreach \a in {0,1,2,3}
{
\begin{scope}[rotate=-15+90*\a]	

\draw
	(0,0) -- (35:0.4)
	(0:1) -- (-5:1.3) -- (0:1.6) -- (15:1.8) -- (-15:1.9) -- (0:2.2)
	(45:1.9) -- (75:1.9)
	(50:1.3) -- (45:1.6) -- (30:1.8)
	(60:2.2) -- (60:2.6)
	(22.5:0.8) -- (22.5:1.3)
	(-5:1.3) -- (-40:1.3)
	(-22.5:0.8) -- (35:0.4)
	(15:1.8) -- (30:1.8)
	(22.5:1.3) -- (45:1.6) -- (67.5:1.6)
	(67.5:1.6) -- (45:1.9)
	(0:2.2) -- (-30:2.2)
	(-22.5:0.8) -- (0:1) -- (22.5:0.8);	

\draw[line width=1.5]
	(22.5:0.8) -- (45:1) -- (67.5:0.8)
	(0:2.2) -- (30:2.2) -- (45:1.9)
	(-5:1.3) -- (0:1.6) -- (15:1.8);

\draw[densely dashed]
	(35:0.4) -- (45:1) -- (50:1.3)
	(30:1.8) -- (30:2.2) -- (60:2.2)
	(-22.5:1.6) -- (0:1.6) -- (22.5:1.3);

\end{scope}
}


\foreach \a in {0,...,4}
{
\begin{scope}[xshift=5.5cm, rotate=72*\a]	

\draw
	(0,0) -- (-18:0.3)
	(4:0.8) -- (0:1) 
	(18:0.5) -- (36:0.6) -- (54:0.6)
	(36:0.6) -- (32:0.8) -- (36:1) -- (36:1.2) -- (36:1.4) -- (25:1.3) 
	(18:0.5) -- (54:0.3)
	(54:0.6) -- (54:0.9)
	(4:0.8) -- (32:0.8)
	(36:1.9) -- (27:1.6)
	(9:1.2) -- (0:1) -- (-9:1.2)
	(-27:1.7) -- (-36:1.9) %
	(9:1.5) -- (0:1.6) -- (-6:1.4)
	(-27:1.5) -- (-36:1.4) 
	(0:1.6) -- (-9:1.7) -- (0:1.9)
	(9:1.2) -- (-6:1.4)
	(18:1) -- (25:1.3)
	(-9:1.2) -- (-36:1.2)
	(-27:1.5) -- (-45:1.6)
	(-9:1.7) -- (-27:1.7)
	(9:1.5) -- (18:1.8)
	(0:1) -- (18:1)
	(54:0.9) -- (72:1) 
	(0:1.9) -- (9:2.1) -- (27:2.1) -- (36:1.9)
	(18:1.8) -- (36:1.9) -- (54:1.9) -- (45:2.1) -- (63:2.1)
-- (81:2.1)
	(6:2.35) -- (-9:2.1)
	(54:2.7) -- (54:2.35) -- (78:2.35)
	;
	
\draw[line width=1.5]
	(-18:0.6) -- (0:0.6) -- (18:0.5)
	(25:1.3) -- (16:1.4) -- (9:1.5)
	(-6:1.4) -- (-16:1.4) -- (-27:1.5)
	(-9:1.7) -- (0:1.9) -- (9:2.1) 
	(32:0.8) -- (36:1) -- (36:1.2)
	(45:2.1) -- (30:2.35) -- (6:2.35);

\draw[dash pattern=on 1pt off 1pt]
	(-18:0.3) -- (0:0.6) -- (4:0.8)
	(18:1) -- (36:1) -- (54:0.9)
	(27:1.6) -- (16:1.4) -- (9:1.2)
	(-9:1.2) -- (-16:1.4) -- (-27:1.7)
	(-18:1.9) -- (0:1.9) -- (18:1.8)
	(27:2.1) -- (30:2.35) -- (54:2.35);

\end{scope}
}

\end{tikzpicture}
\caption{Double pentagonal subdivision tilings for $a^3bc$.}
\label{dsubdivision_tiling}
\end{figure}

We developped the general combinatorial theory of pentagonal subdivision and double pentagonal subdivision on any surface in \cite{yan2}. There is a more fundamental theory of simple pentagonal subdivision of quadrilateral tilings, that underlies the two pentagonal subdivisions.

The edge combination $a^3b^2$ is more complicated due to less edge length information. The second paper \cite{wy2} of the series will handle this combination, where we find fifteen tilings. The edge combination $a^5$ (i.e., equilateral) has no edge length information, and requires a completely different technique. The third paper \cite{awy} will handle this case, where we find eight tilings. 

After the current series of three papers, the remaining case is the edge combination $a^4b$, which we call {\em almost equilateral}. The case is much more challenging, and will be the subject of another series \cite{ly1,ly2}.

This paper is organized as follows. Section \ref{basic_facts} is the basic facts and techniques about tilings of the sphere by congruent pentagons. Section \ref{div} introduces pentagonal and double pentagonal subdivisions. Sections \ref{2a2bc} and \ref{3abc} prove the two classification theorems.

We would like to thank Ka Yue Cheuk and Ho Man Cheung. Some of their initial work on the pentagonal tilings of the sphere are included in this paper. We would also like to thank Hoi Ping Luk, whose MPhil thesis \cite{luk} underlies the computation of the anglewise vertex combination in Section \ref{avc}, and who contributed to the efficient notation for the adjacent angle deduction in Section \ref{aad}.

\section{Basic Facts}
\label{basic_facts}

By \cite[Lemma 1]{gsy}, any tiling of the sphere has a simple tile, in the sense that the boundary does not cross itself. Since all tiles in our tiling are congruent, all tiles are simple.

Throughout this paper, a pentagonal tiling is always an {\em edge-to-edge} tiling of the sphere by {\em simple} pentagons, such that all vertices have degree $\ge 3$.

\subsection{Vertex}

Let $v,e,f$ be the numbers of vertices, edges, and tiles. Let $v_k$ be the number of vertices of degree $k$. We have
\begin{align*}
2
&=v-e+f, \\
2e=5f
&=\sum_{k=3}^{\infty}kv_k=3v_3+4v_4+5v_5+\cdots, \\
v
&=\sum_{k=3}^{\infty}v_k=v_3+v_4+v_5+\cdots.
\end{align*}
Then it is easy to derive $2v=3f+4$ and  
\begin{align}
\tfrac{f}{2}-6
&=\sum_{k\ge 4}(k-3)v_k=v_4+2v_5+3v_6+\cdots, \label{vcountf} \\
v_3
&=20+\sum_{k\ge 4}(3k-10)v_k=20+2v_4+5v_5+8v_6+\cdots. \label{vcountv}
\end{align}
By \eqref{vcountf}, $f$ is an even integer $\ge 12$. Since tilings by $12$ congruent pentagons have been classified in \cite{ay1, gsy}, we may assume $f>12$. Moreover, by \eqref{vcountf}, $f=14$ implies $v_4=1$ and $v_k=0$ for $k>4$. By \cite[Theorem 1]{yan}, this is impossible. Therefore we will always assume that $f$ is an even integer $\ge 16$, in all the papers in the series.

The equality \eqref{vcountv} shows that most vertices have degree $3$. We call vertices of degree $>3$ {\em high degree} vertices.

\begin{lemma}\label{base_tile}
Any pentagonal tiling of the sphere has a tile, such that four vertices have degree $3$ and the fifth vertex has degree $3$, $4$ or $5$.
\end{lemma}

We call the tile in the lemma a {\em special tile}. We have three types of special tiles, which we call $3^5$-tile, $3^44$-tile, and $3^45$-tile. The neighbourhood of this special tile is given by the first three of Figure \ref{nhd}. For $3^44$-tile and $3^45$-tile, we indicate the fifth vertex $H$ (of degree $4$ or $5$) by $\bullet$. The fourth of Figure \ref{nhd} is the common part of the three types of neighbourhoods, which we call a {\em partial neighbourhood}. We always label the tiles in the partial neighbourhood as in the picture.

\begin{figure}[htp]
\centering
\begin{tikzpicture}[>=latex]

\foreach \x in {0,...,4}
\draw[xshift=-3cm, rotate=-72*\x]
	(-54:0.5) -- (18:0.5) -- (18:0.9) -- (-18:1.2) -- (-54:0.9);

\foreach \a in {0,1,2}
{
\begin{scope}[xshift=3*\a cm]

\foreach \x in {0,1,2}
\draw[rotate=-72*\x]
	(18:0.5) -- (18:0.9) -- (-18:1.2) -- (-54:0.9) -- (-54:0.5);

\foreach \x in {0,...,4}
\draw[rotate=-72*\x]
	(18:0.5) -- (90:0.5);

\fill 
	(90:0.5) circle (0.1);

\end{scope}
}

\foreach \a in {0,1}
{
\node[xshift=3*\a cm] at (90:0.2) {\small $H$};

\draw[xshift=6*\a cm]
	(90:0.5) -- (65:1) -- (35:1.2) -- (18:0.9)
	(90:0.5) -- (115:1) -- (145:1.2) -- (162:0.9);
}

\draw 
	(65:1) -- (75:1.4) -- (105:1.4) -- (115:1);
	
\draw[xshift=3cm]
	(90:0.5) -- (55:0.9) -- (35:1.2) -- (18:0.9)
	(90:0.5) -- (125:0.9) -- (145:1.2) -- (162:0.9)
	(90:0.5) -- (90:1) -- (80:1.4) -- (60:1.4) -- (55:0.9)
	(90:0.5) -- (90:1) -- (100:1.4) -- (120:1.4) -- (125:0.9);

\begin{scope}[xshift=6cm]

\node[draw,shape=circle, inner sep=0.5] at (0,0) {\small $1$};
\node[draw,shape=circle, inner sep=0.5] at (48:0.75) {\small $2$};
\node[draw,shape=circle, inner sep=0.5] at (-18:0.75) {\small $3$};
\node[draw,shape=circle, inner sep=0.5] at (-90:0.75) {\small $4$};
\node[draw,shape=circle, inner sep=0.5] at (198:0.75) {\small $5$};
\node[draw,shape=circle, inner sep=0.5] at (135:0.75) {\small $6$};

\end{scope}

\end{tikzpicture}
\caption{Neighborhood and partial neighbourhood of a special tile.}
\label{nhd}
\end{figure}

\begin{proof}
If a pentagonal tiling of the sphere has no special tile, then each tile either has at least one vertex of degree $\ge 6$, or has at least two vertices of degree $4$ or $5$. Since a degree $k$ vertex is shared by at most $k$ tiles, the number of tiles of the first kind is $\le \sum_{k\ge 6}kv_k$, and the number of tiles of the second kind is $\le\frac{1}{2}(4v_4+5v_5)$. Therefore we have
\[
f\le \tfrac{1}{2}(4v_4+5v_5)+\sum_{k\ge 6}kv_k.
\]
On the other hand, by \eqref{vcountf}, we have
\begin{align*}
f-\tfrac{1}{2}(4v_4+5v_5)-\sum_{k\ge 6}kv_k 
&=12+\sum_{k\ge 4}2(k-3)v_k-\tfrac{1}{2}(4v_4+5v_5)-\sum_{k\ge 6}kv_k \\
&= 12 + \tfrac{3}{2}v_5+\sum_{k\ge 6}(k-6)v_k
>0.
\end{align*}
We get a contradiction.
\end{proof}

\begin{lemma}\label{base_tile2}
If a pentagonal tiling of the sphere has no $3^5$-tile, then $f\ge 24$. Moreover, if $f=24$, then each tile is a $3^44$-tile. 
\end{lemma}

\begin{proof}
If there is no $3^5$-tile, then each tile has at least one vertex of degree $\ge 4$. This implies $f\le\sum_{k\ge 4}kv_k$. By \eqref{vcountf}, we have
\begin{align*}
f=2f-f
&\ge 24+ \sum_{k\ge 4}4(k-3)v_k -\sum_{k\ge 4}kv_k \\
&=24+ \sum_{k\ge 4}3(k-4)v_k
\ge 24.
\end{align*}
Moreover, the equality happens if and only if $v_k=0$ for $k>4$ and $f=\sum_{k\ge 4}kv_k=4v_4$. By no $3^5$-tile, this means that each tile is a $3^44$-tile.
\end{proof}

\begin{lemma}\label{base_tile3}
If a pentagonal tiling of the sphere has no $3^5$-tile and no $3^44$-tile, then $f\ge 60$. Moreover, if $f=60$, then each tile is a $3^45$-tile. 
\end{lemma}

\begin{proof}
If there is no $3^5$-tile and no $3^44$-tile, then each tile either has at least two vertices of degree $4$, or has at least one vertex of degree $\ge 5$. This implies $f\le \frac{1}{2}4v_4+\sum_{k\ge 5}kv_k$. By \eqref{vcountf}, we have
\begin{align*}
f=5f-4f
&\ge 60+\sum_{k\ge 4}10(k-3)v_k-8v_4-\sum_{k\ge 5}4kv_k \\
&=60+2v_4+\sum_{k\ge 5}6(k-5)v_k
\ge 60.
\end{align*}
Moreover, the equality happens if and only if $v_4=v_6=v_7=\cdots=0$ and $f= \frac{1}{2}4v_4+\sum_{k\ge 5}kv_k=5v_5$. By no $3^5$-tile and no $3^44$-tile, this means that each tile is a $3^45$-tile.
\end{proof}

\subsection{Angle}

The sum of all angles ({\em angle sum}) at a vertex is $2\pi$. The following is the {\em angle sum for pentagon}.

\begin{lemma}\label{anglesum}
If all tiles in a tiling of the sphere by $f$ pentagons have the same five angles $\alpha,\beta,\gamma,\delta,\epsilon$, then 
\[
\alpha+\beta+\gamma+\delta+\epsilon
=(3 + \tfrac{4}{f})\pi.
\]
\end{lemma}

\begin{proof}
Since the angle sum of each vertex is $2\pi$, the total sum of all angles in the tiling is $2\pi v$. Moreover, the sum of five angles in each tile is $\Sigma=\alpha+\beta+\gamma+\delta+\epsilon$, and the total sum of all angles is $f\Sigma$. Therefore we get $2\pi v = f\Sigma$. By $3f=2v-4$, we further get
\[
\Sigma = 2\pi\tfrac{v}{f} = (3 + \tfrac{4}{f})\pi. \qedhere
\]
\end{proof}

The lemma does not require that the angles are arranged in the same way in all tiles, and does not require that the edges are straight (i.e., great arcs). However, if we additionally know that all edges are straight, then all tiles have the same area $\Sigma-3\pi$, and the equality in the lemma follows from the fact that the total area $f(\Sigma-3\pi)$ is the area $4\pi$ of the sphere.

The angles in Lemma \ref{anglesum} refer to the values, and some angles among the five may have the same value. For example, if the five values are $\alpha,\alpha,\alpha,\beta,\beta$, with $\alpha\ne\beta$ (distinct values), then we say the pentagon has the {\em angle combination} $\alpha^3\beta^2$. The following is about the distribution of angles.

\begin{lemma}\label{deg3a}
If an angle appears at every degree $3$ vertex in a tiling of the sphere by pentagons with the same angle combination, then the angle appears at least two times in the pentagon.
\end{lemma}

\begin{proof}
If an angle $\theta$ appears only once in the pentagon, then the total number of times $\theta$ appears in the whole tiling is $f$, and the total number of non-$\theta$-angles is $4f$. If $\theta$ appears at every degree $3$ vertex, then $f\ge v_3$, and non-$\theta$-angles appear $\le 2v_3$ times at degree $3$ vertices. Moreover, non-$\theta$-angles appear $\le \sum_{k\ge 4}kv_k$ times at high degree vertices. Therefore we have
\[
4v_3 \le 4f \le 2v_3+\sum_{k\ge 4}kv_k.
\]
On the other hand, by \eqref{vcountv}, we have
\begin{align*}
4v_3-2v_3-\sum_{k\ge 4}kv_k 
&=40+\sum_{k\ge 4}(2(3k-10)-k)v_k \\
&=40+\sum_{k\ge 4}5(k-4)v_k
>0.
\end{align*}
We get a contradiction.
\end{proof}

Unlike Lemma \ref{anglesum}, which is explicitly about the values of angles, Lemma \ref{deg3a} only counts the number of angles. The counting only requires us to distinguish angles. Besides values, we may also use the bounding edge lengths to distinguish angles. For example, the five angles in the first of Figure \ref{pentagon} are distinguished as $ab$-angle $\alpha$, $a^2$-angle $\beta$, $b^2$-angle $\gamma$, $ac$-angle $\delta$, and $bc$-angle $\epsilon$. In fact, we may even use both values and bounding edges to distinguish angles. Lemma \ref{deg3a} (and the subsequent Lemmas \ref{deg3b}, \ref{deg3c}, \ref{hdeg}) applies to any way that distinguishes angles. 

\begin{lemma}\label{deg3b}
If an angle appears at least twice at every degree $3$ vertex in a tiling of the sphere by pentagons with the same angle combination, then the angle appears at least three times in the pentagon. 
\end{lemma}

\begin{proof}
If an angle $\theta$ appears only once in the pentagon, then by Lemma \ref{deg3a}, it cannot appear at every degree $3$ vertex. If it appears twice in the pentagon, then the total number of $\theta$ in the whole tiling is $2f$, and the total number of non-$\theta$-angles is $3f$. If we also know that $\theta$ appears at least twice at every degree $3$ vertex, then $2f\ge 2v_3$, and non-$\theta$-angles appear $\le v_3$ times at degree $3$ vertices. Moreover, non-$\theta$-angles appear $\le \sum_{k\ge 4}kv_k$ times at high degree vertices. Therefore 
\[
3v_3 \le 3f \le v_3+\sum_{k\ge 4}kv_k.
\]
This leads to the same contradiction as in the proof of Lemma \ref{deg3a}.
\end{proof}

The proof of Lemma \ref{deg3b} can be easily modified to get the following.

\begin{lemma}\label{deg3c}
If two angles together appear at least twice at every degree $3$ vertex in a tiling of the sphere by pentagons with the same angle combination, then the two angles together appear at least three times in the pentagon. 
\end{lemma}

The following is about angles not appearing at degree $3$ vertices.

\begin{lemma}\label{hdeg}
Suppose an angle $\theta$ does not appear at degree $3$ vertices in a tiling of the sphere by pentagons with the same angle combination.
\begin{enumerate}
\item There can be at most one such angle $\theta$.
\item The angle $\theta$ appears only once in the pentagon.
\item $2v_4+v_5\ge 12$.
\item One of $\theta^3\rho$, $\theta^4$, $\theta^5$ is a vertex, where $\rho\ne\theta$.
\end{enumerate}
\end{lemma}

The first statement implies that the angle $\rho$ in the fourth statement must appear at a degree $3$ vertex.

\begin{proof}
Suppose two angles $\theta_1$ and $\theta_2$ do not appear at degree $3$ vertices. Then the total number of times these two angles appear is at least $2f$, and is at most the total number $\sum_{k\ge 4}kv_k$ of angles at high degree vertices. Therefore we have $2f\le \sum_{k\ge 4}kv_k$. On the other hand, by \eqref{vcountf}, we have
\[
2f - \sum_{k\ge 4}kv_k  
= 24 + \sum_{k\ge 4}3(k-4)v_k
>0.
\]
The contradiction proves the first statement.

The argument above also applies to the case $\theta_1=\theta_2$, which means the same angle appearing at least twice in the pentagon. This proves  the second statement. 

The first two statements imply that $\theta$ appears exactly $f$ times. Since this should be no more than the total number $\sum_{k\ge 4}kv_k$ of angles at high degree vertices, by \eqref{vcountf}, we have
\[
0
\ge f - \sum_{k\ge 4}kv_k  
= 12 - 2v_4 - v_5 + \sum_{k\ge 6}(k-6)v_k.
\]
This implies the third statement.

For the last statement, we assume that $\theta^3\rho,\theta^4,\theta^5$ are not vertices. This means that $\theta$ appears at most twice at every degree $4$ vertex, and at most four times at every degree $5$ vertex. Since $\theta$ also does not appear at degree $3$ vertices, the total number of times $\theta$ appears is $\le 2v_4+4v_5+\sum_{k\ge 6}kv_k$. However, the total number of times $\theta$ appears should also be $f$. Therefore $f\le 2v_4+4v_5+\sum_{k\ge 6}kv_k$. On the other hand, by \eqref{vcountf}, we have
\[
f-2v_4-4v_5-\sum_{k\ge 6}kv_k
= 12 + \sum_{k\ge 6}(k-6)v_k
>0.
\]
We get a contradiction.
\end{proof}

\subsection{Edge}

\begin{lemma}\label{edge_combo}
In an edge-to-edge tiling of the sphere by congruent pentagons, the edge lengths of the pentagon are arranged in one of the five ways in Figure \ref{edges1}, with distinct edge lengths $a,b,c$. 
\end{lemma}

The lemma is an extension of \cite[Proposition 7]{gsy}. We always use the normal line, the thick line and the dashed line to indicate distinct edge lengths $a,b,c$. 

The proof of lemma does not use any angle information. Therefore the lemma also applies to tilings by {\em edge congruent} pentagons. 

\begin{figure}[htp]
\centering
\begin{tikzpicture}[>=latex,scale=1]

\begin{scope}[xshift=-2.2cm]

\draw
	(0,0.4) -- node[above=-2] {\small $a$} ++(1,0);

\draw[line width=1.5]
	(0,-0.1) -- node[above=-2] {\small $b$} ++(1,0);

\draw[dashed]
	(0,-0.6) -- node[above=-2] {\small $c$} ++(1,0);

\end{scope}


\draw
	(90:0.8) -- (162:0.8) -- (234:0.8);

\draw[line width=1.5]
	(90:0.8) -- (18:0.8) -- (-54:0.8);

\draw[dashed]
	(-54:0.8) -- (-126:0.8);
	
\node at (0,0) {\small $a^2b^2c$};


\begin{scope}[xshift=2cm]

\foreach \x in {2,...,4}
\draw[rotate=72*\x]
	(90:0.8) -- (18:0.8);

\draw[line width=1.5]
	(162:0.8) -- (90:0.8);

\draw[dashed]
	(18:0.8) -- (90:0.8);
	
\node at (0,0) {\small $a^3bc$};

\end{scope}


\begin{scope}[xshift=4cm]

\draw
	(162:0.8) -- (234:0.8) -- (-54:0.8) -- (18:0.8);

\draw[line width=1.5]
	(162:0.8) -- (90:0.8) -- (18:0.8);
	
\node at (0,0) {\small $a^3b^2$};

\end{scope}


\begin{scope}[xshift=6cm]

\foreach \x in {-1,...,2}
\draw[rotate=72*\x]
	(90:0.8) -- (18:0.8);

\draw[line width=1.5]
	(-54:0.8) -- (-126:0.8);
	
\node at (0,0) {\small $a^4b$};

\end{scope}


\begin{scope}[xshift=8cm]

\foreach \x in {1,...,5}
\draw[rotate=72*\x]
	(90:0.8) -- (18:0.8);
\node at (0,0) {\small $a^5$};

\end{scope}

\end{tikzpicture}
\caption{Edge combinations suitable for tiling, with $a,b,c$ distinct.}
\label{edges1}
\end{figure}

\begin{proof}
By purely numerical consideration, there are seven possible edge combinations ($a,b,c,d,e$ are distinct)
\[
abcde,\;
a^2bcd,\;
a^2b^2c,\;
a^3bc,\;
a^3b^2,\;
a^4b,\;
a^5.
\]
Then we need to consider various ways the edges are arranged. Without loss of generality, the following are all the possible edge arrangements:
\begin{itemize}
\item $abcde$: first of Figure \ref{edges2}.
\item $a^2bcd$: second of Figure \ref{edges2} (two $a$ adjacent), third of Figure \ref{edges2} (two $a$ separated).
\item $a^2b^2c$: forth of Figure \ref{edges2} (two $a$ separated, two $b$ adjacent), fifth of Figure \ref{edges2} (two $a$ separated, two $b$ separated), first of Figure \ref{edges1} (two $a$ adjacent, two $b$ adjacent).  
\item $a^3bc$: sixth of Figure \ref{edges2} ($b,c$ separated), second of Figure \ref{edges1} ($b,c$ adjacent).
\item $a^3b^2$: seventh of Figure \ref{edges2} (two $b$ separated), third of Figure \ref{edges1} (two $b$ adjacent).
\item $a^4b$: fourth of Figure \ref{edges1}.
\item $a^5$: fifth of Figure \ref{edges1}.
\end{itemize}

\begin{figure}[htp]
\centering
\begin{tikzpicture}[>=latex,scale=1]

\foreach \a in {0,...,4}
\foreach \b in {0,...,6}
\draw[xshift=1.8*\b cm, rotate=72*\a]
	(18:0.8) -- (90:0.8);

\foreach \b in {0,...,3}
\draw[xshift=1.8*\b cm]
	(-54:0.8) -- node[fill=white,inner sep=1] {\small $x$} 
	(-54:1.5);
	
\foreach \a/\b in {126/0, 54/1, 126/1, 198/2, -18/2, 198/3, -18/3, 54/4, 198/4, 54/5, 126/5, -90/5, 54/6, 126/6, -90/6}
\node[xshift=1.8*\b cm, fill=white, inner sep=1] at (\a:0.65) {\small $a$};

\foreach \a/\b in {198/0, 198/1, 126/2, 54/3, 126/3, 126/4, -18/4, 198/5, -18/6, 198/6}
\node[xshift=1.8*\b cm, fill=white, inner sep=1] at (\a:0.65) {\small $b$};

\foreach \a/\b in {-90/0, -90/1, -90/2, -90/3, -90/4, -18/5}
\node[xshift=1.8*\b cm, fill=white, inner sep=1] at (\a:0.65) {\small $c$};

\foreach \a/\b in {-18/0, -18/1, 54/2}
\node[xshift=1.8*\b cm, fill=white, inner sep=1] at (\a:0.65) {\small $d$};

\node[fill=white, inner sep=1] at (54:0.65) {\small $e$};

\node[xshift=7.2cm] at (234:0.55) {\small $\theta$};

\node at (9,0.55) {\small $\theta$};
\node at (10.8,0.55) {\small $\theta$};

\end{tikzpicture}
\caption{Not suitable for tiling.}
\label{edges2}
\end{figure}

By Lemma \ref{base_tile}, we may assume the first four of Figure \ref{edges2} are special tiles. This implies one of the two ends of $c$ has degree $3$. Without loss of generality, we may assume the right end of $c$ has degree $3$. Let $x$ be the third edge at this vertex. In the first picture, we see $x,c$ are adjacent in a tile, and $x,d$ are adjacent in another tile. Since there is no edge in the first pentagon that is adjacent to both $c$ and $d$, we get a contradiction. Similar argument for the second, third and fourth pictures lead to the same contradiction. 

The fifth pentagon in Figure \ref{edges2} has no $a^2$-angle, no $b^2$-angle, and no $c^2$-angle. Therefore each degree $3$ vertex has at most one $a$, at most one $b$, and at most one $c$. This implies that each degree $3$ vertex has exactly one $a$, one $b$, and one $c$. Therefore each degree $3$ vertex has the $ac$-angle $\theta$. By Lemma \ref{deg3a}, this implies the $ac$-angle $\theta$ appears at least twice in the pentagon, a contradiction. 

The sixth pentagon in Figure \ref{edges2} has no $b^2$-angle, no $c^2$-angle, and no $bc$-angle. The seventh pentagon in Figure \ref{edges2} has no $b^2$-angle. By the similar argument as the fifth of Figure \ref{edges2}, this implies each degree $3$ vertex has at least two $a$, and therefore has the $a^2$-angle $\theta$. By Lemma \ref{deg3a}, this implies the $a^2$-angle $\theta$ appears at least twice in the pentagon, a contradiction. 
\end{proof}

\subsection{Anglewise Vertex Combination}
\label{avc}

In this paper, the pentagons have the edge combination $a^2b^2c$ or $a^3bc$. By Lemma \ref{edge_combo}, each combination has the unique edge arrangement. Then we denote the angles as in Figure \ref{pentagon}. 

We denote a vertex by $\alpha^a\beta^b\gamma^c\delta^d\epsilon^e$, to mean the vertex has $a$ copies of $\alpha$, $b$ copies of $\beta$, etc. The {\em angle sum of the vertex} is
\[
a\alpha+b\beta+c\gamma+d\delta+e\epsilon=2\pi.
\] 
We use $\alpha\beta^2\cdots$ to mean $a\ge 1$ and $b\ge 2$, and we use $R(\alpha\beta^2\cdots)$ to mean the {\em remainder}. The remainder is either the remaining angle combination
\[
R(\alpha\beta^2\cdots)
=\alpha^{a-1}\beta^{b-2}\gamma^c\delta^d\epsilon^e,
\]
or the remaining angle value
\[
R(\alpha\beta^2\cdots)
=2\pi-\alpha-2\beta
=(a-1)\alpha+(b-2)\beta+c\gamma+d\delta+e\epsilon.
\]

The {\em anglewise vertex combination}, abbreviated as AVC, is the collection of all vertices in a tiling. The following is an example in Section \ref{pent_ddivision}
\begin{equation}\label{avc13}
\text{AVC}=\{\alpha\beta^2,\beta^2\epsilon,\gamma^2\delta,\delta^3,\alpha^4,\epsilon^4\}.
\end{equation}
The AVC tells us $\beta\epsilon\cdots=\beta^2\epsilon$, $\alpha^2\cdots=\alpha^4$, $\gamma^2\cdots=\gamma^2\delta$, $\epsilon^2\cdots=\epsilon^4$, etc. We will use these properties in constructing the tiling.

We explain how the AVC is derived. More can be found in \cite{luk}. In the specific setting in Section \ref{pent_ddivision}, we know
\[
\alpha=\tfrac{1}{2}\pi,\;
\beta=(\tfrac{5}{6}-\tfrac{4}{f})\pi,\;
\gamma=\delta=\tfrac{2}{3}\pi,\;
\epsilon=(\tfrac{1}{3}+\tfrac{8}{f})\pi,\; 
f>24.
\]
If $\alpha^a\beta^b\gamma^c\delta^d\epsilon^e$ is a vertex, then the angle sum of the vertex is (we omit $\pi$ on both sides)
\[
\tfrac{1}{2}a+(\tfrac{5}{6}-\tfrac{4}{f})b+\tfrac{2}{3}(c+d)+(\tfrac{1}{3}+\tfrac{8}{f})e=2.
\]
This implies
\[
(3a+5b+4c+4d+2e-12)f=24(b-2e).
\]
By $f>24$, we have $\beta>\frac{2}{3}\pi$. We also have $\epsilon>\frac{1}{3}\pi$ and the specific values of $\alpha,\gamma,\delta$. This implies $a\le 4,b\le 2,c\le 3,d\le 3,e\le 5$. We substitute the finitely many combinations of the exponents satisfying these bounds into the equation above and solve for $f$. Those combinations yielding even $f>24$ are given in Table \ref{AVC32b}. In the table, ``$f=$ all'' means that the angle combinations can be vertices for any $f$. The vertices in the column ``not vertex'' are numerically possible, but are practically impossible by the edge length consideration. 

\begin{table}[htp]
\centering
\begin{tabular}{|c|c|c|}
\hline 
$f$ & vertex & not vertex \\
\hline \hline 
all
& $\beta^2\epsilon,\gamma^2\delta,\delta^3,\alpha^4$
& $\gamma\delta^2,\gamma^3$ \\
\hline 
$48$
& $\alpha\beta^2,\epsilon^4$  
& $\alpha^3\epsilon,\alpha^2\epsilon^2,\alpha\epsilon^3$ \\
\hline 
$72$
& $\delta\epsilon^3$  
& $\gamma\epsilon^3$ \\
\hline  
$96$ 
& & $\alpha\gamma\epsilon^2,\alpha\delta\epsilon^2$  \\
\hline 
$120$  
& $\epsilon^5$
& $\beta\epsilon^3$ \\
\hline
$192$   
& & $\alpha\epsilon^4$ \\
\hline 
\end{tabular}
\caption{AVC for Case 1.3, $H=\alpha^4$.}
\label{AVC32b}
\end{table}

For $f=48$, we combine those in the column ``vertex'' for all $f$ and for $f=48$. Then we get the AVC \eqref{avc13}. By the same idea, we get the AVCs for the other cases
\begin{align}
f=72 &\colon
\text{AVC}=\{\beta^2\epsilon,\gamma^2\delta,\delta^3,\alpha^4,\delta\epsilon^3\}, \nonumber \\
f=120 &\colon
\text{AVC}=\{\beta^2\epsilon,\gamma^2\delta,\delta^3,\alpha^4,\epsilon^5\}, \label{avc120} \\
f\ne 48, 72, 120 &\colon
\text{AVC}=\{\beta^2\epsilon,\gamma^2\delta,\delta^3,\alpha^4\}. \nonumber
\end{align}

The AVC is not used much in this paper. It will be used much more in the later papers, and we even need to introduce more sophistication in the notation for AVC.

\subsection{Adjacent Angle Deduction}
\label{aad}

In addition to angle combinations, we often need to know the {\em arrangement} of angles at a vertex. Figure \ref{anglecombo} shows four tiles around $\beta^4$ (for the edge combination $a^2b^2c$ in Figure \ref{pentagon}) and $\alpha\beta\gamma\delta$ (for the edge combination $a^3bc$ in Figure \ref{pentagon}). We indicate the arrangements of angles and edges by denoting the vertices as $\thin\beta\thin\beta\thin\beta\thin\beta\thin$ and $\dash\alpha\thick\beta\thin\delta\thin\gamma\dash$. The notation can be reversed, such as $\dash\alpha\thick\beta\thin\delta\thin\gamma\dash=\dash\gamma\thin\delta\thin\beta\thick\alpha\dash$. The notation can be rotated, such as $\dash\alpha\thick\beta\thin\delta\thin\gamma\dash=\thick\beta\thin\delta\thin\gamma\dash\alpha\thick=\thin\gamma\dash\alpha\thick\beta\thin\delta\thin$.

We also apply the notation to consecutive angles at a vertex. For example, we also denote the two vertices in Figure \ref{anglecombo} as $\beta\thin\beta\cdots$, $\thin\beta\thin\beta\thin\cdots,\alpha\thick\beta\cdots$, $\thick\beta\thin\delta\thin\gamma\dash\cdots$, and denote the consecutive angle segments as $\beta\thin\beta,\thin\beta\thin\beta\thin,\alpha\thick\beta,\thick\beta\thin\delta\thin\gamma\dash$. The notation can be reversed, such as $\alpha\thick\beta=\beta\thick\alpha$ and $\thick\beta\thin\delta\thin\gamma\dash=\dash\gamma\thin\delta\thin\beta\thick$, but not rotated unless it is the whole vertex, such as $\beta\thin\delta\thin\gamma\ne \delta\thin\gamma\dash\beta$.

\begin{figure}[htp]
\centering
\begin{tikzpicture}[>=latex,scale=1]


\foreach \a in {1,-1}
\foreach \b in {1,-1}
{
\begin{scope}[yscale=\a,xscale=\b]

\draw
	(0.8,0) -- (0,0) -- (0,0.8);

\draw[line width=1.5]
	(0,0.8) -- (0.6,1.2) -- (1.2,0.6);

\draw[dashed]
	(0.8,0) -- (1.2,0.6);

\node at (0.7,0.2) {\small $\delta$}; 
\node at (1,0.6) {\small $\epsilon$};
\node at (0.6,0.95) {\small $\gamma$};
\node at (0.2,0.7) {\small $\alpha$};
\node at (0.2,-0.2) {\small $\beta$};

\end{scope}
}


\foreach \a in {0,1,2,3}
{
\begin{scope}[xshift=3cm,rotate=90*\a]

\draw
	(0,0) -- (0.8,0);

\draw[line width=1.5]
	(1.2,0.6) -- (0.6,1.2) -- (0,0.8);

\draw[dashed]
	(0.8,0) -- (1.2,0.6);

\node at (0.7,0.2) {\small $\delta$};
\node at (1,0.6) {\small $\epsilon$};
\node at (0.6,0.95) {\small $\gamma$};
\node at (0.2,0.7) {\small $\alpha$};
\node at (0.2,0.2) {\small $\beta$};

\end{scope}
}


\foreach \a in {2,3}
{
\begin{scope}[xshift=3*\a cm]

\draw
	(-1.2,0.6) -- (-0.6,1.2) -- (0,0.8) -- (0,0) -- (0.8,0) -- (1.2,-0.6) -- (0.6,-1.2)
	(-0.8,0) -- (-1.2,-0.6) -- (-0.6,-1.2) -- (0,-0.8);

\draw[line width=1.5]
	(0,0) -- (-0.8,0)
	(0,-0.8) -- (0.6,-1.2);

\draw[dashed]
	(-0.8,0) -- (-1.2,0.6)
	(0,0) -- (0,-0.8);

\node at (-0.7,0.2) {\small $\alpha$}; 
\node at (-1,0.6) {\small $\gamma$};
\node at (-0.6,1) {\small $\epsilon$};
\node at (-0.2,0.7) {\small $\delta$};
\node at (-0.2,0.2) {\small $\beta$};

\node at (-0.7,-0.25) {\small $\beta$}; 
\node at (-1,-0.55) {\small $\delta$};
\node at (-0.6,-1) {\small $\epsilon$};
\node at (-0.2,-0.7) {\small $\gamma$};
\node at (-0.2,-0.2) {\small $\alpha$};

\node at (0.7,-0.2) {\small $\epsilon$}; 
\node at (1,-0.6) {\small $\delta$};
\node at (0.6,-0.95) {\small $\beta$};
\node at (0.2,-0.7) {\small $\alpha$};
\node at (0.2,-0.2) {\small $\gamma$};
	
\end{scope}
}


\begin{scope}[xshift=6cm]

\draw
	(0,0.8) -- (0.6,1.2);

\draw[line width=1.5]
	(0.8,0) -- (1.2,0.6);

\draw[dashed]
	(0.6,1.2) -- (1.2,0.6);

\node at (0.95,0.6) {\small $\alpha$};
\node at (0.7,0.2) {\small $\beta$};
\node at (0.6,0.95) {\small $\gamma$};
\node at (0.2,0.7) {\small $\epsilon$};
\node at (0.2,0.2) {\small $\delta$};

\end{scope}


\begin{scope}[xshift=9cm]

\draw
	(0.8,0) -- (1.2,0.6);

\draw[line width=1.5]
	(0,0.8) -- (0.6,1.2);

\draw[dashed]
	(0.6,1.2) -- (1.2,0.6);

\node at (1,0.55) {\small $\gamma$};
\node at (0.7,0.2) {\small $\epsilon$};
\node at (0.6,0.95) {\small $\alpha$};
\node at (0.2,0.7) {\small $\beta$};
\node at (0.2,0.2) {\small $\delta$};

\end{scope}

\end{tikzpicture}
\caption{Adjacent angle deduction.}
\label{anglecombo}
\end{figure}

The first and second of Figure \ref{anglecombo} have the same vertex $\thin\beta\thin\beta\thin\beta\thin\beta\thin$, but different arrangements of the four tiles. To indicate the difference, we introduce {\em adjacent angle deduction}, abbreviated as AAD. We write ${}^{\lambda}\theta^{\mu}$ to mean $\lambda,\mu$ are the two angles adjacent to $\theta$ in the pentagon. The first of Figure \ref{anglecombo} has the AAD $\thin^{\alpha}\beta^{\delta}\thin^{\delta}\beta^{\alpha}\thin^{\alpha}\beta^{\delta}\thin^{\delta}\beta^{\alpha}\thin$ (or equivalently, $\thin^{\delta}\beta^{\alpha}\thin^{\alpha}\beta^{\delta}\thin^{\delta}\beta^{\alpha}\thin^{\alpha}\beta^{\delta}\thin$). The second has the AAD $\thin^{\alpha}\beta^{\delta}\thin^{\alpha}\beta^{\delta}\thin^{\alpha}\beta^{\delta}\thin^{\alpha}\beta^{\delta}\thin$. The third and fourth have the respective AADs $\dash^{\gamma}\alpha^{\beta}\thick^{\alpha}\beta^{\delta}\thin^{\epsilon}\delta^{\beta}\thin^{\epsilon}\gamma^{\alpha}\dash$ and $\dash^{\gamma}\alpha^{\beta}\thick^{\alpha}\beta^{\delta}\thin^{\beta}\delta^{\epsilon}\thin^{\epsilon}\gamma^{\alpha}\dash$. 

The use of AAD naturally extends to consecutive angles. For example, we have $\beta^{\alpha}\thin^{\alpha}\beta,\thin^{\delta}\beta^{\alpha}\thin^{\alpha}\beta^{\delta}\thin,\thin^{\alpha}\beta^{\delta}\thin^{\delta}\beta^{\alpha}\thin$ in the first of Figure \ref{anglecombo}, and we have $\beta^{\alpha}\thin^{\delta}\beta$ in the second. We can be very flexible in using the AAD notation. For example, for the second pentagon in Figure \ref{pentagon}, $\thick^{\alpha}\beta^{\delta}\thin\delta\thin$ means $\thick^{\alpha}\beta^{\delta}\thin^{\epsilon}\delta^{\beta}\thin$ or $\thick^{\alpha}\beta^{\delta}\thin^{\beta}\delta^{\epsilon}\thin$. The AAD $\thick^{\alpha}\beta^{\delta}\thin\delta\thin$ appears in both the third and fourth of Figure \ref{anglecombo}.

The AAD can be reversed. For example, $\thick^{\alpha}\beta^{\delta}\thin^{\epsilon}\delta^{\beta}\thin^{\epsilon}\gamma^{\alpha}\dash$ is the same as $\dash^{\alpha}\gamma^{\epsilon}\thin^{\beta}\delta^{\epsilon}\thin^{\delta}\beta^{\alpha}\thick$. The AAD cannot be rotated unless it is the whole vertex. For example, the third and fourth of Figure \ref{anglecombo} have the AAD $\dash^{\gamma}\alpha^{\beta}\thick^{\alpha}\beta^{\delta}\thin\delta\thin^{\epsilon}\gamma^{\alpha}\dash$, which is the same as $\thin\delta\thin^{\epsilon}\gamma^{\alpha}\dash^{\gamma}\alpha^{\beta}\thick^{\alpha}\beta^{\delta}\thin$.

We use AAD as substitution for drawing some simple pictures. When we use AAD to make an argument, we actually have a picture in mind. For example, we can see  that the first of Figure \ref{anglecombo} has a vertex $\alpha\thin\alpha\cdots$ just off the central vertex $\beta^4$. We also see the same vertex $\alpha\thin\alpha\cdots$ from the corresponding AAD $\thin^{\delta}\beta^{\alpha}\thin^{\alpha}\beta^{\delta}\thin$. In general, ${\theta}^{\lambda}\thin^{\mu}\rho$ in an AAD implies a vertex $\lambda\thin\mu\cdots$ ($\thin$ can be replaced by $\thick$ or $\dash$ in this argument). In this process, we ``deduce'' a vertex  $\lambda\thin\mu\cdots$ from the AAD ${\theta}^{\lambda}\thin^{\mu}\rho$, which is actually a description of ``adjacent angles''. This is the reason we call adjacent angle deduction.

The AAD has the following {\em reciprocity} property: An AAD $\lambda^{\theta}\thin^{\rho}\mu$ at $\lambda\mu\cdots$ implies an AAD $\theta^{\lambda}\thin^{\mu}\rho$ at $\theta\rho\cdots$. 

Here is a typical AAD argument. For the first pentagon in Figure \ref{pentagon}, $\thin\beta\thin\beta\thin$ has three possible AADs (up to reversion) $\thin^{\alpha}\beta^{\delta}\thin^{\alpha}\beta^{\delta}\thin$, $\thin^{\alpha}\beta^{\delta}\thin^{\delta}\beta^{\alpha}\thin$, $\thin^{\delta}\beta^{\alpha}\thin^{\alpha}\beta^{\delta}\thin$. If we know $\alpha\thin\alpha\cdots$ and $\delta\thin\delta\cdots$ are not vertices, then the AAD of $\thin\beta\thin\beta\thin$ is  $\thin^{\alpha}\beta^{\delta}\thin^{\alpha}\beta^{\delta}\thin$. In the actual proof, we say the following: By no $\alpha\thin\alpha\cdots$ and no $\delta\thin\delta\cdots$, we know $\thin\beta\thin\beta\thin$ has the {\em unique} AAD $\thin^{\alpha}\beta^{\delta}\thin^{\alpha}\beta^{\delta}\thin$. We also write the conclusion as $\thin\beta\thin\beta\thin=\thin^{\alpha}\beta^{\delta}\thin^{\alpha}\beta^{\delta}\thin$.

As a matter of fact, at a vertex $\beta^n$, no $\delta\thin\delta\cdots$ already implies the unique AAD $\thin^{\alpha}\beta^{\delta}\thin^{\alpha}\beta^{\delta}\thin^{\alpha}\beta^{\delta}\thin\cdots$ of the vertex. We consider one $\beta$ at $\beta^n=\thin\beta\thin\beta\thin\beta\thin\cdots$, and the adjacent $\beta$ on the $\delta$-side of the first $\beta$. This means the AAD $\thin^{\alpha}\beta^{\delta}\thin\beta\thin=\thin^{\alpha}\beta^{\delta}\thin^{\alpha}\beta^{\delta}\thin$ or $\thin^{\alpha}\beta^{\delta}\thin^{\delta}\beta^{\alpha}\thin$. By no $\delta\thin\delta\cdots$, we cannot have $\thin^{\alpha}\beta^{\delta}\thin^{\delta}\beta^{\alpha}\thin$. Therefore $\thin^{\alpha}\beta^{\delta}\thin\beta\thin=\thin^{\alpha}\beta^{\delta}\thin^{\alpha}\beta^{\delta}\thin$. Then we apply the equality to the second and third $\beta$ in $\thin^{\alpha}\beta^{\delta}\thin^{\alpha}\beta^{\delta}\thin\beta\thin$ to get $\thin^{\alpha}\beta^{\delta}\thin^{\alpha}\beta^{\delta}\thin^{\alpha}\beta^{\delta}\thin$. We continue the argument and get the unique AAD of $\beta^n$. 

If we dig further, we find that the number of $\beta^{\alpha}\thin^{\alpha}\beta$ equals the number of $\beta^{\delta}\thin^{\delta}\beta$ at $\beta^n$. This implies that, if $n$ is odd, then we must have $\beta^{\alpha}\thin^{\delta}\beta$ at $\beta^n$. The argument can be summarised into the following result. We note that $\thin$ can be replaced by $\thick$ in the lemma.

\begin{lemma}\label{aadlemma}
Suppose $\lambda$ and $\mu$ are the two angles adjacent to $\theta$ in a pentagon.
\begin{itemize}
\item If $\lambda\thin\lambda\cdots$ is not a vertex, then $\theta^n$ has the unique AAD $\thin^{\lambda}\theta^{\mu}\thin^{\lambda}\theta^{\mu}\thin^{\lambda}\theta^{\mu}\thin\cdots$.
\item If $n$ is odd, then we have the AAD $\thin^{\lambda}\theta^{\mu}\thin^{\lambda}\theta^{\mu}\thin$ at $\theta^n$.
\end{itemize}
\end{lemma}

The first statement means that, if the AAD of $\theta^n$ gives $\lambda\thin\lambda\cdots$, then the AAD also gives $\mu\thin\mu\cdots$. The second statement means that, if $\theta^n$ is a vertex and $n$ is odd, then the AAD of $\theta^n$ gives a vertex $\lambda\thin\mu\cdots$.

\section{Subdivision Tiling}
\label{div}

\subsection{Pentagonal Subdivision}
\label{pdiv}

Given any tiling of an oriented closed surface, we add two vertices to each edge. If $E$ is an edge of a tile $T$, then we can use the orientation to label the two vertices on $E$ as the first and the second, with respect to the tile $T$ and the orientation. See the labels for the five edges of a pentagonal tile on the left of Figure \ref{subdivision}. Note that each edge is shared by two tiles, and the labels from the viewpoints of two tiles are different.

\begin{figure}[htp]
\centering
\begin{tikzpicture}[>=latex]

\foreach \a in {0,1}
\draw[xshift=6*\a cm]
	(-1.8,-1.8) rectangle (0,0)
	(0,0) -- (1.8,-0.9) -- (0,-1.8)
	(0,0) -- (-0.9,1.8) -- (-1.8,0)
	(1.8,-0.9) -- (1.8,0.6) -- (0.6,1.8) -- (-0.9,1.8)
	;

	
\draw[->]
	(-1.1,-0.9) arc (-180:90:0.2);
\draw[->]
	(-1.1,0.6) arc (-180:90:0.2);
\draw[->]
	(0.4,-0.9) arc (-180:90:0.2);
\draw[->]
	(0.4,0.6) arc (-180:90:0.2);
	
\fill
	(0.6,-0.3) circle (0.05)
	(1.2,-0.6) circle (0.05)
	(1.8,-0.4) circle (0.05)
	(1.8,0.1) circle (0.05)
	(1.4,1) circle (0.05)
	(1,1.4) circle (0.05)
	(0.1,1.8) circle (0.05)
	(-0.4,1.8) circle (0.05)
	(-0.6,1.2) circle (0.05)
	(-0.3,0.6) circle (0.05);

\node at (0.7,-0.15) {\scriptsize 1};
\node at (1.3,-0.45) {\scriptsize 2};
\node at (-0.45,1.3) {\scriptsize 1};
\node at (-0.15,0.7) {\scriptsize 2};
\node at (1.3,0.85) {\scriptsize 1};
\node at (0.9,1.25) {\scriptsize 2};
\node at (0.1,1.6) {\scriptsize 1};
\node at (-0.4,1.6) {\scriptsize 2};
\node at (1.65,-0.4) {\scriptsize 1};
\node at (1.65,0.1) {\scriptsize 2};

\node at (-0.45,0.5) {\scriptsize 1};
\node at (-0.75,1.1) {\scriptsize 2};

\node at (1.1,-0.75) {\scriptsize 1};
\node at (0.5,-0.45) {\scriptsize 2};
			
\draw[very thick, ->]
	(2.5,0) -- ++(1.2,0);


\draw[xshift=6cm]
	(-1.8,-0.6) -- (0,-1.2)
	(-0.6,0) -- (-1.2,-1.8)
	(0.6,-0.9) -- (0,-0.6)
	(0.6,-0.9) -- (0.6,-1.5)
	(0.6,-0.9) -- (1.2,-0.6)
	(-0.9,0.6) -- (-1.2,0)
	(-0.9,0.6) -- (-0.3,0.6)
	(-0.9,0.6) -- (-1.2,1.2)
	(0.6,0.6) -- (0.6,-0.3)
	(0.6,0.6) -- (-0.6,1.2)
	(0.6,0.6) -- (1.4,1)
	(0.6,0.6) -- (1.8,-0.4)
	(0.6,0.6) -- (0.1,1.8)
	;

\end{tikzpicture}
\caption{Pentagonal subdivision.}
\label{subdivision}
\end{figure}

We further add one vertex at the center of each tile, and connect the center vertex to the first vertex of each boundary edge of the tile. This divides an $m$-gon tile into $m$ pentagons. See the right of Figure \ref{subdivision}. The process turns any tiling into a pentagonal tiling. We call the process {\em pentagonal subdivision}.

Each tile in the pentagonal subdivision has one vertex from the original tiling, three newly added edge vertices, and one center vertex. The edge vertices have degree $3$. The degree of the original vertex remains the same. The degree of the center vertex is the ``degree'' (number of edges) of the original tile. It is easy to see that the pentagonal subdivision of a tiling and the pentagonal subdivision of the dual tiling are combinatorially the same, with original vertices and center vertices exchanged.

Suppose we start with a regular tiling of the sphere (i.e., Platonic solid), such that the distance from the new edge vertices to the nearby original vertices are the same. Then the pentagonal subdivision is a tiling by congruent pentagons. Figure \ref{cube_oct} gives the pentagonal subdivisions of the cube and octahedron.

\begin{figure}[htp]
\centering
\begin{tikzpicture}[>=latex,scale=1]


\draw[dashed]
	(0,0) -- (0,2.4) -- (-2,3.2) -- (-2,0.8) -- (0,0) -- (1.6,0.8) -- (1.6,3.2) -- (0,2.4)
	(-2,3.2) -- (-0.4,4) -- (1.6,3.2);

\draw
	(-2,2.6) -- (0,0.6)
	(1.6,1.4) -- (0,1.8)
	(-0.5,2.6) -- (-1.5,0.6)
	(-1.5,3) -- (1.2,3.4)
	(-0.8,3.8) -- (0.4,2.6)
	(1.2,3) -- (0.4,0.2)
	;

\draw[line width=1.5]
	(0,0) -- (0.4,0.2)
	(0,0) -- (-0.5,0.2)
	(0,0) -- (0,0.6)
	(0,2.4) -- (-0.5,2.6)
	(0,2.4) -- (0,1.8)
	(0,2.4) -- (0.4,2.6)
	(-2,0.8) -- (-1.5,0.6)
	(-2,0.8) -- (-2,1.4)
	(-2,3.2) -- (-1.5,3)
	(-2,3.2) -- (-1.6,3.4)
	(-2,3.2) -- (-2,2.6)
	(-0.4,4) -- (-0.8,3.8)
	(-0.4,4) -- (0.1,3.8)
	(1.6,3.2) -- (1.1,3.4)
	(1.6,3.2) -- (1.2,3)
	(1.6,3.2) -- (1.6,2.6)
	(1.6,0.8) -- (1.6,1.4)
	(1.6,0.8) -- (1.2,0.6);
	

\begin{scope}[shift={(5cm,2cm)}]

\draw[dashed]
	(0,2) -- (2,0) -- (0,-2) -- (-2,0) -- cycle
	(0,2) -- (0.4,-0.4) -- (0,-2)
	(2,0) -- (0.4,-0.4) -- (-2,0);
		
\draw[line width=1.5]
	(2,0) -- (1.5,0.5)
	(2,0) -- (1.6,-0.1)
	(2,0) -- (1.5,-0.5)
	(0,2) -- (-0.5,1.5)
	(0,2) -- (0.1,1.4)
	(0,2) -- (0.5,1.5)
	(-2,0) -- (-1.5,0.5)
	(-2,0) -- (-1.4,-0.1)
	(-2,0) -- (-1.5,-0.5)
	(0,-2) -- (-0.5,-1.5)
	(0,-2) -- (0.1,-1.6)
	(0,-2) -- (0.5,-1.5)
	(0.4,-0.4) -- (0.3,0.2)
	(0.4,-0.4) -- (0.8,-0.3)
	(0.4,-0.4) -- (-0.2,-0.3)
	(0.4,-0.4) -- (0.3,-0.8)
	;

\draw
	(0.8,0.5) -- (0.1,1.4)
	(0.8,0.5) -- (0.8,-0.3)
	(0.8,0.5) -- (1.5,0.5)
	(-0.5,0.5) -- (0.3,0.2)
	(-0.5,0.5) -- (-0.5,1.5)
	(-0.5,0.5) -- (-1.4,-0.1)
	(-0.5,-0.8) -- (-0.2,-0.3)
	(-0.5,-0.8) -- (-1.5,-0.5)
	(-0.5,-0.8) -- (0.1,-1.6)
	(0.8,-0.8) -- (0.3,-0.8)
	(0.8,-0.8) -- (1.6,-0.1)
	(0.8,-0.8) -- (0.5,-1.5)
	;

\end{scope}

\end{tikzpicture}
\caption{Pentagonal subdivisions of the cube and octahedron.}
\label{cube_oct}
\end{figure}

The first of Figure \ref{platonic_subdivision} shows the pentagonal subdivision of one (regular) triangular face of the (regular) tetrahedron, octahedron, or icosahedron, with $\beta=\frac{2}{3}\pi$ and $\gamma=\frac{2}{n}\pi$ ($n=3,4,5$ respectively, for tetrahedron, octahedron, icosahedron). We also have $\beta=\frac{2}{m}\pi$ and $\gamma=\frac{2}{3}\pi$ for the cube ($m=4$) and dodecahedron ($m=5$). The second of Figure \ref{platonic_subdivision} is the pentagon in the pentagonal subdivision tiling, and has the edge combination $a^2b^2c$.

We remark that we no longer require the original (dotted) edges to be straight. This means that we only require $\alpha+\delta+\epsilon=2\pi$ instead of $\alpha+\delta=\epsilon=\pi$. This allows the pentagonal subdivision of a Platonic solid and its dual to give the same tiling. For example, the two pentagonal subdivision tilings in Figure \ref{cube_oct} are obtained from each other by exchanging the normal $a$-edges and the thick $b$-edges. Since dual tilings give the same pentagonal subdivision, we get three tilings of the sphere by congruent pentagons:
\begin{enumerate}
\item Pentagonal subdivision of the tetrahedron: $f=12$. 
\item Pentagonal subdivision of the octahedron (or cube): $f=24$. 
\item Pentagonal subdivision of the icosahedron (or dodecahedron): $f=60$.
\end{enumerate}
The first tiling is the deformed dodecahedron $T_5$ in \cite{gsy}, and is the first of Figure \ref{subdivision_tiling}. Further by \cite{ay1}, this is the only edge-to-edge tiling of the sphere by $12$ congruent pentagons. The second tiling is the second of Figure \ref{subdivision_tiling}, and Figure \ref{cube_oct}. The third tiling is the third of Figure \ref{subdivision_tiling}.

\begin{figure}[htp]
\centering
\begin{tikzpicture}[>=latex]


\begin{scope}[shift={(-4cm,-0.3cm)}]

\foreach \a in {0,1,2}
{
\begin{scope}[rotate=120*\a]

\draw[dotted]
	(-30:2) -- (90:2);

\draw[line width=1.5]
	(90:2) -- ++(-40:1.2)
	(90:2) -- ++(-100:1.2);

\draw
	(0,0) -- (104:0.86);
	
\draw[dashed]
	(54:1.53) -- (-16:0.86);

\fill
	 (90:2) circle (0.1);

\end{scope}
}

\filldraw[fill=white]
	 (0,0) circle (0.1);

\node at (3:0.65) {\small $\delta$};

\node at (-30:0.85) {\small $\alpha$};
\node at (-70:0.3) {\small $\beta$};
\node at (-35:1.6) {\small $\gamma$};
\node at (-120:0.6) {\small $\delta$};
\node at (-63:1.35) {\small $\epsilon$};

\node at (-8:1) {\small $\epsilon$};
\node at (-30:0.85) {\small $\alpha$};

\node at (50:0.3) {\small $\beta$};
\node at (170:0.25) {\small $\beta$};


\node at (90:2.3) {\small $C'$};
\node at (210:2.4) {\small $C''$};
\node at (-30:2.3) {\small $C$};
 	
\end{scope}


\draw
	(90:1.2) -- node[fill=white,inner sep=1] {\small $a$} 
	(162:1.2) -- node[fill=white,inner sep=1] {\small $a$}
	(234:1.2);
\draw[line width=1.5]
	(90:1.2) -- node[fill=white,inner sep=1] {\small $b$}
	(18:1.2) -- node[fill=white,inner sep=1] {\small $b$}
	(-54:1.2);
\draw[dashed]
	(234:1.2) -- node[fill=white,inner sep=1] {\small $c$}
	(-54:1.2);

\node at (90:0.9) {\small $\alpha$};

\draw[shift={(162:1.2)}]
	(36:0.2) arc (36:-72:0.2);
\node at (162:0.8) {\small $\beta$};

\fill[shift={(18:1.2)}]
	(144:0.2) arc (144:252:0.2) -- (0,0);
\node at (18:0.8) {\small $\gamma$};

\node at (234:0.9) {\small $\delta$};
\node at (-54:0.9) {\small $\epsilon$};

\node at (90:1.4) {\small $A$};
\node at (162:1.4) {\small $B$};
\node at (18:1.4) {\small $C$};
\node at (234:1.4) {\small $D$};
\node at (-54:1.4) {\small $E$};


\begin{scope}[shift={(4cm,-0.3cm)}]

\foreach \a in {0,1,2}
{
\begin{scope}[rotate=120*\a]

\draw[dotted]
	(-30:2) -- (90:2);

\draw[line width=1.5]
	(90:2) -- ++(-40:1.2)
	(90:2) -- ++(-100:1.2);

\draw
	(0,0) -- (54:1.53);
	
\draw[dashed]
	(54:1.53) -- (-16:0.86);

\fill
	 (90:2) circle (0.1);
	 
\end{scope}
}

\filldraw[fill=white]
	 (0,0) circle (0.1);
	 	
\node at (183:1.3) {\small $\alpha'$};
\node at (235:0.35) {\small $\beta'$};
\node at (203:1.55) {\small $\gamma'$};
\node at (-77:0.9) {\small $\delta'$};
\node at (210:0.7) {\small $\epsilon'$};
	 	
\end{scope}

\end{tikzpicture}
\caption{Pentagonal subdivision of Platonic solid, and the companion.}
\label{platonic_subdivision}
\end{figure}

In the second of Figure \ref{platonic_subdivision}, we denote the vertices of the pentagon by $A,B,C,D,E$, to match the notations of the five angles. Then $B$ is the center of a triangular face in the first of Figure \ref{platonic_subdivision}. We denote by $C'$ and $C''$ the rotations of $C$ around $B$ by $\frac{2}{3}\pi$ and $\frac{4}{3}\pi$. Then $C,C',C''$ are the vertices of a triangular face. 

The pentagon is determined by the location of $A$. By rotating $A$ around $B$ by $-\beta=-\frac{2}{3}\pi$, we get $D$. By rotating $A$ around $C$ by $\gamma=\frac{2}{n}\pi$, we get $E$. Then $DE$ and $CC''$ intersect at the midpoints of both arcs. This gives a one-to-one correspondence between the pentagonal subdivision tiling and the location of $A$, subject to the only condition that the pentagon $\pentagon ABCDE$ is simple. In particular, the pentagonal subdivision tiling allows two free parameters. We study the $2$-dimensional moduli of pentagonal subdivision tilings in \cite{wy3}.

Finally, if we change the orientation of the surface, then the labels of the first and second points in Figure \ref{subdivision} are exchanged. In other words, we get another pentagonal subdivision tiling by connecting the centers of tiles to the second boundary vertices according to the original label. We call the new pentagonal subdivision tiling the {\em companion} of the original one. When the companion operation is applied to the pentagonal subdivision of Platonic solids, we change the first of Figure \ref{platonic_subdivision} to the third of Figure \ref{platonic_subdivision}. The angles of the two tilings are related by
\[
\beta'=\beta,\;
\gamma'=\gamma,\;
\alpha'+\delta'=\epsilon,\;
\alpha+\delta=\epsilon',\;
\epsilon+\epsilon'=2\pi.
\]
The orientation $\alpha'\to \beta'\to \delta'\to \epsilon'\to \gamma'$ is opposite to the orientation $\alpha\to \beta\to \delta\to \epsilon\to \gamma$. To compare the pentagonal subdivision tiling and its companion, therefore, we need to flip the whole companion tiling. 

The companion construction may turn a simple pentagon into a non-simple one. Therefore only some pentagonal subdivision tilings have the companion. The exact subset of the moduli where the companion exists is given in \cite{wy3}.

\subsection{Double Pentagonal Subdivision}
\label{dpdiv}

Given any tiling of a closed surface, we have the dual tiling. Combining the tiling and its dual, we get a quadrilateral tiling. If the surface is also oriented, then we may use the orientation to further cut each quadrilateral into two halves in a compatible way. The process turns any tiling into a pentagonal tiling. We call the process {\em double pentagonal subdivision}.

\begin{figure}[htp]
\centering
\begin{tikzpicture}[>=latex]

\foreach \a in {0,1}
\draw[xshift=6*\a cm]
	(-1.2,-1.2) rectangle (1.2,1.2)
	(1.2,1.2) -- (3.2,0) -- (1.2,-1.2)
	;

	
\draw[->]
	(-0.2,0) arc (-180:90:0.2);
\draw[->]
	(1.8,0) arc (-180:90:0.2);

\draw[very thick, ->]
	(3.5,0) -- ++(1,0);


\begin{scope}[xshift=6cm]

\foreach \b in {0,...,3}
\draw[rotate=90*\b]
	(0,0) -- (1.2,0)
	(0.6,0) -- (0.6,1.2) 
	;

\draw
	(-0.6,1.2) -- ++ (0,0.2)
	(0.6,-1.2) -- ++ (0,-0.2)
	(-1.2,-0.6) -- ++ (-0.2,0)
	(1.2,0) -- (1.8,0)
	(2.2, 0.6) -- (1.8,0) -- (2.2, -0.6) 
	(1.7,0.9) -- ++ (0.1,0.2)
	(2.7,-0.3) -- ++ (0.1,-0.2)
	(1.7,-0.9) -- (1.5,0)
	(1.2,0.6) -- (2,0.3)
	(2,-0.3) -- (2.7,0.3)
	;

\end{scope}

\end{tikzpicture}
\caption{Double pentagonal subdivision.}
\label{double_subdivision}
\end{figure}

Similar to pentagonal subdivision, we hope that the double pentagonal subdivision of a Platonic solid can be a tiling of the sphere by congruent pentagons. Since the dual tiling gives the same double pentagonal subdivision, we only need to consider the double pentagonal subdivision of one triangular face of the regular tetrahedron ($n=3$), octahedron ($n=4$), icosahedron ($n=5$). This is given by the first of Figure \ref{double_tiling}. By the angle sums of the vertices, we get the angles of the pentagonal tile in the second of Figure \ref{double_tiling}
\[
\alpha=\tfrac{1}{2}\pi,\;
\beta=\left(1-\tfrac{1}{n}\right)\pi,\;
\gamma=\delta=\tfrac{2}{3}\pi,\;
\epsilon=\tfrac{2}{n}\pi.
\]
The number of tiles are respectively $f=24,48,120$.

\begin{figure}[htp]
\centering
\begin{tikzpicture}[>=latex]

\foreach \a in {0,1,2}
{
\begin{scope}[rotate=120*\a, scale=1.2]

\draw[dashed]
	(1.3,-0.25) -- (0.433,1.25);
	
\draw
	(0,0) -- (30:0.5) -- (-0.433,1.25) -- (90:2) -- (0.433,1.25)
	(0.433,1.25) -- ++(30:0.4)
	(80:2.2) -- (90:2) -- (100:2.2);

\draw[line width=1.5]
	(30:0.5) -- (30:1.3);

\fill
	 (90:2) circle (0.1);

\filldraw[fill=white]
	 (0,0) circle (0.1);
	 	 	 
\node at (140:0.85) {\small $\alpha$};
\node at (130:0.5) {\small $\beta$};
\node at (112:1.05) {\small $\gamma$};
\node at (90:0.22) {\small $\delta$};
\node at (48:0.37) {\small $\epsilon$};

\node at (40:0.87) {\small $\alpha$};
\node at (45:0.65) {\small $\beta$};
\node at (75:1.2) {\small $\gamma$};
\node at (101:1.25) {\small $\delta$};
\node at (90:1.75) {\small $\epsilon$};

\node at (38:1.15) {\small $\alpha$};
\node at (22:1.15) {\small $\alpha$};
\node at (62:1.35) {\small $\gamma$};
\node at (112:1.45) {\small $\gamma$};
\node at (71:1.52) {\small $\delta$};
\node at (83:1.95) {\small $\epsilon$};
\node at (97:1.95) {\small $\epsilon$};

\end{scope}
}

\filldraw[fill=gray]
	 (0,-1.2) circle (0.1);

\begin{scope}[shift={(5cm,0.3cm)}]
	
\draw
	(162:1.2) -- node[fill=white,inner sep=1] {\small $a$}
	(234:1.2) -- node[fill=white,inner sep=1] {\small $a$}
	(-54:1.2) -- node[fill=white,inner sep=1] {\small $a$}
	(18:1.2);
\draw[line width=1.5]
	(90:1.2) -- node[fill=white,inner sep=1] {\small $b$} 
	(162:1.2);
\draw[dashed]
	(90:1.2) -- node[fill=white,inner sep=1] {\small $c$}
	(18:1.2);

\node at (90:0.9) {\small $\alpha$};
\node at (162:0.9) {\small $\beta$};
\node at (18:0.9) {\small $\gamma$};
\node at (234:0.9) {\small $\delta$};
\node at (-54:0.9) {\small $\epsilon$};

\end{scope}

\end{tikzpicture}
\caption{Double pentagonal subdivision of Platonic solid.}
\label{double_tiling}
\end{figure}

Unlike pentagonal subdivision, all angles of the pentagon in a double pentagonal subdivision tiling are fixed. Note that a general pentagon has $7$ degrees of freedom. On the other hand, the specific values of $5$ angles and $3$ equal edges amount to $7$ equations. Therefore we expect the double pentagonal subdivision tilings to be isolated examples. 

In what follows, we calculate the exact values for the pentagon, and justifies the existence of the pentagon.

\subsubsection*{Calculate the Pentagon}

For a triangular face, we form the right triangle in the first of Figure \ref{pexist1} by connecting the center point $\circ$ of the face, the middle point ${\color{gray} \bullet}$ of an edge, and a vertex $\bullet$ of the face. The right triangle is one sixth of the triangular face, and has the respective angles $\frac{1}{3}\pi$, $\frac{1}{2}\pi$, $\frac{1}{n}\pi$ at $\circ$, ${\color{gray} \bullet}$, $\bullet$. The edges $x,y,z$ of the right triangle are $<\frac{1}{2}\pi$ and the exact values are given below.
\renewcommand{\arraystretch}{1.3}
\begin{center}
 \begin{tabular}{|c| c|| c c c|} 
 \hline
 $f$ & $n$ & $\cos x$ & $\cos y$ & $\cos z$ \\   
 \hline\hline
 12 & 3 
 & $\frac{1}{3}$ 
 & $\frac{1}{\sqrt{3}}$
 & $\frac{1}{\sqrt{3}}$ \\ 
 \hline
 24 & 4 
 & $\frac{1}{\sqrt{3}}$
 & $\frac{\sqrt{2}}{\sqrt{3}}$ 
 & $\frac{1}{\sqrt{2}}$ \\
 \hline
 60 & 5  
 & $\frac{\sqrt{5}+1}{\sqrt{6(5-\sqrt{5})}}$
 & $\frac{\sqrt{5}+1}{2\sqrt{3}}$
 & $\frac{\sqrt{2}}{\sqrt{5-\sqrt{5}}}$ \\ 
 \hline
\end{tabular}
\end{center}

\begin{figure}[htp]
\centering
\begin{tikzpicture}[>=latex,scale=1]


\draw[gray!50]
	(0,0) -- (0,2.6) -- (4,0) -- (0,0);
	
\draw
	(0,2.6) -- (1,1.2) -- (2.8,1.5) -- (4,0) -- (1.7,-0.7);

\draw[line width=1.5]
	(0,0) -- (1,1.2);

\draw[dashed]
	(0,0) -- (1.7,-0.7);

\fill
	 (4,0) circle (0.1);
	 
\filldraw[fill=white]
	 (0,2.6) circle (0.1);
	 
\filldraw[fill=gray]
	 (0,0) circle (0.1);
	 
\node[fill=white,inner sep=1] at (2.4,1) {\small $x$};
\node[fill=white,inner sep=1] at (0,1.3) {\small $y$};
\node[fill=white,inner sep=1] at (2,0) {\small $z$};

\node at (-0.3,-0.1) {\small $\alpha^4$};
\node at (-0.3,2.7) {\small $\delta^3$};
\node at (4.3,0) {\small $\epsilon^n$};

\node[gray] at (0.35,0.35) {\small $\frac{1}{2}\pi$};
\node[gray] at (0.35,2.05) {\small $\frac{1}{3}\pi$};
\node[gray] at (3,0.35) {\small $\frac{1}{n}\pi$};

\node[fill=white,inner sep=1] at (0.7,1.6) {\small $a$};
\node[fill=white,inner sep=1] at (1.6,1.3) {\small $a$};
\node[fill=white,inner sep=1] at (3.35,0.8) {\small $a$};
\node[fill=white,inner sep=1] at (3,-0.3) {\small $a$};
\node[fill=white,inner sep=1] at (0.6,0.75) {\small $b$};
\node[fill=white,inner sep=1] at (0.8,-0.3) {\small $c$};

\node at (0.75,1.2) {\small $\beta$};
\node at (1.1,1) {\small $\beta$};
\node at (1.1,1.4) {\small $\epsilon$};
\node at (2.7,1.3) {\small $\delta$};
\node at (1.7,-0.5) {\small $\gamma$};


\begin{scope}[shift={(7cm,1cm)}]

\draw[line width=1.5]
	(90:1.5) -- (162:1.5);

\draw[dashed]	
	(90:1.5) -- (18:1.5);
	
\draw[gray]
	(-54:1.5) -- (90:1.5) -- (234:1.5);

\draw
	(162:1.5) -- (234:1.5) -- (-54:1.5) -- (18:1.5);

\node[fill=white, inner sep=1.5] at (44:1.25) {\small $c$};
\node[fill=white, inner sep=1.5] at (136:1.25) {\small $b$};
\node[fill=white, inner sep=1.5] at (-8:1.25) {\small $a$};
\node[fill=white, inner sep=1.5] at (188:1.25) {\small $a$};
\node[fill=white, inner sep=1.5] at (-90:1.2) {\small $a$};

\node[fill=white, inner sep=1.5] at (162:0.45) {\small $y$};
\node[fill=white, inner sep=1.5] at (18:0.45) {\small $z$};

\node at (90:1.7) {\small $A$};
\node at (162:1.7) {\small $B$};
\node at (18:1.7) {\small $C$};
\node at (234:1.7) {\small $D$};
\node at (-54:1.7) {\small $E$};
	
\node at (162:1.2) {\small $\beta$};
\node at (18:1.2) {\small $\gamma$};

\node at (-0.85,-0.6) {\small $\delta_1$};
\node at (0.9,-0.6) {\small $\epsilon_1$};
\node at (-0.6,-1) {\small $\delta_2$};
\node at (0.6,-1) {\small $\epsilon_2$};
\node at (-0.4,0.9) {\small $\alpha_1$};
\node at (0,0.7) {\small $\alpha_2$};
\node at (0.42,0.9) {\small $\alpha_3$};

\end{scope}

\end{tikzpicture}
\caption{The pentagon in double pentagonal subdivision.}
\label{pexist1}
\end{figure}

To calculate the pentagon, we need the following extended cosine law. The proof can be found in \cite[Lemma 3]{ay1}.

\begin{lemma}\label{fourth}
For the spherical quadrilateral in Figure \ref{quad2}, we have
\begin{align*}
\cos x
&=\cos a\cos b\cos c
+\sin a\sin b\cos c\cos\theta
+\cos a\sin b\sin c\cos\rho \nonumber \\
&\quad 
+\sin a\sin c\sin\theta\sin\rho
-\sin a\cos b\sin c\cos \theta\cos\rho.
\end{align*}
\end{lemma}

\begin{figure}[htp]
\centering
\begin{tikzpicture}[scale=0.8]

\draw
	(-1,1.8) -- node[fill=white,inner sep=1] {\small $a$}
	(0,0) -- node[fill=white,inner sep=1] {\small $b$}
	(2,0) -- node[fill=white,inner sep=1] {\small $c$}
	(2.8,1.5);
\draw[dashed]
	(-1,1.8) -- node[fill=white,inner sep=1] {\small $x$} 
	(2.8,1.5);
	
\node at (0.15,0.2) {\small $\theta$};
\node at (1.9,0.2) {\small $\rho$};

\end{tikzpicture}
\caption{The length of fourth edge in a quadrilateral.}
\label{quad2}
\end{figure}

The first of Figure \ref{double_tiling} gives the edges $a,b,c$ and the angles $\beta,\gamma,\delta,\epsilon$ in the first of Figure \ref{pexist1}. Then by Lemma \ref{fourth} and the cosine law, we have
\begin{align*}
\cos x
&=\cos^3a(1-\cos\delta)(1-\cos\epsilon)
+\cos^2a\sin\delta\sin\epsilon \\
&\quad +\cos a(\cos\delta+\cos\epsilon-\cos\delta\cos\epsilon)
-\sin\delta\sin\epsilon,
\end{align*}
By the know values of $\delta,\epsilon,x$, this gives cubic equations for $\cos a$
\begin{align*}
f=24 &\colon
\cos x=\tfrac{1}{3}
=\tfrac{9}{4}\cos^3a
+\tfrac{3}{4}\cos^2a
-\tfrac{5}{4}\cos a
-\tfrac{3}{4}; \\
f=48 &\colon
\cos x=\tfrac{1}{\sqrt{3}}
=\tfrac{3}{2}\cos^3a
+\tfrac{\sqrt{3}}{2}\cos^2a
-\tfrac{1}{2}\cos a
-\tfrac{\sqrt{3}}{2}; \\
f=120 &\colon
\cos x
=\tfrac{\sqrt{5}+1}{\sqrt{6(5-\sqrt{5})}} 
=\tfrac{3}{8}(5-\sqrt{5})\cos^3a
+\tfrac{1}{8}\sqrt{6}{\textstyle \sqrt{5+\sqrt{5}}}\cos^2a  \\
&\quad\quad +\tfrac{1}{8}(3\sqrt{5}-7)\cos a
-\tfrac{1}{8}\sqrt{6}{\textstyle \sqrt{5+\sqrt{5}}}.
\end{align*}
Each equation has only one real solution
\begin{align*}
f=24 &\colon
\cos a
=\tfrac{2}{9}R^{\frac{1}{3}}
+\tfrac{8}{9}R^{-\frac{1}{3}}
-\tfrac{1}{9}; \\
f=48 &\colon
\cos a
=\tfrac{1}{9}R^{\frac{1}{3}}
+\tfrac{4}{3}R^{-\frac{1}{3}}
-\tfrac{1}{3\sqrt{3}}; \\
f=120 &\colon
\cos a
=\tfrac{1}{45}R^{\frac{1}{3}}
+\tfrac{4}{3}(5-\sqrt{5})R^{-\frac{1}{3}}
-\tfrac{\sqrt{6}}{180}(3\sqrt{5}+5){\textstyle \sqrt{5-\sqrt{5}}},
\end{align*}
where
\[
R=\begin{cases}
19+3\sqrt{33}, & f=24; \\
186\sqrt{3}+54\sqrt{35}, & f=48; \\
75{\textstyle \sqrt{5-\sqrt{5}}}(\sqrt{6}(103-55\sqrt{5})
+18{\textstyle \sqrt{1055+473\sqrt{5}}}), & f=120.
\end{cases}
\]

Then we may use 
\begin{align}
\cos y
&=\cos a\cos b+\sin a\sin b\cos \beta, \label{dpeq2} \\
\cos z
&=\cos a\cos c+\sin a\sin c\cos \gamma, \label{dpeq3} 
\end{align} 
and the known values of $\beta,\gamma,a,y,z$ to calculate  $b$ and $c$
\begin{align*}
f=24 \colon
&\cos b
=\cos c
=\tfrac{1}{12}(3\sqrt{3}-\sqrt{11})R^{\frac{1}{3}}
+\tfrac{1}{3}(3\sqrt{3}+\sqrt{11})R^{-\frac{1}{3}}
-\tfrac{1}{\sqrt{3}}; \\
f=48 \colon
&\cos b
=\tfrac{\sqrt{6}}{72}(13-\sqrt{3}\sqrt{35})R^{\frac{1}{3}}
+\tfrac{\sqrt{6}}{6}(13+\sqrt{3}\sqrt{35})R^{-\frac{1}{3}}
-\tfrac{2\sqrt{2}}{3}, \\
&\cos c
=\tfrac{\sqrt{2}}{36}(11-\sqrt{3}\sqrt{35})R^{\frac{1}{3}}
+\tfrac{\sqrt{2}}{3}(11+\sqrt{3}\sqrt{35})R^{-\frac{1}{3}}
-\tfrac{2\sqrt{2}}{3\sqrt{3}}; \\
f=120 \colon
&\cos b
=\tfrac{\sqrt{3}(1+\sqrt{5})}{720}(17+4\sqrt{5}-\sqrt{3}{\textstyle \sqrt{91+40\sqrt{5}}})R^{\frac{1}{3}}  \\
&\qquad\quad +\tfrac{\sqrt{5}}{\sqrt{3}}(17+4\sqrt{5}+\sqrt{3}{\textstyle \sqrt{91+40\sqrt{5}}})R^{-\frac{1}{3}} \\
&\qquad\quad -\tfrac{1}{30}(10+3\sqrt{5}){\textstyle \sqrt{10-2\sqrt{5}}}, \\
&\cos c
=\tfrac{\sqrt{2}(1+\sqrt{5})^{\frac{3}{2}}5^{\frac{1}{4}}}{1440}(19+2\sqrt{5}-\sqrt{3}{\textstyle \sqrt{91+40\sqrt{5}}})R^{\frac{1}{3}} \\
&\qquad\quad 
+\tfrac{\sqrt{2}(1+\sqrt{5})^{\frac{1}{2}}5^{\frac{3}{4}}}{6}(19+2\sqrt{5}+\sqrt{3}{\textstyle \sqrt{91+40\sqrt{5}}})R^{-\frac{1}{3}} 
-\tfrac{\sqrt{5}}{3\sqrt{3}}.
\end{align*}
The following are the approximate values of the edge lengths ($a= 0.14869\pi$ means $0.14869\pi\le a<0.14870\pi$)
\begin{align*}
f=24 &\colon
a= 0.14869\pi,\; 
b=c= 0.20564\pi; \\
f=48 &\colon
a= 0.12780\pi, \; 
b= 0.08402\pi, \;
c= 0.16277\pi; \\
f=120 &\colon
a= 0.09605\pi, \;
b= 0.02381\pi, \;
c= 0.10819\pi.
\end{align*}
We note that for $f=24$, the equality $\beta=\gamma=\delta=\epsilon$ implies the pentagon is symmetric, and we have $b=c$.

\subsubsection*{Verify the Pentagon}

We derived the pentagon based on the assumption of the existence of double pentagonal subdivision. However, this does not necessarily imply the converse, in particular the existence of the pentagon. To show the pentagon actually exist, we can do the following:
\begin{enumerate}
\item Use the data obtained above to reconstruct the pentagon.
\item Verify the pentagon is simple.
\item Verify the expected values of the five angles of the pentagon.
\end{enumerate}

In the second of Figure \ref{pexist1}, we let $A,B,C,D,E$ be the vertices of the pentagon at $\alpha,\beta,\gamma,\delta,\epsilon$. Then the triangle $\triangle ABD$ is congruent to the triangle connecting ${\color{gray} \bullet},\beta^2\epsilon,\circ$ in the first of Figure \ref{pexist1}. Therefore $AD=y$. We also see directly from the first of Figure \ref{pexist1} that $AE=z$. This suggests that we construct the pentagon in the following way: Use $a,b$ and $\beta=(1-\frac{1}{n})\pi$ to construct $\triangle ABD$. The equality \eqref{dpeq2} implies $AD=y$. Use $a,c$ and $\gamma=\frac{3}{2}\pi$ to construct $\triangle ACE$. The equality \eqref{dpeq3} implies $AE=z$. Use $a,y,z$ to construct $\triangle ADE$. For the triangle to exist, we need to verify
\[
a+y+z<2\pi,\;
a<y+z,\;
y<a+z,\;
z<a+y.
\]
The known values of $\cos y$ and $\cos z$ give the approximate values of $y$ and $z$
\begin{align*}
f=24 &\colon
y= 0.3040\pi,\; 
z= 0.3040\pi; \\
f=48 &\colon
y= 0.1959\pi, \; 
z= 0.2500\pi; \\
f=120 &\colon
y= 0.1161\pi, \;
z= 0.1762\pi.
\end{align*}
Then we can verify the inequality conditions for $\triangle ADE$ to exist.

We have constructed a pentagon with the expected edge lengths and expected $\beta,\gamma$ values. Next, we verify the pentagon is simple. We first use the known values of $a,b,c,y,z,\beta,\gamma$ to calculate the approximate values of the angles of the three triangles
\begin{align*}
f=24\colon &
\alpha_1=0.1585\pi,\;
\alpha_2=0.1829\pi,\;
\alpha_3=0.1585\pi, 
\\
&
\delta_1=0.2204\pi,\;
\delta_2=0.4461\pi,\;
\epsilon_1=0.2204\pi,\;
\epsilon_2=0.4461\pi; \\
f=48\colon &
\alpha_1= 0.1588\pi,\;
\alpha_2= 0.1822\pi,\;
\alpha_3= 0.1588\pi,
\\
&
\delta_1= 0.1035\pi,\;
\delta_2= 0.5631\pi,\;
\epsilon_1= 0.2045\pi,\;
\epsilon_2= 0.2954\pi; \\
f=120\colon &
\alpha_1= 0.1628\pi,\;
\alpha_2= 0.1743\pi,\;
\alpha_3= 0.1628\pi,
\\
&
\delta_1= 0.0392\pi,\;
\delta_2= 0.6273\pi,\;
\epsilon_1= 0.1850\pi,\;
\epsilon_2= 0.2149\pi.
\end{align*}
This implies
\begin{align*}
\alpha &
=\alpha_1+\alpha_2+\alpha_3 <0.501\pi<\pi, \\
\delta &
=\delta_1+\delta_2<0.667\pi<\pi, \\
\epsilon &
=\epsilon_1+\epsilon_2
<\begin{cases}
0.667\pi \text{ for $n=3$} \\
0.501\pi \text{ for $n=4$} \\
0.401\pi \text{ for $n=5$} 
\end{cases}
<\pi.
\end{align*}
Combined with $\beta,\epsilon<\pi$, we see the pentagon is convex, and is therefore simple.

Next we rigorously (i.e., symbolically) verify the expected values of $\alpha,\delta,\epsilon$. By Lemma \ref{fourth}, we have
\begin{align*}
\cos z
&=\cos^2a\cos b
+\sin^2a\cos b\cos\delta \\
&\quad +\cos a\sin a\sin b\cos\beta (1-\cos\delta)
+\sin a\sin b\sin\beta\sin\delta.
\end{align*}
We regard the equality as a linear equation in $\cos\delta,\sin\delta$, with the coefficients provided by the known values of $a,b,z,\beta$. We substitute the expected value $\delta=\frac{2}{3}\pi$ and find the equality holds. Since we know $\delta=\delta_1+\delta_2=0.666\pi$, this rigorously establishes $\delta=\frac{2}{3}\pi$.

Similarly, we use the equality
\begin{align*}
\cos y
&=\cos^2a\cos c
+\sin^2a\cos c\cos\epsilon \\
&\quad +\cos a\sin a\sin c\cos\gamma (1-\cos\epsilon)
+\sin a\sin c\sin\gamma\sin\epsilon
\end{align*}
to establish $\epsilon=\frac{2}{n}\pi$.

After we establish $\delta=\frac{2}{3}\pi$ and $\epsilon=\frac{2}{n}\pi$, we use the equality
\begin{align*}
\cos b\cos c+\sin b\sin c\cos \alpha
&=\cos^3a+\sin^2 a\sin\delta\sin\epsilon \\
&\quad+\sin^2a\cos a(\cos\delta+\cos\epsilon-\cos\delta\cos\epsilon)
\end{align*}
to establish $\alpha=\frac{1}{2}\pi$.

\section{Tiling for Edge Combination $a^2b^2c$}
\label{2a2bc}

This section is devoted to the classification of edge-to-edge tilings of the sphere by congruent pentagons with the edge combination $a^2b^2c$. By Lemma \ref{edge_combo}, the pentagon is the first of Figure \ref{pentagon}. We divide the classification into three cases in Figure \ref{2a2bc_pentagon}, according to the location of the vertex $H$ of a special tile (indicated by $\bullet$). In the first case, the degree of $H$ is $3,4$ or $5$. In the second and third cases, the degree of $H$ is $4$ or $5$. In other words, we do not need to consider $3^5$-tile in the second and third cases.

\begin{figure}[htp]
\centering
\begin{tikzpicture}[>=latex,scale=1]


\foreach \a in {1,2,3}
{
\begin{scope}[xshift=-2cm + 2*\a cm]

\draw
	(90:0.8) -- (162:0.8) -- (234:0.8);

\draw[line width=1.5]
	(90:0.8) -- (18:0.8) -- (-54:0.8);

\draw[dashed]
	(-54:0.8) -- (-126:0.8);

\fill (-54+72*\a:0.8) circle (0.1);

\node at (90:0.5) {\small $\alpha$};
\node at (162:0.5) {\small $\beta$};
\node at (18:0.5) {\small $\gamma$};
\node at (234:0.5) {\small $\delta$};
\node at (-54:0.5) {\small $\epsilon$};

\node at (0,0) {\a};

\end{scope}
}

\end{tikzpicture}
\caption{Special tile for the edge combination $a^2b^2c$.}
\label{2a2bc_pentagon}
\end{figure}

Our strategy is to first tile the partial neighbourhood in the fourth of Figure \ref{nhd}. Then we study the configuration around $H$, and further tile beyond the partial neighbourhood if necessary. We will always start by assuming that the center tile $T_1$ is given by the pentagon in Figure \ref{2a2bc_pentagon}.

To facilitate discussion, we denote by $T_i$ the tile labeled $i$, by $\theta_i$ the angle $\theta$ in $T_i$, by $E_{ij}$ the edge shared by $T_i,T_j$, and by $A_{i,jk}$ the angle of $T_i$ at the vertex shared by $T_i,T_j,T_k$. When we say a tile is {\em determined}, we mean that we know all the edges and angles of the tile.

\subsection{Case 1$(a^2b^2c)$}
\label{2a2bcCase1}

Let the first of Figure \ref{2a2bc_pentagon} be the center tile $T_1$ in the partial neighbourhood in Figure \ref{2a2bcF1}. The two pictures show two possible ways of arranging the edges and angles of $T_4$. In the first picture, we may determine $T_3,T_5$, and then get the $b^2$-angle $\gamma_2$ and the $a^2$-angle $\beta_6$. In the second picture, we may get the $b^2$-angle $\gamma_3$ and the $a^2$-angle $\beta_5$. Then $E_{23}=a$ or $c$, and either way gives $\gamma_2$. Similarly, we get $\beta_6$.

\begin{figure}[htp]
\centering
\begin{tikzpicture}[>=latex]


\foreach \a in {0,1}
\foreach \b in {0,1,2,3,4}
{
\begin{scope}[xshift=3.5*\a cm]

\coordinate  (A\a X\b) at (90+72*\b:0.7);
\coordinate  (B\a X\b) at (90+72*\b:1.3);
\coordinate  (C\a X\b) at (126+72*\b:1.7);

\coordinate  (P\a X\b) at (90+72*\b:0.45);

\coordinate  (Q\a X\b) at (102+72*\b:0.8);
\coordinate  (R\a X\b) at (100+72*\b:1.16);
\coordinate  (S\a X\b) at (126+72*\b:1.45);
\coordinate  (T\a X\b) at (152+72*\b:1.16);
\coordinate  (U\a X\b) at (150+72*\b:0.8);

\end{scope}
}

\foreach \a in {0,1}
{
\begin{scope}[xshift=3.5*\a cm]

\draw
	(A\a X0) -- (A\a X1) -- (A\a X2);

\draw[line width=1.5]
	(A\a X0) -- (A\a X4) -- (A\a X3);

\draw[dashed]
	(A\a X2) -- (A\a X3);

\fill (0,0.7) circle (0.1);

\node at (P\a X0) {\small $\alpha$};
\node at (P\a X1) {\small $\beta$};
\node at (P\a X4) {\small $\gamma$};
\node at (P\a X2) {\small $\delta$};
\node at (P\a X3) {\small $\epsilon$};

\node[draw,shape=circle, inner sep=0.5] at (0,0) {\small $1$};
\node[draw,shape=circle, inner sep=0.5] at (45:1.07) {\small $2$};\node[draw,shape=circle, inner sep=0.5] at (-18:1) {\small $3$};
\node[draw,shape=circle, inner sep=0.5] at (-90:1) {\small $4$};
\node[draw,shape=circle, inner sep=0.5] at (198:1) {\small $5$};
\node[draw,shape=circle, inner sep=0.5] at (135:1.07) {\small $6$};

\end{scope}
}


\draw
	(A0X1) -- (B0X1)
	(A0X0) -- (A0X1) -- (A0X2)
	(A0X3) -- (B0X3) -- (C0X2)
	(B0X3) -- (C0X3);

\draw[line width=1.5]
	(A0X4) -- (B0X4)
	(A0X0) -- (A0X4) -- (A0X3)
	(A0X2) -- (B0X2) -- (C0X2)
	(B0X2) -- (C0X1);

\draw[dashed]
	(A0X2) -- (A0X3)
	(B0X1) -- (C0X1)
	(B0X4) -- (C0X3);

\draw[dotted]
	(90:0.7) -- (60:1.5) -- (40:1.7) -- (18:1.3)
	(90:0.7) -- (120:1.5) -- (140:1.7) -- (162:1.3);
	
\node at (P0X0) {\small $\alpha$};
\node at (P0X1) {\small $\beta$};
\node at (P0X4) {\small $\gamma$};
\node at (P0X2) {\small $\delta$};
\node at (P0X3) {\small $\epsilon$};

\node at (Q0X2) {\small $\epsilon$};
\node at (R0X2) {\small $\gamma$};
\node at (S0X2) {\small $\alpha$};
\node at (T0X2) {\small $\beta$};
\node at (U0X2) {\small $\delta$};

\node at (Q0X1) {\small $\beta$};
\node at (R0X1) {\small $\delta$};
\node at (S0X1) {\small $\epsilon$};
\node at (T0X1) {\small $\gamma$};
\node at (U0X1) {\small $\alpha$};

\node at (Q0X3) {\small $\alpha$};
\node at (R0X3) {\small $\beta$};
\node at (S0X3) {\small $\delta$};
\node at (T0X3) {\small $\epsilon$};
\node at (U0X3) {\small $\gamma$};
	
\node at (32:0.8) {\small $\gamma$};
\node at (148:0.8) {\small $\beta$};


\draw
	(A1X2) -- (B1X2) -- (C1X2);
	
\draw[line width=1.5]
	(A1X3) -- (B1X3) -- (C1X2);

\draw[dotted]
	(A1X1) -- (B1X1) -- (C1X1) -- (B1X2)
	(A1X4) -- (B1X4) -- (C1X3) -- (B1X3);

\node at (U1X1) {\small $\beta$};
\node at (Q1X3) {\small $\gamma$};

\node at (Q1X2) {\small $\delta$};
\node at (R1X2) {\small $\beta$};
\node at (S1X2) {\small $\alpha$};
\node at (T1X2) {\small $\gamma$};
\node at (U1X2) {\small $\epsilon$};

\begin{scope}[xshift=3.5cm]

\draw
	(90:0.7) -- (120:1.5);

\draw[line width=1.5]
	(90:0.7) -- (60:1.5);

\draw[dotted]
	(60:1.5) -- (40:1.7) -- (18:1.3)
	(120:1.5) -- (140:1.7) -- (162:1.3);
	
\node at (65:0.75) {\small $\gamma$};
\node at (115:0.75) {\small $\beta$};

\end{scope}

\end{tikzpicture}
\caption{Cases 1 for $a^2b^2c$.}
\label{2a2bcF1}
\end{figure}

In the first of Figure \ref{2a2bcF1}, the angle sums of $\alpha\delta\epsilon,\beta^3,\gamma^3$ and the angle sum for pentagon imply $f=12$. By $f\ge 16$ (see the remark before Lemma \ref{base_tile}), we may dismiss the case. In fact, the case is the deformed dodecahedron in \cite{ay1,gsy}, which is the pentagonal subdivision of the regular tetrahedron. 

In the second of Figure \ref{2a2bcF1}, by the edge length consideration, the vertex $H=\alpha^2\beta\gamma, \alpha^2\beta^2\gamma,\alpha^2\beta\gamma^2, \alpha\beta\gamma\delta\epsilon$. By Lemma \ref{anglesum}, $\alpha\beta\gamma\delta\epsilon$ is not a vertex. Moreover, the first three possibilities imply $2\alpha+\beta+\gamma\le 2\pi$. Combined with the angle sums of $\beta\delta^2,\gamma\epsilon^2$, we get 
\[
\alpha+\beta+\gamma+\delta+\epsilon
=\tfrac{1}{2}((2\alpha+\beta+\gamma)+(\beta+2\delta)+(\gamma+2\epsilon))\le 3\pi,
\]
again contradicting Lemma \ref{anglesum}.

\subsection{Case 2$(a^2b^2c)$}
\label{2a2bcCase2}

Let the second of Figure \ref{2a2bc_pentagon} be the center tile $T_1$ in the partial neighbourhood in Figure \ref{2a2bcF2}. There are two possible arrangements of $T_5$. The first picture shows one arrangement, and the second and third show the other arrangement. We label the three partial neighbourhood tilings as Cases 2.1, 2.2, 2.3. We also recall that, after Case 1, the degree of $H$ is $4$ or $5$.

In the first picture, we get the $b^2$-angle $\gamma_4$ and the $a^2$-angle $\beta_6$. Then $E_{34}=a$ or $c$, and either way gives $\gamma_3$ and $E_{23}=b$. This determines $T_2$. 

In the second and third pictures, we may determine $T_4,T_6$. Then the two pictures show two possible arrangements of $T_3$. In the second picture, we get the $a^2$-angle $\beta_2$. In the third picture, we may determine $T_2$. 

\begin{figure}[htp]
\centering
\begin{tikzpicture}[>=latex]

\foreach \a in {0,1,2}
\foreach \b in {0,1,2,3,4}
{
\begin{scope}[xshift=3.5*\a cm]

\coordinate  (A\a X\b) at (90+72*\b:0.7);
\coordinate  (B\a X\b) at (90+72*\b:1.3);
\coordinate  (C\a X\b) at (126+72*\b:1.7);

\coordinate  (P\a X\b) at (90+72*\b:0.45);

\coordinate  (Q\a X\b) at (102+72*\b:0.8);
\coordinate  (R\a X\b) at (100+72*\b:1.16);
\coordinate  (S\a X\b) at (126+72*\b:1.45);
\coordinate  (T\a X\b) at (152+72*\b:1.16);
\coordinate  (U\a X\b) at (150+72*\b:0.8);

\coordinate  (L\a X\b) at (0:0.8);

\fill (0,0.7) circle (0.1);

\end{scope}
}

\foreach \a in {0,1,2}
{
\begin{scope}[xshift=3.5*\a cm]

\draw
	(A\a X1) -- (A\a X0) -- (A\a X4);

\draw[line width=1.5]
	(A\a X4) -- (A\a X3) -- (A\a X2);

\draw[dashed]
	(A\a X1) -- (A\a X2);

\node at (P\a X0) {\small $\beta$};
\node at (P\a X1) {\small $\delta$};
\node at (P\a X4) {\small $\alpha$};
\node at (P\a X2) {\small $\epsilon$};
\node at (P\a X3) {\small $\gamma$};

\node[draw,shape=circle, inner sep=0.5] at (0,0) {\small $1$};
\node[draw,shape=circle, inner sep=0.5] at (45:1.07) {\small $2$};\node[draw,shape=circle, inner sep=0.5] at (-18:1) {\small $3$};
\node[draw,shape=circle, inner sep=0.5] at (-90:1) {\small $4$};
\node[draw,shape=circle, inner sep=0.5] at (198:1) {\small $5$};
\node[draw,shape=circle, inner sep=0.5] at (135:1.07) {\small $6$};
	
\end{scope}
}


\draw
	(A0X1) -- (B0X1)  -- (C0X1)
	(A0X1) -- (60:1.5);
	
\draw[line width=1.5]
	(A0X4) -- (B0X4) -- (40:1.7)
	(A0X2) -- (B0X2) -- (C0X1);

\draw[dashed]
	(60:1.5) -- (40:1.7);
	
\draw[dotted]
	(A0X0) -- (120:1.5) -- (140:1.7) -- (B0X1)
	(B0X4) -- (C0X3) -- (B0X3) -- (C0X2) -- (B0X2)
	(A0X3) -- (B0X3);

\node at (Q0X2) {\small $\gamma$};

\node at (U0X3) {\small $\gamma$};

\node at (Q0X1) {\small $\delta$};
\node at (R0X1) {\small $\beta$};
\node at (S0X1) {\small $\alpha$};
\node at (T0X1) {\small $\gamma$};
\node at (U0X1) {\small $\epsilon$};

\node at (148:0.8) {\small $\beta$};

\node at (66:0.79) {\small $\beta$};
\node at (32:0.8) {\small $\alpha$};
\node at (28:1.2) {\small $\gamma$};
\node at (42:1.5) {\small $\epsilon$};
\node at (55:1.35) {\small $\delta$};


\draw
	(A1X2) -- (B1X2) -- (C1X1)
	(A1X4) -- (B1X4) -- (C1X3)
	(B1X2) -- (C1X2);
	
\draw[line width=1.5]
	(A1X1) -- (B1X1) -- (C1X1)
	(A1X3) -- (B1X3);

\draw[dashed]
	(C1X3) -- (B1X3) -- (C1X2);

\node at (Q1X1) {\small $\epsilon$};
\node at (R1X1) {\small $\gamma$};
\node at (S1X1) {\small $\alpha$};
\node at (T1X1) {\small $\beta$};
\node at (U1X1) {\small $\delta$};

\node at (Q1X2) {\small $\alpha$};
\node at (R1X2) {\small $\beta$};
\node at (S1X2) {\small $\delta$};
\node at (T1X2) {\small $\epsilon$};
\node at (U1X2) {\small $\gamma$};

\node at (Q1X3) {\small $\gamma$};
\node at (R1X3) {\small $\epsilon$};
\node at (S1X3) {\small $\delta$};
\node at (T1X3) {\small $\beta$};
\node at (U1X3) {\small $\alpha$};

\begin{scope}[xshift=3.5cm]

\draw
	(90:0.7) -- (120:1.5);

\draw[line width=1.5]
	(162:1.3) -- (140:1.7);

\draw[dashed]
	(140:1.7) -- (120:1.5);
	
\draw[dotted]
	(90:0.7) -- (60:1.5) -- (40:1.7) -- (18:1.3);

\node at (115:0.75) {\small $\beta$};
\node at (148:0.8) {\small $\alpha$};
\node at (152:1.2) {\small $\gamma$};
\node at (139:1.5) {\small $\epsilon$};
\node at (125:1.35) {\small $\delta$};

\node at (30:0.8) {\small $\beta$};

\end{scope}


\draw
	(A2X2) -- (B2X2) -- (C2X1)
	(B2X4) -- (C2X3) -- (B2X3)
	(B2X2) -- (C2X2);
	
\draw[line width=1.5]
	(A2X1) -- (B2X1) -- (C2X1)
	(A2X3) -- (B2X3);

\draw[dashed]
	(B2X3) -- (C2X2)
	(A2X4) -- (B2X4);

\node at (Q2X1) {\small $\epsilon$};
\node at (R2X1) {\small $\gamma$};
\node at (S2X1) {\small $\alpha$};
\node at (T2X1) {\small $\beta$};
\node at (U2X1) {\small $\delta$};

\node at (Q2X2) {\small $\alpha$};
\node at (R2X2) {\small $\beta$};
\node at (S2X2) {\small $\delta$};
\node at (T2X2) {\small $\epsilon$};
\node at (U2X2) {\small $\gamma$};

\node at (Q2X3) {\small $\gamma$};
\node at (R2X3) {\small $\alpha$};
\node at (S2X3) {\small $\beta$};
\node at (T2X3) {\small $\delta$};
\node at (U2X3) {\small $\epsilon$};

\begin{scope}[xshift=7cm]

\draw
	(120:1.5) -- (90:0.7) -- (60:1.5);

\draw[line width=1.5]
	(162:1.3) -- (140:1.7)
	(60:1.5) -- (40:1.7) -- (18:1.3);

\draw[dashed]
	(140:1.7) -- (120:1.5);

\node at (114:0.73) {\small $\beta$};
\node at (148:0.8) {\small $\alpha$};
\node at (152:1.2) {\small $\gamma$};
\node at (138:1.5) {\small $\epsilon$};
\node at (125:1.35) {\small $\delta$};

\node at (66:0.77) {\small $\beta$};
\node at (32:0.8) {\small $\delta$};
\node at (28:1.2) {\small $\epsilon$};
\node at (41:1.5) {\small $\gamma$};
\node at (55:1.35) {\small $\alpha$};

\end{scope}


\foreach \a in {1,...,3}
\node[xshift=-3.5 cm + 3.5*\a cm] at (-54:1.8) {2.\a};

\end{tikzpicture}
\caption{Case 2 for $a^2b^2c$.}
\label{2a2bcF2}
\end{figure}

\subsubsection*{Case 2.1}

By (the adjacency to) $\beta_6$, we have $A_{6,12}=\alpha,\delta$. Then by (the angle sums of) $\alpha^2\gamma$ and $\beta\delta^2$, and the edge length consideration, we get $H=\thin\alpha\thick\alpha\thin\beta\thin\beta\thin$, $\thin\alpha\thick\alpha\thin\beta\thin\beta\thin\beta\thin$, $\thick\alpha\thin\beta\thin\beta\thin\delta\dash\epsilon\thick$. The angle sums of $\alpha^2\gamma, \gamma\epsilon^2, \beta\delta^2,H$ and the angle sum for pentagon imply
\begin{align*}
H=\alpha^2\beta^2
&\colon
\alpha=\epsilon=(1-\tfrac{8}{f})\pi,\;
\beta=\tfrac{8}{f}\pi,\;
\gamma=\tfrac{16}{f}\pi,\;
\delta=(1-\tfrac{4}{f})\pi; \\
H=\alpha^2\beta^3
&\colon
\alpha=\epsilon=(1-\tfrac{12}{f})\pi,\;
\beta=\tfrac{8}{f}\pi,\;
\gamma=\tfrac{24}{f}\pi,\;
\delta=(1-\tfrac{4}{f})\pi; \\
H=\alpha\beta^2\delta\epsilon
&\colon
\alpha=\epsilon=(\tfrac{1}{2}-\tfrac{6}{f})\pi,\;
\beta=\tfrac{8}{f}\pi,\;
\gamma=(1+\tfrac{12}{f})\pi,\;
\delta=(1-\tfrac{4}{f})\pi.
\end{align*}

For $H=\alpha\beta^2\delta\epsilon$, we have $H=\thin\alpha^{\gamma}\thick^{\gamma}\epsilon\dash\cdots$. This gives a vertex $\gamma^2\cdots$, contradicting $\gamma>\pi$.

For $H=\alpha^2\beta^2,\alpha^2\beta^3$, we have $H=\thin\alpha^{\gamma}\thick^{\gamma}\alpha\thin\cdots$. This gives a vertex $\thick^{\epsilon}\gamma^{\alpha}\thick^{\alpha}\gamma^{\epsilon}\thick\cdots$. If $\thick^{\epsilon}\gamma^{\alpha}\thick^{\alpha}\gamma^{\epsilon}\thick\cdots$ has $a$ or $c$, then by the edge length consideration, the vertex has at least two of $\alpha$ or $\epsilon$. By $\gamma+\alpha>\pi$ and $\gamma+\epsilon>\pi$, we get a contradiction. Therefore $\thick^{\epsilon}\gamma^{\alpha}\thick^{\alpha}\gamma^{\epsilon}\thick\cdots$ has only $b$. This means $\thick^{\epsilon}\gamma^{\alpha}\thick^{\alpha}\gamma^{\epsilon}\thick\cdots=\thick\gamma\thick\gamma\thick\cdots\thick\gamma\thick=\gamma^k$. By Lemma \ref{aadlemma}, we have $\thick^{\alpha}\gamma^{\epsilon}\thick^{\epsilon}\gamma^{\alpha}\thick$ at this $\gamma^k$, which gives a vertex $\epsilon\thick\epsilon\cdots$. We have $\epsilon\thick\epsilon\cdots=\theta\dash\epsilon\thick\epsilon\dash\rho\cdots$, with $\theta,\rho=\delta$ or $\epsilon$. Moreover, by no $c^2$-angle, we know $\theta,\rho$ are two angles. By $\delta>\epsilon$, the angle sum of $\theta\dash\epsilon\thick\epsilon\dash\rho\cdots$ implies $4\epsilon\le 2\pi$. On the other hand, the angle sum of $\gamma^k$ implies $\gamma\le \frac{2}{3}\pi$. Then $\gamma+2\epsilon\le \frac{5}{3}\pi<2\pi$, contradicting the vertex $\gamma\epsilon^2$.

\subsubsection*{Case 2.2}

By $\beta_2$, we have $A_{2,16}=\alpha,\delta$. Then by $\alpha^2\beta, \alpha\delta\epsilon$, and the edge length consideration, we get $H=\beta^2\delta^2,\beta^3\delta^2$. The angle sums of $\alpha^2\beta, \alpha\delta\epsilon, \gamma^3,H$ and the angle sum for pentagon imply
\begin{align*}
H=\beta^2\delta^2
&\colon
\alpha=(\tfrac{5}{6}-\tfrac{2}{f})\pi,\;
\beta=(\tfrac{1}{3}+\tfrac{4}{f})\pi,\;
\gamma=\tfrac{2}{3}\pi,\;
\delta=(\tfrac{2}{3}-\tfrac{4}{f})\pi,\;
\epsilon=(\tfrac{1}{2}+\tfrac{6}{f})\pi; \\
H=\beta^3\delta^2
&\colon
\alpha=(\tfrac{5}{6}-\tfrac{2}{f})\pi,\;
\beta=(\tfrac{1}{3}+\tfrac{4}{f})\pi,\;
\gamma=\tfrac{2}{3}\pi,\;
\delta=(\tfrac{1}{2}-\tfrac{6}{f})\pi,\;
\epsilon=(\tfrac{2}{3}+\tfrac{8}{f})\pi.
\end{align*}
By the edge length consideration, we have $\epsilon_3\epsilon_4\cdots=\theta\dash\epsilon\thick\epsilon\dash\rho\cdots$, with $\theta,\rho=\delta$ or $\epsilon$, and $\theta,\rho$ are two angles. By $\delta+\epsilon>\pi$ and $2\epsilon>\pi$, we get a contradiction. 

\subsubsection*{Case 2.3}

By the edge length consideration, we have $H=\alpha^2\beta^3,\beta^3\delta^2,\beta^4,\beta^5$. 

Suppose $H=\alpha^2\beta^3$. The angle sums of $\alpha\delta\epsilon,\gamma^3,\alpha^2\beta^3$ and the angle sum for pentagon imply
\[
\alpha=(\tfrac{1}{2}-\tfrac{6}{f})\pi,\;
\beta=(\tfrac{1}{3}+\tfrac{4}{f})\pi,\;
\gamma=\tfrac{2}{3}\pi, \;
\delta+\epsilon=(\tfrac{3}{2}+\tfrac{6}{f})\pi.
\]
By the third of Figure \ref{2a2bcF2}, we have $H=\thin^{\alpha}\beta^{\delta}\thin^{\beta}\alpha^{\gamma}\thick^{\gamma}\alpha^{\beta}\thin^{\alpha}\beta^{\delta}\thin^{\alpha}\beta^{\delta}\thin$. This gives $\thin\beta\thin\delta\dash\cdots=\thin\beta\thin\delta\dash\delta\thin\cdots$ or $\thin\beta\thin\delta\dash\epsilon\thick\cdots$.

The vertex $\thin\beta\thin\delta\dash\delta\thin\cdots$ implies 
$\delta\le \pi-\frac{1}{2}\beta=(\frac{5}{6}-\frac{2}{f})\pi$ and $\epsilon=(\frac{3}{2}+\frac{6}{f})\pi-\delta\ge (\frac{2}{3}+\frac{8}{f})\pi>\alpha,\gamma$. The AAD of $\thin\delta\dash\delta\thin$ is $\thin^{\beta}\delta^{\epsilon}\dash^{\epsilon}\delta^{\beta}\thin$. This gives a vertex $\thick\epsilon\dash\epsilon\thick\cdots$. By $\gamma+2\epsilon>3\gamma=2\pi$, we know $\thick\epsilon\dash\epsilon\thick\cdots$ has no $\gamma$. Then by the edge length consideration, we have $\thick\epsilon\dash\epsilon\thick\cdots=\theta\thick\epsilon\dash\epsilon\thick\rho\cdots$, where $\theta,\rho=\alpha,\epsilon$ are two angles. By $\alpha<\epsilon$ and $\alpha+\epsilon\ge (\tfrac{1}{2}-\tfrac{6}{f})\pi+(\frac{2}{3}+\frac{8}{f})\pi>\pi$, we know $\theta\thick\epsilon\dash\epsilon\thick\rho\cdots$ is not a vertex. 

By $R(\thin\beta\thin\delta\dash\epsilon\thick\cdots)=(\frac{1}{6}-\frac{10}{f})\pi<\alpha,\beta,\gamma$, the remainder has only $\delta,\epsilon$. Then we get $\thin\beta\thin\delta\dash\epsilon\thick\cdots=\delta\thin\beta\thin\delta\dash\epsilon\thick\epsilon\cdots$, contradicting $\delta+\epsilon>\pi$.

Suppose $H=\beta^3\delta^2$. The angle sums of $\alpha\delta\epsilon,\gamma^3,\beta^3\delta^2$ and the angle sum for pentagon imply
\[
\alpha+\epsilon=(\tfrac{3}{2}+\tfrac{6}{f})\pi,\;
\beta=(\tfrac{1}{3}+\tfrac{4}{f})\pi,\;
\gamma=\tfrac{2}{3}\pi,\;
\delta=(\tfrac{1}{2}-\tfrac{6}{f})\pi.
\]
By the third of Figure \ref{2a2bcF2}, we have $H=\thin^{\alpha}\beta^{\delta}\thin^{\beta}\delta^{\epsilon}\dash^{\epsilon}\delta^{\beta}\thin^{\alpha}\beta^{\delta}\thin^{\alpha}\beta^{\delta}\thin$. Therefore $\thick\alpha\thin\beta\thin\cdots$ and $\thick\epsilon\dash\epsilon\thick\cdots$ are vertices.

We have $\thick\epsilon\dash\epsilon\thick\cdots=\thick\gamma\thick\epsilon\dash\epsilon\thick,\theta\thick\epsilon\dash\epsilon\thick\rho\cdots$, where $\theta,\rho=\alpha,\gamma,\epsilon$ are two angles. This implies one of $\gamma+2\epsilon\le 2\pi$, $\alpha+\epsilon\le\pi$, $\gamma+\epsilon\le\pi$,  $2\epsilon\le \pi$, holds. By $\alpha+\epsilon>\pi$ and $\gamma=\tfrac{2}{3}\pi$, this means $\epsilon\le \frac{2}{3}\pi=\gamma$. Then $\alpha=(\tfrac{3}{2}+\tfrac{6}{f})\pi-\epsilon\ge (\tfrac{5}{6}+\tfrac{6}{f})\pi>\gamma\ge \epsilon$.

We have $\thick\alpha\thin\beta\thin\cdots=\thin\alpha\thick\alpha\thin\beta\thin\cdots,\thick\gamma\thick\alpha\thin\beta\thin\cdots,\dash\epsilon\thick\alpha\thin\beta\thin\cdots$. By $2\alpha+\beta\ge 2(\tfrac{5}{6}+\tfrac{6}{f})\pi+(\tfrac{1}{3}+\tfrac{4}{f})\pi>2\pi$, we know $\thin\alpha\thick\alpha\thin\beta\thin\cdots$ is not a vertex.

By $R(\dash\epsilon\thick\alpha\thin\beta\thin\cdots)=(\frac{1}{6}-\frac{10}{f})\pi<\beta,\gamma,\delta$, the remainder has only $\alpha,\epsilon$. Then we get $\dash\epsilon\thick\alpha\thin\beta\thin\cdots=\epsilon\dash\epsilon\thick\alpha\thin\beta\thin\alpha\cdots$, contradicting $\alpha+\epsilon>\pi$.

Suppose $\thick\gamma\thick\alpha\thin\beta\thin\cdots$ is a vertex. Then $\alpha\le 2\pi-\beta-\gamma=(1-\frac{4}{f})\pi$, and $\epsilon=(\tfrac{3}{2}+\tfrac{6}{f})\pi-\alpha\ge(\frac{1}{2}+\frac{10}{f})\pi>\frac{1}{2}\pi$. By $\alpha>\gamma\ge \epsilon$, this implies $\theta\thick\epsilon\dash\epsilon\thick\rho\cdots$ ($\theta,\rho=\alpha,\gamma,\epsilon$ are two angles) is not a vertex. Therefore $\thick\epsilon\dash\epsilon\thick\cdots=\thick\gamma\thick\epsilon\dash\epsilon\thick$ is a vertex. The angle sum of $\gamma\epsilon^2$ further implies
\[
\alpha=(\tfrac{5}{6}+\tfrac{6}{f})\pi,\;
\beta=(\tfrac{1}{3}+\tfrac{4}{f})\pi,\;
\gamma=\epsilon=\tfrac{2}{3}\pi,\;
\delta=(\tfrac{1}{2}-\tfrac{6}{f})\pi.
\]
Then $R(\thick\gamma\thick\alpha\thin\beta\thin\cdots)=(\tfrac{1}{6}-\tfrac{10}{f})\pi<$ all angles. Since $\thick\gamma\thick\alpha\thin\beta\thin$ is not a vertex by the edge length consideration, this implies $\thick\gamma\thick\alpha\thin\beta\thin\cdots$ is not a vertex.

We will discuss the remaining cases $H=\beta^4,\beta^5$ in Section \ref{pent_division}.

\subsection{Case 3$(a^2b^2c)$}
\label{2a2bcCase3}

Let the third of Figure \ref{2a2bc_pentagon} be the center tile $T_1$ in the partial neighbourhood in Figure \ref{2a2bcF4}. We consider two possible arrangements of $T_6$. The first and second pictures show one arrangement. The third picture shows the other arrangement. We label the three partial neighbourhood tilings as Cases 3.1, 3.2, 3.3. Again, the degree of $H$ is $4$ or $5$.

In the first and second pictures, we may determine $T_5$ and the $b^2$-angle $\gamma_4$. Then the two pictures show two possible arrangements of $T_4$. In the first picture, we may determine $T_3$ and the $a^2$-angle $\beta_2$. In the second picture, we get the $a^2$-angle $\beta_3$. Then $E_{56}=b$ or $c$, and either way gives $\beta_2$. In the third picture, we get the $b^2$-angle $\gamma_5$. Then $E_{34}=a$ or $c$, and either way gives $\gamma_4$. This further determines $T_3$ and gives the $a^2$-angle $\beta_2$. 

\begin{figure}[htp]
\centering
\begin{tikzpicture}[>=latex]

\foreach \a in {0,1,2}
\foreach \b in {0,1,2,3,4}
{
\begin{scope}[xshift=3.5*\a cm]

\coordinate  (A\a X\b) at (90+72*\b:0.7);
\coordinate  (B\a X\b) at (90+72*\b:1.3);
\coordinate  (C\a X\b) at (126+72*\b:1.7);

\coordinate  (P\a X\b) at (90+72*\b:0.45);

\coordinate  (Q\a X\b) at (102+72*\b:0.8);
\coordinate  (R\a X\b) at (100+72*\b:1.16);
\coordinate  (S\a X\b) at (126+72*\b:1.45);
\coordinate  (T\a X\b) at (152+72*\b:1.16);
\coordinate  (U\a X\b) at (150+72*\b:0.8);

\coordinate  (L\a X\b) at (0:0.8);

\fill (0,0.7) circle (0.1);

\end{scope}
}

\foreach \a in {0,1,2}
{
\begin{scope}[xshift=3.5*\a cm]

\draw
	(A\a X0) -- (A\a X4) -- (A\a X3);

\draw[line width=1.5]
	(A\a X1) -- (A\a X2) -- (A\a X3);

\draw[dashed]
	(A\a X0) -- (A\a X1);

\node at (P\a X0) {\small $\delta$};
\node at (P\a X1) {\small $\epsilon$};
\node at (P\a X4) {\small $\beta$};
\node at (P\a X2) {\small $\gamma$};
\node at (P\a X3) {\small $\alpha$};

\node[draw,shape=circle, inner sep=0.5] at (0,0) {\small $1$};
\node[draw,shape=circle, inner sep=0.5] at (45:1.07) {\small $2$};
\node[draw,shape=circle, inner sep=0.5] at (-18:1) {\small $3$};
\node[draw,shape=circle, inner sep=0.5] at (-90:1) {\small $4$};
\node[draw,shape=circle, inner sep=0.5] at (198:1) {\small $5$};
\node[draw,shape=circle, inner sep=0.5] at (135:1.07) {\small $6$};
	
\end{scope}
}


\draw
	(A0X1) -- (B0X1) -- (C0X1)
	(B0X2) -- (C0X2) -- (B0X3)
	(A0X4) -- (B0X4)
	(162:1.3) -- (140:1.7);
	
\draw[line width=1.5]
	(A0X2) -- (B0X2)
	(B0X3) -- (C0X3) -- (B0X4)
	(A0X0) -- (120:1.5) -- (140:1.7);

\draw[dashed]
	(C0X1) -- (B0X2)
	(A0X3) -- (B0X3);

\draw[dotted]
	(A0X1) -- (60:1.5) -- (40:1.7) -- (B0X4);
	
\node at (Q0X1) {\small $\alpha$};
\node at (R0X1) {\small $\beta$};
\node at (S0X1) {\small $\delta$};
\node at (T0X1) {\small $\epsilon$};
\node at (U0X1) {\small $\gamma$};

\node at (Q0X2) {\small $\gamma$};
\node at (R0X2) {\small $\alpha$};
\node at (S0X2) {\small $\beta$};
\node at (T0X2) {\small $\delta$};
\node at (U0X2) {\small $\epsilon$};

\node at (Q0X3) {\small $\delta$};
\node at (R0X3) {\small $\epsilon$};
\node at (S0X3) {\small $\gamma$};
\node at (T0X3) {\small $\alpha$};
\node at (U0X3) {\small $\beta$};

\node at (115:0.75) {\small $\epsilon$};
\node at (148:0.8) {\small $\delta$};
\node at (152:1.2) {\small $\beta$};
\node at (139:1.5) {\small $\alpha$};
\node at (125:1.35) {\small $\gamma$};

\node at (32:0.8) {\small $\beta$};


\begin{scope}[xshift=3.5cm]

\draw
	(A1X1) -- (B1X1) -- (C1X1)
	(A1X3) -- (B1X3) -- (C1X2)
	(90:0.7) -- (60:1.5)
	(B1X1) -- (140:1.7);
	
\draw[line width=1.5]
	(A1X2) -- (B1X2)
	(A1X0) -- (120:1.5) -- (140:1.7);

\draw[dashed]
	(C1X1) -- (B1X2) -- (C1X2);

\draw[dotted]
	(B1X3) -- (C1X3) -- (B1X4) -- (A1X4)
	(60:1.5) -- (40:1.7) -- (18:1.3);

\node at (Q1X1) {\small $\alpha$};
\node at (R1X1) {\small $\beta$};
\node at (S1X1) {\small $\delta$};
\node at (T1X1) {\small $\epsilon$};
\node at (U1X1) {\small $\gamma$};

\node at (Q1X2) {\small $\gamma$};
\node at (R1X2) {\small $\epsilon$};
\node at (S1X2) {\small $\delta$};
\node at (T1X2) {\small $\beta$};
\node at (U1X2) {\small $\alpha$};

\node at (Q1X3) {\small $\beta$};

\node at (115:0.75) {\small $\epsilon$};
\node at (148:0.8) {\small $\delta$};
\node at (152:1.2) {\small $\beta$};
\node at (139:1.5) {\small $\alpha$};
\node at (125:1.35) {\small $\gamma$};

\node at (65:0.8) {\small $\beta$};

\end{scope}


\begin{scope}[xshift=7cm]

\draw
	(A2X4) -- (B2X4)
	(A2X0) -- (120:1.5) -- (140:1.7);
	
\draw[line width=1.5]
	(A2X1) -- (B2X1) -- (140:1.7)
	(A2X3) -- (B2X3) -- (C2X3);

\draw[dashed]
	(B2X4) -- (C2X3);

\draw[dotted]
	(A2X1) -- (60:1.5) -- (40:1.7) -- (B2X4)
	(B2X1) -- (C2X1) -- (B2X2) -- (C2X2) -- (B2X3)
	(A2X2) -- (B2X2);

\node at (Q2X1) {\small $\gamma$};
\node at (U2X2) {\small $\gamma$};

\node at (Q2X3) {\small $\alpha$};
\node at (R2X3) {\small $\gamma$};
\node at (S2X3) {\small $\delta$};
\node at (T2X3) {\small $\epsilon$};
\node at (U2X3) {\small $\beta$};
	
\node at (32:0.8) {\small $\beta$};

\node at (115:0.75) {\small $\delta$};
\node at (148:0.8) {\small $\epsilon$};
\node at (152:1.2) {\small $\gamma$};
\node at (139:1.5) {\small $\alpha$};
\node at (127:1.33) {\small $\beta$};

\end{scope}


\foreach \a in {1,...,3}
\node[xshift=-3.5 cm + 3.5*\a cm] at (-54:1.8) {3.\a};

\end{tikzpicture}
\caption{Case 3 for $a^2b^2c$.}
\label{2a2bcF4}
\end{figure}

\subsubsection*{Case 3.1}

The angle sums of $\alpha\delta\epsilon,\beta^3,\gamma^3$ and the angle sum for pentagon imply $f=12$. By $f\ge 16$, we may dismiss the case.

\subsubsection*{Case 3.2}

By $\alpha\delta\epsilon$, we know $R(\beta\delta\epsilon\cdots)$ has no $\alpha$. Then by the edge length consideration, we have $H=\beta\delta^2\epsilon^2$. The angle sums of $\alpha^2\beta,\gamma^3,\beta\delta^2\epsilon^2$ imply
\[
\alpha+\beta+\gamma+\delta+\epsilon
=\tfrac{1}{2}(2\alpha+\beta)+\tfrac{1}{3}3\gamma+\tfrac{1}{2}(\beta+2\delta+2\epsilon)
=\tfrac{8}{3}\pi,
\]
contradicting the angle sum for pentagon.

\subsubsection*{Case 3.3}

The angle sums of $\alpha^2\gamma,\beta^3,\gamma\epsilon^2$ and the angle sum for pentagon imply
\[
\alpha=\epsilon=\pi-\tfrac{1}{2}\gamma,\;
\beta=\tfrac{2}{3}\pi,\;
\delta=(\tfrac{1}{3}+\tfrac{4}{f})\pi.
\]
By $\beta_2$, we have $A_{2,16}=\alpha,\delta$. Then by $\alpha^2\gamma$, $\beta+4\delta>2\pi$, and the edge length consideration, we have $H=\alpha^2\beta\delta^2,\alpha^2\delta^2,\alpha\delta^3\epsilon,\delta^4$. 

The angle sum of $H$ further implies 
\begin{align*}
H=\alpha^2\beta\delta^2
&\colon
\alpha=\epsilon=(\tfrac{1}{3}-\tfrac{4}{f})\pi,\;
\beta=\tfrac{2}{3}\pi,\;
\gamma=(\tfrac{4}{3}+\tfrac{8}{f})\pi,\;
\delta=(\tfrac{1}{3}+\tfrac{4}{f})\pi; \\
H=\alpha^2\delta^2
&\colon
\alpha=\epsilon=(\tfrac{2}{3}-\tfrac{4}{f})\pi,\;
\beta=\tfrac{2}{3}\pi,\;
\gamma=(\tfrac{2}{3}+\tfrac{8}{f})\pi,\;
\delta=(\tfrac{1}{3}+\tfrac{4}{f})\pi; \\
H=\alpha\delta^3\epsilon
&\colon
\alpha=\epsilon=(\tfrac{1}{2}-\tfrac{6}{f})\pi,\;
\beta=\tfrac{2}{3}\pi,\;
\gamma=(1+\tfrac{12}{f})\pi,\;
\delta=(\tfrac{1}{3}+\tfrac{4}{f})\pi.
\end{align*}
By the edge length consideration, $H=\alpha^2\beta\delta^2, \alpha^2\delta^2$ is $\thin^{\beta}\alpha^{\gamma}\thick^{\gamma}\alpha^{\beta}\thin\cdots$, and $H=\alpha\delta^3\epsilon$ is $\thin^{\beta}\alpha^{\gamma}\thick^{\gamma}\epsilon^{\delta}\dash\cdots$. Therefore $\thick\gamma\thick\gamma\thick\cdots$ is a vertex. This implies $\gamma<\pi$. By the list of angle values, this implies $H=\alpha^2\delta^2$. Then  $R(\thick\gamma\thick\gamma\thick\cdots)=(\tfrac{2}{3}-\tfrac{16}{f})\pi<\alpha,\beta,\gamma,2\delta,\epsilon$. Therefore the remainder is a single $\delta$. However, $\gamma^2\delta$ is not a vertex by the edge length consideration. We get a contradiction.

We will discuss the remaining case $H=\delta^4$ in Section \ref{pent_division}.

\subsection{Pentagonal Subdivision}
\label{pent_division}

After Sections \ref{2a2bcCase1}, \ref{2a2bcCase2}, \ref{2a2bcCase3}, the only remaining cases for the edge combination $a^2b^2c$ are the following.
\begin{description}
\item Case 2.3, $H=\beta^4$: We have the third of Figure \ref{2a2bcF2}. The angle sums of $
\alpha\delta\epsilon,\gamma^3,\beta^4$ and the angle sum for pentagon imply
\[
f=24\colon 
\alpha+\delta+\epsilon=2\pi,\;
\beta=\tfrac{1}{2}\pi,\;
\gamma=\tfrac{2}{3}\pi. 
\]
\item Case 2.3, $H=\beta^5$: We have the third of Figure \ref{2a2bcF2}. The angle sums of $
\alpha\delta\epsilon,\gamma^3,\beta^5$ and the angle sum for pentagon imply
\[
f=60\colon
\alpha+\delta+\epsilon=2\pi,\;
\beta=\tfrac{2}{5}\pi,\;
\gamma=\tfrac{2}{3}\pi. 
\]
\item Case 3.3, $H=\delta^4$: We have the third of Figure \ref{2a2bcF4}. The angle sums of $
\alpha^2\gamma,\beta^3,\gamma\epsilon^2,\delta^4$ and the angle sum for pentagon imply
\[
f=24\colon
\alpha=\epsilon=\pi-\tfrac{1}{2}\gamma,\;
\beta=\tfrac{2}{3}\pi,\;
\delta=\tfrac{1}{4}\pi.
\]
\end{description}
We also recall that there is no $3^5$-tile. By Lemma \ref{base_tile2}, for $f=24$, this means each tile is a $3^44$-tile. 

\subsubsection*{Case 3.3. $H=\delta^4$}

By the edge length consideration, we have $H=\delta^4=\thick\delta\thin\delta\thick\delta\thin\delta\thick$. The four tiles $T_1,T_2,T_3,T_4$ around a vertex $\delta^4$ is given by the first of Figure \ref{2a2bcF7}. Since every tile is a $3^44$-tile, all vertices of the four tiles except the central $\delta^4$ have degree $3$. In particular, we have tiles $T_5,T_6,T_7$ as indicated, with their edges (those not shared with $T_1,T_2,T_3$) and angles to be determined. Up to the symmetry of vertical flip, we may assume that the edges of $T_5$ are given as indicated. Then $E_{26}=b$ and $E_{56}=c$ determine $T_6$. Then $T_7$ has $3$ $b$-edges, a contradiction. Therefore the case has no tiling.

\begin{figure}[htp]
\centering
\begin{tikzpicture}[>=latex,scale=1]


\foreach \a in {1,-1}
{
\begin{scope}[yscale=\a]

\draw
	(-1.2,0.6) -- (-0.8,0) -- (0.8,0) -- (1.2,0.6);

\draw[line width=1.5]
	(1.2,0.6) -- (0.6,1.2) -- (0,0.8) -- (-0.6,1.2) -- (-1.2,0.6);

\draw[dashed]
	(0,0) -- (0,0.8);

\node at (0.2,-0.2) {\small $\delta$};
\node at (-0.2,-0.2) {\small $\delta$};

\end{scope}
}

\draw
	(1.8,0.6) -- (1.8,1.8) -- (0.6,1.8);
	
\draw[line width=1.5]
	(1.2,-0.6) -- (1.8,-0.6) -- (1.8,0.6)
	(0.6,1.2) -- (0.6,1.8);

\draw[dashed]
	(1.2,0.6) -- (1.8,0.6);

\draw[dotted]
	(0.6,1.8) -- (-0.6,1.8) -- (-0.6,1.2);

\node[inner sep=1,draw,shape=circle] at (-0.55,0.55) {\small $1$};
\node[inner sep=1,draw,shape=circle] at (0.55,0.55) {\small $2$};
\node[inner sep=1,draw,shape=circle] at (0.55,-0.55) {\small $3$};
\node[inner sep=1,draw,shape=circle] at (-0.55,-0.55) {\small $4$};
\node[inner sep=1,draw,shape=circle] at (1.45,0) {\small $5$};
\node[inner sep=1,draw,shape=circle] at (1.3,1.3) {\small $6$};
\node[inner sep=1,draw,shape=circle] at (0,1.45) {\small $7$};


\begin{scope}[xshift=4cm]

\foreach \a in {1,-1}
\foreach \b in {1,-1}
{
\begin{scope}[yscale=\a,xscale=\b]

\draw
	(0.8,0) -- (0,0) -- (0,0.8);

\draw[line width=1.5]
	(0,0.8) -- (0.6,1.2) -- (1.2,0.6);

\draw[dashed]
	(0.8,0) -- (1.2,0.6);

\node at (0.2,0.2) {\small $\beta$};

\end{scope}
}

\node at (1,0) {?};

\end{scope}


\begin{scope}[xshift=8cm]

\foreach \a in {-1,0,1}
{
\begin{scope}[rotate=90*\a]

\draw
	(0,0) -- (0.8,0);

\draw[line width=1.5]
	(1.2,0.6) -- (0.6,1.2) -- (0,0.8);

\draw[dashed]
	(0.8,0) -- (1.2,0.6);

\end{scope}
}

\foreach \a in {0,1}
{
\begin{scope}[rotate=90*\a]

\draw
	(1.2,0.6) -- (1.8,0.6)
	(1.8,-0.6) -- (1.8,1.8);

\draw[line width=1.5]
	(0.6,1.2) -- (0.6,1.8);

\draw[dashed]
	(0.6,1.8) -- (1.8,1.8);

\node at (0.7,0.2) {\small $\delta$}; 
\node at (1,0.6) {\small $\epsilon$};
\node at (0.6,0.95) {\small $\gamma$};
\node at (0.2,0.7) {\small $\alpha$};
\node at (0.2,-0.2) {\small $\beta$};

\node at (1.3,0.8) {\small $\alpha$}; 
\node at (1.6,0.8) {\small $\beta$};
\node at (1.6,1.6) {\small $\delta$};
\node at (0.8,1.6) {\small $\epsilon$};
\node at (0.8,1.3) {\small $\gamma$};

\node at (1,0) {\small $\epsilon$}; 
\node at (1.3,-0.4) {\small $\gamma$};
\node at (1.3,0.4) {\small $\delta$};
\node at (1.6,-0.4) {\small $\alpha$};
\node at (1.6,0.4) {\small $\beta$};

\end{scope}
}

\draw
	(0,0) -- (-0.8,0)
	(-0.6,1.8) -- (-0.6,2.5);

\draw[line width=1.5]
	(1.2,-0.6) -- (1.8,-0.6)
	(-0.6,2.5) -- (1.8,2.5) -- (1.8,1.8);

\draw[dotted]
	(-0.8,0) -- (-1.2,-0.6) -- (-0.6,-1.2) -- (0,-0.8);

\node at (-0.4,2) {\small $\beta$}; 
\node at (1.6,2) {\small $\epsilon$};
\node at (1.6,2.3) {\small $\gamma$};
\node at (0.6,2) {\small $\delta$};
\node at (-0.4,2.3) {\small $\alpha$};

\node at (-0.2,0.2) {\small $\beta$};
\node at (-0.2,-0.2) {\small $\beta$};
\node at (-0.8,2) {\small $\beta$};

\node[inner sep=1,draw,shape=circle] at (-0.55,0.55) {\small $1$};
\node[inner sep=1,draw,shape=circle] at (0.55,0.55) {\small $2$};
\node[inner sep=1,draw,shape=circle] at (0.55,-0.55) {\small $3$};
\node[inner sep=1,draw,shape=circle] at (-0.55,-0.55) {\small $4$};
\node[inner sep=1,draw,shape=circle] at (0,1.45) {\small $5$};
\node[inner sep=1,draw,shape=circle] at (-1.3,1.3) {\small $6$};
\node[inner sep=1,draw,shape=circle] at (1.3,1.3) {\small $7$};
\node[inner sep=1,draw,shape=circle] at (1.45,0) {\small $8$};
\node[inner sep=1,draw,shape=circle] at (0,2.15) {\small $9$};

\end{scope}

\end{tikzpicture}
\caption{Case 3.3, $H=\delta^4$, and Case 2.3, $H=\beta^4$.}
\label{2a2bcF7}
\end{figure}

\subsubsection*{Case 2.3. $H=\beta^4$}

We consider the four tiles around $\beta^4$. First, we assume that adjacent tiles always have opposite orientations (in the sense of the direction of edges $a\to a\to b\to b\to c$). Then we get the second of Figure \ref{2a2bcF7}. Since every tile is a $3^44$-tile, the vertex labeled by the question mark has degree $3$. This implies that there is a $c^2$-angle, a contradiction. Therefore there are two adjacent tiles around $\beta^4$ with the same orientation, say $T_1,T_2$ in the third of Figure \ref{2a2bcF7}. We may assume that the edges and angles of $T_1,T_2$ are given as indicated. Since every tile is a $3^44$-tile, we also have tiles $T_3,T_4,\dots,T_8$, with their edges and angles to be determined.

By $E_{15}=c$ and $E_{25}=b$, we determine $T_5$. By $E_{16}=b$ and $E_{56}=a$, we determine $T_6$. Since $T_7$ already has two $b$-edges, and $E_{28}=c$, we have $E_{78}=a$. This determines $T_7,T_8$. Then $E_{23}=a$ and $E_{38}=b$ determine $T_3$. This shows that, by starting with $T_1,T_2$, we can determine $T_3$. By repeating with $T_2,T_3$ in place of $T_1,T_2$, we can further determine $T_4$. We conclude that the neighbourhood of $\beta^4$ is uniquely given by four tiles with the same orientation.

If $\beta_5\beta_6\cdots=\beta^2\cdots$ has degree $3$, then by the edge length consideration, it is $\beta^3$, contradicting the fact that $\beta^4$ is already a vertex. Since every tile is a $3^44$-tile, the vertex $\beta_5\beta_6\cdots$ has degree $4$, and the vertex $\alpha_5\epsilon_7\cdots$ has degree $3$. Then $\alpha_5\epsilon_7\cdots=\alpha_5\delta_9\epsilon_7$ determines $T_9$. By the edge length consideration, the degree four vertex $\beta_5\beta_6\beta_9\cdots=\beta^4$.

We started with $\beta_1\beta_2\beta_3\beta_4=\beta^4$, and derived $\beta_5\beta_6\beta_9\cdots=\beta^4$. We may repeat the argument from $\beta_5\beta_6\beta_9\cdots$ in place of $\beta_1\beta_2\beta_3\beta_4$. After repeating six times, we get the pentagonal subdivision of the octahedron, in the second of Figure \ref{subdivision_tiling}.

\subsubsection*{Case 2.3. $H=\beta^5$}

After finishing Case 2.3, $H=\beta^4$, and Case 3.3, $H=\delta^4$, we may further assume that there is no $3^44$-tile. By Lemma \ref{base_tile3}, for $f=60$, we know each tile is a $3^45$-tile. 

Among the five tiles around $\beta^5$, there must be two adjacent tiles with the same orientation. Then the argument given by the third of Figure \ref{2a2bcF7} can be carried out, because the argument does not use $T_4$. The argument shows that all five tiles have the same orientation, and also derives another degree $5$ vertex $\beta_5\beta_6\beta_9\cdots$, with three tiles $T_5,T_6,T_9$ having the same orientation. By the edge length consideration, we have $\beta_5\beta_6\beta_9\cdots=\beta^5,\alpha^2\beta^3,\beta^3\delta^2$. We need to exclude $\alpha^2\beta^3$ and $\beta^3\delta^2$.

For $\alpha^2\beta^3$, the tiles $T_5,T_6,T_9$ give the AAD $\thin^{\alpha}\beta^{\delta}\thin^{\alpha}\beta^{\delta}\thin^{\beta}\alpha^{\gamma}\thick^{\gamma}\alpha^{\beta}\thin^{\alpha}\beta^{\delta}\thin$. The AAD gives a vertex $\thick\alpha^{\beta}\thin^{\alpha}\beta\thin\cdots$. Since $\alpha^2\beta^3$ and $\thick\alpha^{\beta}\thin^{\alpha}\beta\thin\cdots$ are vertices of the same tile, and each tile is a $3^45$-tile, we know $\thick\alpha^{\beta}\thin^{\alpha}\beta\thin\cdots$ has degree $3$. Then by the edge lengths, we get $\thick\alpha^{\beta}\thin^{\alpha}\beta\thin\cdots=\thin\alpha\thick\alpha\thin\beta\thin=\alpha^2\beta$. However, the vertices $\alpha^2\beta^3$ and $\alpha^2\beta$ are contradictory.

Similarly, the vertex $\beta^3\delta^2$ has the AAD $\thin^{\alpha}\beta^{\delta}\thin^{\alpha}\beta^{\delta}\thin^{\beta}\delta^{\epsilon}\thick^{\epsilon}\delta^{\beta}\thin^{\alpha}\beta^{\delta}\thin$. The AAD gives a vertex $\thin\beta^{\delta}\thin^{\beta}\delta\dash\cdots$. By $3^45$-tile, the vertex $\thin\beta^{\delta}\thin^{\beta}\delta\dash\cdots$ has degree $3$ and therefore is $\beta\delta^2$. This contradicts the vertex $\beta^3\delta^2$.

We conclude $\beta_5\beta_6\beta_9\cdots=\beta^5$. Then the process of deriving new $\beta_5\beta_6\beta_9\cdots=\beta^5$ from the original $\beta^5$ can be repeated. After repeating the process twelve times, we get the pentagonal subdivision of the icosahedron, in the third of Figure \ref{subdivision_tiling}.

\section{Tiling for Edge Combination $a^3bc$}
\label{3abc}

This section is devoted to the classification of edge-to-edge tilings of the sphere by congruent pentagons with the edge combination $a^3bc$. By Lemma \ref{edge_combo}, the pentagon is the second of Figure \ref{pentagon}. Similar to the edge combination $a^2b^2c$, we divide the classification into three cases in Figure \ref{3abc_pentagon}, according to the location of the vertex $H$ of a special tile (indicated by $\bullet$). In the first case, the degree of $H$ is $3,4$ or $5$. In the second and third cases, the degree of $H$ is $4$ or $5$. In other words, we do not need to consider $3^5$-tile in the second and third cases. By Lemma \ref{base_tile2}, we have $f\ge 24$ in the second and third cases. 

\begin{figure}[htp]
\centering
\begin{tikzpicture}[>=latex,scale=1]

\foreach \a in {1,2,3}
{
\begin{scope}[xshift=-2cm+2*\a cm]

\draw
	(18:0.8) -- (-54:0.8) -- (234:0.8) -- (162:0.8);

\draw[line width=1.5]
	(162:0.8) -- (90:0.8);

\draw[dashed]
	(18:0.8) -- (90:0.8);

\fill (18+72*\a:0.8) circle (0.1);

\node at (90:0.5) {\small $\alpha$};
\node at (162:0.5) {\small $\beta$};
\node at (18:0.5) {\small $\gamma$};
\node at (234:0.5) {\small $\delta$};
\node at (-54:0.5) {\small $\epsilon$};

\node at (0,0) {\a};

\end{scope}
}

\end{tikzpicture}
\caption{Special tile for the edge combination $a^3bc$.}
\label{3abc_pentagon}
\end{figure}

For the edge combination $a^3bc$, we will need \cite[Lemma 21]{gsy}. The special case relevant to us is reproduced below.

\begin{lemma}\label{geometry}
For the pentagon in the second of Figure \ref{pentagon}, we cannot have $\beta=\gamma$ and $\delta=\epsilon$ holding at the same time.
\end{lemma}

\subsection{Case 1$(a^3bc)$}
\label{3abcCase1}

Let the first of Figure \ref{3abc_pentagon} be the center tile $T_1$ in the partial neighbourhood in Figure \ref{3abcF1}. We consider whether $E_{23}=a$ or $E_{56}=a$. If both are not $a$, then we get the first picture. If one is $a$ and the other is not, by the symmetry of exchanging $b$ and $c$, we may assume $E_{23}=b$ and $E_{56}=a$, and get the second picture. If both are $a$, then we get the third picture. We label the three partial neighbourhood tilings as Cases 1.1, 1.2, 1.3.

\begin{figure}[htp]
\centering
\begin{tikzpicture}[>=latex]


\foreach \a in {0,1,2}
{
\begin{scope}[xshift=3.5*\a cm]

\fill (0,0.7) circle (0.1);

\foreach \x in {0,1,2} 
\draw[rotate=-72*\x]
	(18:0.7) -- (-54:0.7);
	
\draw[line width=1.5]
	(162:0.7) -- (90:0.7);
	
\draw[dashed]
	(18:0.7) -- (90:0.7);
	
\node at (90:0.45) {\small $\alpha$}; 
\node at (162:0.45) {\small $\beta$};
\node at (18:0.45) {\small $\gamma$};
\node at (234:0.45) {\small $\delta$};
\node at (-54:0.45) {\small $\epsilon$};

\node[draw,shape=circle, inner sep=0.5] at (0,0) {\small $1$};
\node[draw,shape=circle, inner sep=0.5] at (45:1.05) {\small $2$};
\node[draw,shape=circle, inner sep=0.5] at (-18:1.05) {\small $3$};
\node[draw,shape=circle, inner sep=0.5] at (-90:1.05) {\small $4$};
\node[draw,shape=circle, inner sep=0.5] at (198:1.05) {\small $5$};
\node[draw,shape=circle, inner sep=0.5] at (135:1.05) {\small $6$};

\end{scope}
}


\draw
	(90:0.7) -- (60:1.5)
	(90:0.7) -- (60:1.5) -- (40:1.7) -- (18:1.3)
	(90:0.7) -- (120:1.5) -- (140:1.7) -- (162:1.3)
	(-18:1.7) -- (-54:1.3) -- (-54:0.7)
	(198:1.7) -- (234:1.3) -- (234:0.7);
	
\draw[line width=1.5]
	(162:0.7) -- (90:0.7)
	(18:0.7) -- (18:1.3)
	(162:1.3) -- (198:1.7);
	
\draw[dashed]
	(18:0.7) -- (90:0.7)
	(18:1.3) -- (-18:1.7)
	(162:0.7) -- (162:1.3);

\draw[dotted]
	(-54:1.3) -- (-90:1.7) -- (234:1.3);

\node at (70:0.75) {\small $\gamma$}; 
\node at (32:0.8) {\small $\alpha$};
\node at (26:1.2) {\small $\beta$};
\node at (39:1.5) {\small $\delta$};
\node at (57:1.35) {\small $\epsilon$};

\node at (2:0.75) {\small $\beta$}; 
\node at (-40:0.8) {\small $\delta$};
\node at (-44:1.15) {\small $\epsilon$};
\node at (-18:1.45) {\small $\gamma$};
\node at (10:1.15) {\small $\alpha$};	

\node at (-90:1.45) {\small $\alpha$}; 

\node at (176:0.8) {\small $\gamma$}; 
\node at (220:0.8) {\small $\epsilon$};
\node at (224:1.15) {\small $\delta$};
\node at (198:1.45) {\small $\beta$};
\node at (170:1.15) {\small $\alpha$};

\node at (115:0.75) {\small $\beta$}; 
\node at (148:0.8) {\small $\alpha$};
\node at (154:1.2) {\small $\gamma$};
\node at (140:1.5) {\small $\epsilon$};
\node at (125:1.35) {\small $\delta$};


\begin{scope}[xshift=3.5cm]

\draw
	(90:0.7) -- (60:1.5)
	(90:0.7) -- (60:1.5) -- (40:1.7) -- (18:1.3)
	(120:1.5) -- (140:1.7) -- (162:1.3) -- (162:0.7)
	(-18:1.7) -- (-54:1.3) -- (-54:0.7);
	
\draw[line width=1.5]
	(162:0.7) -- (90:0.7)
	(18:0.7) -- (18:1.3);
	
\draw[dashed]
	(18:0.7) -- (90:0.7)
	(18:1.3) -- (-18:1.7)
	(90:0.7) -- (120:1.5);
	
\draw[dotted]
	(234:1.3) -- (234:0.7)
	(-54:1.3) -- (-90:1.7) -- (234:1.3) -- (198:1.7) -- (162:1.3);

\node at (70:0.75) {\small $\gamma$}; 
\node at (32:0.8) {\small $\alpha$};
\node at (26:1.2) {\small $\beta$};
\node at (39:1.5) {\small $\delta$};
\node at (57:1.35) {\small $\epsilon$};

\node at (2:0.75) {\small $\beta$}; 
\node at (-40:0.8) {\small $\delta$};
\node at (-44:1.15) {\small $\epsilon$};
\node at (-18:1.45) {\small $\gamma$};
\node at (10:1.15) {\small $\alpha$};	

\node at (112:0.75) {\small $\alpha$}; 
\node at (146:0.75) {\small $\beta$};
\node at (152:1.2) {\small $\delta$};
\node at (140:1.5) {\small $\epsilon$};
\node at (125:1.35) {\small $\gamma$};

\end{scope}


\begin{scope}[xshift=7cm]

\draw
	(90:0.7) -- (60:1.5)
	(60:1.5) -- (40:1.7) -- (18:1.3) -- (18:0.7)
	(120:1.5) -- (140:1.7) -- (162:1.3) -- (162:0.7);
	
\draw[line width=1.5]
	(162:0.7) -- (90:0.7)
	(90:0.7) -- (60:1.5);
	
\draw[dashed]
	(18:0.7) -- (90:0.7)
	(90:0.7) -- (120:1.5);

\draw[dotted]
	(-54:1.3) -- (-54:0.7)
	(234:1.3) -- (234:0.7)
	(18:1.3) -- (-18:1.7) -- (-54:1.3) -- (-90:1.7) -- (234:1.3) -- (198:1.7) -- (162:1.3);

\node at (70:0.75) {\small $\alpha$}; 
\node at (32:0.8) {\small $\gamma$};
\node at (26:1.2) {\small $\epsilon$};
\node at (39:1.5) {\small $\delta$};
\node at (55:1.3) {\small $\beta$};

\node at (112:0.75) {\small $\alpha$}; 
\node at (146:0.75) {\small $\beta$};
\node at (152:1.2) {\small $\delta$};
\node at (140:1.5) {\small $\epsilon$};
\node at (125:1.35) {\small $\gamma$};

\end{scope}


\foreach \a in {1,2,3}
\node[xshift=-3.5 cm+3.5*\a cm] at (-54:1.8) {1.\a};

\end{tikzpicture}
\caption{Case 1 for $a^3bc$.}
\label{3abcF1}
\end{figure}

In Case 1.1, by the vertex $\alpha_2\beta_3\gamma_1=\alpha\beta\gamma$, we know $H=\alpha_1\beta_6\gamma_2\cdots=\alpha\beta\gamma$ has degree $3$. Therefore $T_1$ is a $3^5$-tile. Moreover, by the edge length consideration, $T_1$ cannot be a $3^5$-tile in Cases 1.2 and 1.3. Therefore we may assume there is no $3^5$-tile after Case 1.1. By Lemma \ref{base_tile2}, we have $f\ge 24$ in Cases 1.2 and 1.3. 

\subsubsection*{Case 1.1}

The three $a$-edges of $T_4$ imply one of $\delta_1\epsilon_5\cdots$ and $\delta_3\epsilon_1\cdots$ is $\delta^2\epsilon$, and the other is $\delta\epsilon^2$. Then the angle sums of $\alpha\beta\gamma,\delta^2\epsilon,\delta\epsilon^2$ and the angle sum for pentagon imply $f=12$. By $f\ge 16$, we may dismiss the case.

\subsubsection*{Case 1.2}

By the edge length consideration, we have $H=\alpha^3\beta\gamma,\alpha^2\gamma^2,\alpha^2\gamma^2\delta,\alpha^2\gamma^2\epsilon$. By $\alpha_2\beta_3\gamma_1$, we have $H\ne \alpha^3\beta\gamma$. We will further show that $\alpha+\gamma>\pi$, so that $H\ne \alpha^2\gamma^2\cdots$.

Figure \ref{3abcF12} divides into further cases according to the edge $E_{45}$. If $E_{45}=a$, then $\delta_1\cdots$ and $\epsilon_1\cdots$ are combinations of $\delta,\epsilon$. If the combinations are different, then the angle sums of $\delta_1\cdots$ and $\epsilon_1\cdots$ imply $\delta=\epsilon=\frac{2}{3}\pi$. Combined with the angle sum of $\alpha\beta\gamma$ and the angle sum for pentagon, we get $f=12$, contradicting $f\ge 24$.  Therefore $\delta_1\cdots$ and $\epsilon_1\cdots$ have the same angle combination. This implies $\delta_1\cdots=\delta_1\delta_5\epsilon_4$ and $\epsilon_1\cdots=\delta_3\delta_4\epsilon_1$. Then we determine $T_4,T_5$, and get the first of Figure \ref{3abcF12}. If $E_{45}=b$ or $c$, then we may determine $T_4,T_5$ and get the second and third of Figure \ref{3abcF12}. The angle sums of the four degree $3$ vertices and the angle sum for pentagon imply
\begin{align*}
E_{45}=a
&\colon
\alpha+\gamma=(1+\tfrac{4}{f})\pi,\;
\beta=\delta=(1-\tfrac{4}{f})\pi,\;
\epsilon=\tfrac{8}{f}\pi; \\
E_{45}=b
&\colon
\alpha+\gamma=(\tfrac{3}{2}-\tfrac{2}{f})\pi,\;
\beta=(\tfrac{1}{2}+\tfrac{2}{f})\pi,\;
\delta=(1-\tfrac{4}{f})\pi,\;
\epsilon=\tfrac{8}{f}\pi; \\
E_{45}=c
&\colon
\alpha=\beta=(\tfrac{1}{2}+\tfrac{2}{f})\pi,\;
\gamma=\epsilon=(1-\tfrac{4}{f})\pi,\;
\delta=\tfrac{8}{f}\pi.
\end{align*}
In all cases, we have $\alpha+\gamma>\pi$.

\begin{figure}[htp]
\centering
\begin{tikzpicture}[>=latex]


\foreach \a in {0,1,2}
{
\begin{scope}[xshift=3.5*\a cm]

\fill (0,0.7) circle (0.1);

\foreach \x in {0,1,2} 
\draw[rotate=-72*\x]
	(18:0.7) -- (-54:0.7);

\draw
	(90:0.7) -- (60:1.5)
	(90:0.7) -- (60:1.5) -- (40:1.7) -- (18:1.3)
	(120:1.5) -- (140:1.7) -- (162:1.3) -- (162:0.7)
	(-18:1.7) -- (-54:1.3) -- (-54:0.7);
	
\draw[line width=1.5]
	(162:0.7) -- (90:0.7)
	(162:0.7) -- (90:0.7)
	(18:0.7) -- (18:1.3);
	
\draw[dashed]
	(18:0.7) -- (90:0.7)
	(18:0.7) -- (90:0.7)
	(18:1.3) -- (-18:1.7)
	(90:0.7) -- (120:1.5);

\node at (90:0.45) {\small $\alpha$}; 
\node at (162:0.45) {\small $\beta$};
\node at (18:0.45) {\small $\gamma$};
\node at (234:0.45) {\small $\delta$};
\node at (-54:0.45) {\small $\epsilon$};

\node at (70:0.75) {\small $\gamma$}; 
\node at (32:0.8) {\small $\alpha$};
\node at (26:1.2) {\small $\beta$};
\node at (39:1.5) {\small $\delta$};
\node at (57:1.35) {\small $\epsilon$};

\node at (2:0.75) {\small $\beta$}; 
\node at (-40:0.8) {\small $\delta$};
\node at (-44:1.15) {\small $\epsilon$};
\node at (-18:1.45) {\small $\gamma$};
\node at (10:1.15) {\small $\alpha$};	

\node at (112:0.75) {\small $\alpha$}; 
\node at (146:0.75) {\small $\beta$};
\node at (152:1.2) {\small $\delta$};
\node at (140:1.5) {\small $\epsilon$};
\node at (125:1.35) {\small $\gamma$};

\node[draw,shape=circle, inner sep=0.5] at (0,0) {\small $1$};
\node[draw,shape=circle, inner sep=0.5] at (45:1.05) {\small $2$};
\node[draw,shape=circle, inner sep=0.5] at (-18:1.05) {\small $3$};
\node[draw,shape=circle, inner sep=0.5] at (-90:1.05) {\small $4$};
\node[draw,shape=circle, inner sep=0.5] at (198:1.05) {\small $5$};
\node[draw,shape=circle, inner sep=0.5] at (135:1.05) {\small $6$};

\end{scope}
}


\draw
	(234:1.3) -- (234:0.7);
	
\draw[dashed]
	(-54:1.3) -- (-90:1.7) 
	(198:1.7) -- (162:1.3);

\draw[line width=1.5]
	(-90:1.7) -- (234:1.3) -- (198:1.7);

\node at (-66:0.8) {\small $\delta$}; 
\node at (-114:0.8) {\small $\epsilon$};
\node at (-116:1.15) {\small $\gamma$};
\node at (-90:1.45) {\small $\alpha$};
\node at (-64:1.15) {\small $\beta$};		

\node at (176:0.8) {\small $\epsilon$}; 
\node at (220:0.8) {\small $\delta$};
\node at (224:1.15) {\small $\beta$};
\node at (198:1.45) {\small $\alpha$};
\node at (172:1.15) {\small $\gamma$};
	


\begin{scope}[xshift=3.5cm]

\draw[line width=1.5]
	(234:1.3) -- (234:0.7);

\draw[dashed]
	(-90:1.7) -- (234:1.3) -- (198:1.7);

\draw
	(-54:1.3) -- (-90:1.7)
	(198:1.7) -- (162:1.3);

\node at (-66:0.8) {\small $\delta$}; 
\node at (-114:0.85) {\small $\beta$};
\node at (-116:1.15) {\small $\alpha$};
\node at (-90:1.45) {\small $\gamma$};
\node at (-62:1.15) {\small $\epsilon$};		

\node at (176:0.8) {\small $\delta$}; 
\node at (220:0.8) {\small $\beta$};
\node at (224:1.15) {\small $\alpha$};
\node at (198:1.45) {\small $\gamma$};
\node at (170:1.2) {\small $\epsilon$};


\end{scope}


\begin{scope}[xshift=7cm]

\draw[dashed]
	(234:1.3) -- (234:0.7);

\draw[line width=1.5]
	(-90:1.7) -- (234:1.3) -- (198:1.7);

\draw
	(-54:1.3) -- (-90:1.7)
	(198:1.7) -- (162:1.3);

\node at (-66:0.8) {\small $\epsilon$}; 
\node at (-114:0.85) {\small $\gamma$};
\node at (-116:1.15) {\small $\alpha$};
\node at (-90:1.45) {\small $\beta$};
\node at (-62:1.15) {\small $\delta$};		

\node at (176:0.8) {\small $\epsilon$}; 
\node at (220:0.8) {\small $\gamma$};
\node at (224:1.15) {\small $\alpha$};
\node at (198:1.45) {\small $\beta$};
\node at (170:1.2) {\small $\delta$};


\end{scope}

\end{tikzpicture}
\caption{Case 1.2 for $a^3bc$.}
\label{3abcF12}
\end{figure}

\subsubsection*{Case 1.3}

By the edge length consideration, we have $H=\alpha^4,\alpha^3\beta\gamma$. Moreover, since $b$ and $c$ of $T_4$ are adjacent, one of $E_{34}$ and $E_{45}$ is $a$. By the symmetry of exchanging $b$ and $c$, we may further assume that $E_{45}=a$. Then we consider three possibilities for $E_{34}$.

\subsubsection*{Subcase. $E_{34}=a$}

Figure \ref{3abcF13a} shows the case $E_{34}=a$. This implies $\delta_1\cdots$ and $\epsilon_1\cdots$ are combinations of $\delta,\epsilon$. 

If $\delta=\epsilon$, then the angle sum of $\delta_1\cdots$ implies $\delta=\epsilon=\frac{2}{3}\pi$. Then the angle sums of $\beta_1\beta_6\cdots$ and $\gamma_1\gamma_2\cdots$ imply $\beta=\gamma=\frac{2}{3}\pi$, contradicting Lemma \ref{geometry}. 

Therefore $\delta\ne \epsilon$. This implies $A_{3,12}=A_{5,16}$, $A_{3,14}=A_{5,14}$, and $\delta_1\cdots$ and $\epsilon_1\cdots$ have the same angle combination. This further implies $A_{4,13}=\delta,A_{4,15}=\epsilon$. Then we determine $T_4$ and get two pictures according to $A_{3,12}=A_{5,16}=\delta$ or $\epsilon$.

\begin{figure}[htp]
\centering
\begin{tikzpicture}[>=latex]


\foreach \a in {0,1}
{
\begin{scope}[xshift=3.5*\a cm]

\fill (0,0.7) circle (0.1);

\foreach \x in {0,1,2} 
\draw[rotate=-72*\x]
	(18:0.7) -- (-54:0.7);

\draw
	(90:0.7) -- (60:1.5)
	(60:1.5) -- (40:1.7) -- (18:1.3) -- (18:0.7)
	(120:1.5) -- (140:1.7) -- (162:1.3) -- (162:0.7)
	(234:0.7) -- (234:1.3)
	(-54:1.3) -- (-54:0.7);
	
\draw[line width=1.5]
	(162:0.7) -- (90:0.7)
	(90:0.7) -- (60:1.5)
	(-54:1.3) -- (-90:1.7);
	
\draw[dashed]
	(18:0.7) -- (90:0.7)
	(90:0.7) -- (120:1.5)
	(-90:1.7) -- (234:1.3);
	
\node at (90:0.45) {\small $\alpha$}; 
\node at (162:0.45) {\small $\beta$};
\node at (18:0.45) {\small $\gamma$};
\node at (234:0.45) {\small $\delta$};
\node at (-54:0.45) {\small $\epsilon$};

\node at (70:0.75) {\small $\alpha$}; 
\node at (32:0.8) {\small $\gamma$};
\node at (26:1.2) {\small $\epsilon$};
\node at (39:1.5) {\small $\delta$};
\node at (55:1.3) {\small $\beta$};

\node at (-68:0.8) {\small $\delta$}; 
\node at (-116:0.8) {\small $\epsilon$};
\node at (-118:1.15) {\small $\gamma$};
\node at (-90:1.45) {\small $\alpha$};
\node at (-64:1.2) {\small $\beta$};

\node at (112:0.75) {\small $\alpha$}; 
\node at (146:0.75) {\small $\beta$};
\node at (152:1.2) {\small $\delta$};
\node at (140:1.5) {\small $\epsilon$};
\node at (125:1.35) {\small $\gamma$};

\node[draw,shape=circle, inner sep=0.5] at (0,0) {\small $1$};
\node[draw,shape=circle, inner sep=0.5] at (45:1.05) {\small $2$};
\node[draw,shape=circle, inner sep=0.5] at (-18:1.05) {\small $3$};
\node[draw,shape=circle, inner sep=0.5] at (-90:1.05) {\small $4$};
\node[draw,shape=circle, inner sep=0.5] at (198:1.05) {\small $5$};
\node[draw,shape=circle, inner sep=0.5] at (135:1.05) {\small $6$};

\end{scope}
}


\draw[line width=1.5]
	(18:1.3) -- (-18:1.7)
	(198:1.7) -- (162:1.3);

\draw[dashed]
	(-18:1.7) -- (-54:1.3)
	(234:1.3) -- (198:1.7);
		
\node at (4:0.75) {\small $\delta$}; 
\node at (-40:0.75) {\small $\epsilon$};
\node at (-46:1.15) {\small $\gamma$};
\node at (-18:1.45) {\small $\alpha$};
\node at (8:1.1) {\small $\beta$};	

\node at (176:0.75) {\small $\delta$}; 
\node at (220:0.75) {\small $\epsilon$};
\node at (226:1.15) {\small $\gamma$};
\node at (198:1.45) {\small $\alpha$};
\node at (172:1.2) {\small $\beta$};


\begin{scope}[xshift=3.5cm]

\draw[dashed]
	(18:1.3) -- (-18:1.7)
	(198:1.7) -- (162:1.3);

\draw[line width=1.5]
	(-18:1.7) -- (-54:1.3)
	(234:1.3) -- (198:1.7);
		
\node at (6:0.75) {\small $\epsilon$}; 
\node at (-40:0.75) {\small $\delta$};
\node at (-44:1.15) {\small $\beta$};
\node at (-18:1.45) {\small $\alpha$};
\node at (8:1.15) {\small $\gamma$};	

\node at (174:0.75) {\small $\epsilon$}; 
\node at (220:0.75) {\small $\delta$};
\node at (226:1.15) {\small $\beta$};
\node at (198:1.45) {\small $\alpha$};
\node at (170:1.15) {\small $\gamma$};

\end{scope}

\end{tikzpicture}
\caption{Case 1.3 for $a^3bc$, $E_{34}=a$, $E_{45}=a$.}
\label{3abcF13a}
\end{figure}

In the first of Figure \ref{3abcF13a}, the angle sums of $\beta^2\delta,\gamma^2\delta,\delta\epsilon^2$ and the angle sum for pentagon imply 
\[
\alpha=\tfrac{4}{f}\pi+\tfrac{1}{2}\delta,\;
\beta=\gamma=\epsilon=\pi-\tfrac{1}{2}\delta.
\]
By $2\alpha+\beta+\gamma>2\pi$, we get $H\ne\alpha^3\beta\gamma$. Then the angle sum of $H=\alpha^4$ further implies 
\[
\alpha=\tfrac{1}{2}\pi,\;
\beta=\gamma=\epsilon=(\tfrac{1}{2}+\tfrac{4}{f})\pi,\;
\delta=(1-\tfrac{8}{f})\pi.
\]
By the edge length consideration, we have $\gamma_4\gamma_5\cdots=\theta\dash\gamma\thin\gamma\dash\rho\cdots$, where $\theta,\rho=\alpha,\gamma$ are two angles. By $\gamma>\alpha=\frac{1}{2}\pi$, the angle sum of the vertex is $>2\pi$, a contradiction. 

By the similar argument, the second of Figure \ref{3abcF13a} leads to a similar contradiction at $\beta_3\beta_4\cdots$. 

\subsubsection*{Subcase. $E_{34}=b$}

Figure \ref{3abcF13b} shows the case $E_{34}=b$. This determines $T_3,T_4$. Then we get two pictures according to the two possible arrangements of $T_5$.

\begin{figure}[htp]
\centering
\begin{tikzpicture}[>=latex]


\foreach \a in {0,1}
{
\begin{scope}[xshift=3.5*\a cm]

\fill (0,0.7) circle (0.1);

\foreach \x in {0,1,2} 
\draw[rotate=-72*\x]
	(18:0.7) -- (-54:0.7);

\draw
	(90:0.7) -- (60:1.5)
	(60:1.5) -- (40:1.7) -- (18:1.3) -- (18:0.7)
	(120:1.5) -- (140:1.7) -- (162:1.3) -- (162:0.7)
	(234:0.7) -- (234:1.3) -- (-90:1.7)
	(18:1.3) -- (-18:1.7);
	
\draw[line width=1.5]
	(162:0.7) -- (90:0.7)
	(-54:1.3) -- (-54:0.7)
	(90:0.7) -- (60:1.5);
	
\draw[dashed]
	(18:0.7) -- (90:0.7)
	(-18:1.7) -- (-54:1.3) -- (-90:1.7)
	(90:0.7) -- (120:1.5);
	
\node at (90:0.45) {\small $\alpha$}; 
\node at (162:0.45) {\small $\beta$};
\node at (18:0.45) {\small $\gamma$};
\node at (234:0.45) {\small $\delta$};
\node at (-54:0.45) {\small $\epsilon$};

\node at (70:0.75) {\small $\alpha$}; 
\node at (32:0.8) {\small $\gamma$};
\node at (26:1.2) {\small $\epsilon$};
\node at (39:1.5) {\small $\delta$};
\node at (55:1.3) {\small $\beta$};

\node at (4:0.75) {\small $\delta$}; 
\node at (-40:0.8) {\small $\beta$};
\node at (-44:1.15) {\small $\alpha$};
\node at (-18:1.45) {\small $\gamma$};
\node at (10:1.15) {\small $\epsilon$};	

\node at (-68:0.85) {\small $\beta$}; 
\node at (-116:0.8) {\small $\delta$};
\node at (-118:1.15) {\small $\epsilon$};
\node at (-90:1.45) {\small $\gamma$};
\node at (-64:1.15) {\small $\alpha$};	

\node at (112:0.75) {\small $\alpha$}; 
\node at (146:0.75) {\small $\beta$};
\node at (152:1.2) {\small $\delta$};
\node at (140:1.5) {\small $\epsilon$};
\node at (126:1.35) {\small $\gamma$};

\node[draw,shape=circle, inner sep=0.5] at (0,0) {\small $1$};
\node[draw,shape=circle, inner sep=0.5] at (45:1.05) {\small $2$};
\node[draw,shape=circle, inner sep=0.5] at (-18:1.05) {\small $3$};
\node[draw,shape=circle, inner sep=0.5] at (-90:1.05) {\small $4$};
\node[draw,shape=circle, inner sep=0.5] at (198:1.05) {\small $5$};
\node[draw,shape=circle, inner sep=0.5] at (135:1.05) {\small $6$};

\end{scope}
}


\draw[line width=1.5]
	(162:1.3) -- (198:1.7);

\draw[dashed]
	(234:1.3) -- (198:1.7);

\node at (176:0.75) {\small $\delta$}; 
\node at (220:0.75) {\small $\epsilon$};
\node at (226:1.15) {\small $\gamma$};
\node at (198:1.45) {\small $\alpha$};
\node at (172:1.2) {\small $\beta$};


\begin{scope}[xshift=3.5cm]

\draw[dashed]
	(198:1.7) -- (162:1.3);

\draw[line width=1.5]
	(234:1.3) -- (198:1.7);

\node at (174:0.75) {\small $\epsilon$}; 
\node at (220:0.75) {\small $\delta$};
\node at (226:1.15) {\small $\beta$};
\node at (198:1.45) {\small $\alpha$};
\node at (170:1.15) {\small $\gamma$};		

\end{scope}

\end{tikzpicture}
\caption{Case 1.3 for $a^3bc$, $E_{34}=b$, $E_{45}=a$.}
\label{3abcF13b}
\end{figure}

In the first of Figure \ref{3abcF13b}, the angle sums of $\beta^2\delta,\beta^2\epsilon,\gamma^2\delta,\delta^2\epsilon$ imply $\beta=\gamma=\delta=\epsilon$, contradicting Lemma \ref{geometry}. 

In the second of Figure \ref{3abcF13b}, if $H=\alpha^3\beta\gamma$, then the angle sums of $\beta^2\epsilon,\gamma^2\delta,\delta^3,H$ and the angle sum for pentagon imply 
\[
\alpha=(\tfrac{1}{4}+\tfrac{1}{f})\pi,\;
\beta=(\tfrac{7}{12}-\tfrac{3}{f})\pi,\;
\gamma=\delta=\tfrac{2}{3}\pi,\;
\epsilon=(\tfrac{5}{6}+\tfrac{6}{f})\pi.
\]
Then by $R(\thin\epsilon_2\thin\epsilon_3\thin\cdots)=(\tfrac{1}{3}-\tfrac{12}{f})\pi<\beta,\gamma,\delta,\epsilon$, the remainder has only $\alpha$. Since $\alpha$ is a $bc$-angle, we get a contradiction.

We will discuss the remaining case $H=\alpha^4$ (for the second of Figure \ref{3abcF13b}) in Section \ref{pent_ddivision}.

\subsubsection*{Subcase. $E_{34}=c$}

Figure \ref{3abcF13c} shows the case $E_{34}=c$. We get two pictures according to the two possible arrangements  of $T_5$.

\begin{figure}[htp]
\centering
\begin{tikzpicture}[>=latex]


\foreach \a in {0,1}
{
\begin{scope}[xshift=3.5*\a cm]

\fill (0,0.7) circle (0.1);

\foreach \x in {0,1,2} 
\draw[rotate=-72*\x]
	(18:0.7) -- (-54:0.7);

\draw
	(90:0.7) -- (60:1.5)
	(60:1.5) -- (40:1.7) -- (18:1.3) -- (18:0.7)
	(120:1.5) -- (140:1.7) -- (162:1.3) -- (162:0.7)
	(234:0.7) -- (234:1.3) -- (-90:1.7)
	(18:1.3) -- (-18:1.7);
	
\draw[line width=1.5]
	(162:0.7) -- (90:0.7)
	(90:0.7) -- (60:1.5)
	(-18:1.7) -- (-54:1.3) -- (-90:1.7);
	
\draw[dashed]
	(18:0.7) -- (90:0.7)
	(-54:1.3) -- (-54:0.7)
	(90:0.7) -- (120:1.5);
	
\node at (90:0.45) {\small $\alpha$}; 
\node at (162:0.45) {\small $\beta$};
\node at (18:0.45) {\small $\gamma$};
\node at (234:0.45) {\small $\delta$};
\node at (-54:0.45) {\small $\epsilon$};

\node at (70:0.75) {\small $\alpha$}; 
\node at (32:0.8) {\small $\gamma$};
\node at (26:1.2) {\small $\epsilon$};
\node at (39:1.5) {\small $\delta$};
\node at (55:1.3) {\small $\beta$};

\node at (4:0.75) {\small $\epsilon$}; 
\node at (-40:0.8) {\small $\gamma$};
\node at (-44:1.15) {\small $\alpha$};
\node at (-18:1.45) {\small $\beta$};
\node at (10:1.15) {\small $\delta$};	

\node at (-68:0.8) {\small $\gamma$}; 
\node at (-114:0.8) {\small $\epsilon$};
\node at (-116:1.15) {\small $\delta$};
\node at (-90:1.45) {\small $\beta$};
\node at (-64:1.15) {\small $\alpha$};	

\node at (112:0.75) {\small $\alpha$}; 
\node at (146:0.75) {\small $\beta$};
\node at (152:1.2) {\small $\delta$};
\node at (140:1.5) {\small $\epsilon$};
\node at (125:1.35) {\small $\gamma$};

\node[draw,shape=circle, inner sep=0.5] at (0,0) {\small $1$};
\node[draw,shape=circle, inner sep=0.5] at (45:1.05) {\small $2$};
\node[draw,shape=circle, inner sep=0.5] at (-18:1.05) {\small $3$};
\node[draw,shape=circle, inner sep=0.5] at (-90:1.05) {\small $4$};
\node[draw,shape=circle, inner sep=0.5] at (198:1.05) {\small $5$};
\node[draw,shape=circle, inner sep=0.5] at (135:1.05) {\small $6$};

\end{scope}
}


\draw
	(-54:2.1) -- (-70:2.4) -- (-90:2.4) -- (-90:1.7)
	(234:1.3) -- (234:2.1) -- (223:2.6) -- (210:2.3);

\draw[line width=1.5]
	(162:1.3) -- (198:1.7) -- (210:2.3);

\draw[dashed]
	(234:1.3) -- (198:1.7)
	(-54:1.3) -- (-54:2.1)
	(198:1.7) -- (190:2);

\node at (176:0.75) {\small $\delta$}; 
\node at (220:0.75) {\small $\epsilon$};
\node at (226:1.15) {\small $\gamma$};
\node at (198:1.45) {\small $\alpha$};
\node at (172:1.2) {\small $\beta$};

\node at (-61:1.5) {\small $\alpha$}; 
\node at (-85:1.8) {\small $\beta$};
\node at (-60:2) {\small $\gamma$};
\node at (-86:2.2) {\small $\delta$};
\node at (-70:2.2) {\small $\epsilon$};

\node at (229:1.5) {\small $\gamma$}; 
\node at (206:1.7) {\small $\alpha$};
\node at (230:2) {\small $\epsilon$};
\node at (213:2.15) {\small $\beta$};
\node at (221:2.35) {\small $\delta$};

\node at (-47:1.5) {\small $\alpha$}; 
\node at (190:1.75) {\small $\alpha$};
\node at (197:1.95) {\small $\alpha$};
\node at (-95:1.8) {\small $\delta$};

\node[draw,shape=circle, inner sep=0.5] at (-70:1.85) {\small $7$};
\node[draw,shape=circle, inner sep=0.5] at (220:1.85) {\small $8$};


\begin{scope}[xshift=3.5cm]

\draw[dashed]
	(198:1.7) -- (162:1.3);

\draw[line width=1.5]
	(234:1.3) -- (198:1.7);

\node at (174:0.75) {\small $\epsilon$}; 
\node at (220:0.75) {\small $\delta$};
\node at (226:1.15) {\small $\beta$};
\node at (198:1.45) {\small $\alpha$};
\node at (170:1.15) {\small $\gamma$};		

\end{scope}

\end{tikzpicture}
\caption{Case 1.3 for $a^3bc$, $E_{34}=c$, $E_{45}=a$.}
\label{3abcF13c}
\end{figure}

In the first of Figure \ref{3abcF13c}, if $H=\alpha^3\beta\gamma$, then the angle sums of $\beta^2\delta,\gamma^2\epsilon,\delta\epsilon^2$, $H$ and the angle sum for pentagon imply 
\[
\alpha=(\tfrac{1}{4}+\tfrac{1}{f})\pi,\;
\beta=\epsilon=(\tfrac{1}{2}-\tfrac{6}{f})\pi,\;
\gamma=(\tfrac{3}{4}+\tfrac{3}{f})\pi,\;
\delta=(1+\tfrac{12}{f})\pi.
\]
By $\dash\gamma_5\thin\delta_4\thin\cdots=\thick\alpha\dash\gamma\thin\delta\thin\cdots$ or $\thin\gamma\dash\gamma\thin\delta\thin\cdots$, and $2\gamma+\delta>\alpha+\gamma+\delta>2\pi$, we get a contradiction. 

Therefore $H=\alpha^4$. The angle sums of $\beta^2\delta,\gamma^2\epsilon,\delta\epsilon^2,H$ and the angle sum for pentagon imply 
\[
\alpha=\tfrac{1}{2}\pi,\;
\beta=\epsilon=(1-\tfrac{8}{f})\pi,\;
\gamma=(\tfrac{1}{2}+\tfrac{4}{f})\pi,\;
\delta=\tfrac{16}{f}\pi.
\]
Since $f=24$ implies $\beta=\gamma=\delta=\epsilon$, contradicting Lemma \ref{geometry}, we have $f>24$.

If $\thick\alpha\dash$ has no more $\alpha$ on either side, then $\thick\alpha\dash\cdots=\beta\thick\alpha\dash\gamma\cdots$. By $0\ne R(\beta\thick\alpha\dash\gamma\cdots)=\tfrac{4}{f}\pi<$ all angles, this is a contradiction. Therefore $\alpha\cdots=\dash\alpha\thick\alpha\dash\cdots,\thick\alpha\dash\alpha\thick\cdots$. We further have $\dash\alpha\thick\alpha\dash\cdots=\theta\dash\alpha\thick\alpha\dash\rho\cdots$, where $\theta,\rho=\alpha,\gamma$ are two angles, and $\thick\alpha\dash\alpha\thick\cdots=\theta\thick\alpha\dash\alpha\thick\rho\cdots$, where $\theta,\rho=\alpha,\beta$ are two angles. Then by $\alpha=\tfrac{1}{2}\pi<\beta,\gamma$, we conclude $\alpha\cdots=\alpha^4$. By $\alpha_3\alpha_4\cdots=\alpha_5\cdots=\alpha^4$, we determine $T_7,T_8$. By $\alpha\cdots=\alpha^4$ and $R(\beta_4\beta_7\cdots)=\delta<\beta,\gamma,\epsilon$, we get $\beta_4\beta_7\cdots=\beta^2\delta$. By $f>24$, we have $\beta\ne\gamma$. Then $\beta^2\delta$ implies $\gamma^2\delta$ is not a vertex. Therefore $\gamma_5\gamma_8\delta_4\cdots=\gamma^2\delta\theta\cdots$, where $\theta=\beta,\epsilon$ is adjacent to $\delta$. By $2\gamma+\delta+\theta>2\pi$, we get a contradiction. 

In the second of Figure \ref{3abcF13c}, the angle sums of $\beta^2\epsilon,\gamma^2\epsilon,\delta^2\epsilon$ and the angle sum for pentagon imply 
\[
\alpha=\tfrac{4}{f}\pi+\tfrac{1}{2}\epsilon,\;
\beta=\gamma=\delta=\pi-\tfrac{1}{2}\epsilon.
\]
By $2\alpha+\beta+\gamma>2\pi$, we have $H\ne \alpha^3\beta\gamma$. Therefore $H=\alpha^4$, and the angle sum of $H$ further implies
\[
\alpha=\tfrac{1}{2}\pi,\;
\beta=\gamma=\delta=(\tfrac{1}{2}+\tfrac{4}{f})\pi,\;
\epsilon=(1-\tfrac{8}{f})\pi.
\]
By the edge length consideration, we have $\thick\beta_5\thin\delta_4\thin\cdots=\dash\alpha\thick\beta\thin\delta\thin\cdots$ or $\thin\beta\thick\beta\thin\delta\thin\cdots$. By $f>24$, we know both $R(\dash\alpha\thick\beta\thin\delta\thin\cdots)=(\frac{1}{2}-\frac{8}{f})\pi$ and $R(\thin\beta\thick\beta\thin\delta\thin\cdots)=(\frac{1}{2}-\frac{12}{f})\pi$ are nonzero and strictly less than all the angles. Therefore $\dash\alpha\thick\beta\thin\delta\thin\cdots$ or $\thin\beta\thick\beta\thin\delta\thin\cdots$ are not vertices. We get a contradiction.

\subsection{Case 2$(a^3bc)$}
\label{3abcCase2}

Let the second of Figure \ref{3abc_pentagon} be the center tile $T_1$ in the partial neighbourhood in the first of Figure \ref{3abcF3}. Since $E_{23}$ is adjacent to $E_{12}=b$ and $E_{13}=c$, we get $E_{23}=a$ and determine $T_2,T_3$. The angle sum of $\alpha_1\beta_2\gamma_3$ implies $H=\dash\alpha\thick\beta\thin\cdots$ has no $\gamma$. Therefore $H=\thick\alpha\dash\alpha\thick\beta\thin\cdots=\alpha^2\beta^2,\alpha^2\beta^2\delta,\alpha^2\beta^2\epsilon$. The angle sum of $H$ implies $\alpha+\beta\le\pi$. Then the angle sum of $\alpha\beta\gamma$ implies $\gamma\ge \pi$. Therefore $\gamma^2\cdots$ is not a vertex. On the other hand, the AAD $H=\thick^{\beta}\alpha^{\gamma}\dash^{\gamma}\alpha^{\beta}\thick\cdots$ implies $\gamma^2\cdots$ is a vertex. We get a contradiction.

\begin{figure}[htp]
\centering
\begin{tikzpicture}[>=latex]


\begin{scope}[xshift=-3.5cm]

\fill (0,0.7) circle (0.1);

\foreach \x in {1,2,3} 
\draw[rotate=72*\x]
	(18:0.7) -- (90:0.7);

\foreach \x in {0,1,2}
\draw[dotted,rotate=-72*\x]
	(18:0.7) -- (18:1.3) -- (-18:1.7) -- (-54:1.3) -- (-54:0.7);

\draw
	(18:0.7) -- (18:1.3) -- (40:1.7) -- (60:1.5)
	(18:1.3) -- (-18:1.7) -- (-54:1.3) ;

\draw[dotted]
	(90:0.7) -- (120:1.5) -- (140:1.7) -- (162:1.3);

\draw[line width=1.5]
	(90:0.7) -- (18:0.7)
	(-54:0.7) -- (-54:1.3);
	
\draw[dashed]
	(18:0.7) -- (-54:0.7)
	(90:0.7) -- (60:1.5);
	
\node at (90:0.4) {\small $\beta$}; 
\node at (162:0.45) {\small $\delta$};
\node at (18:0.45) {\small $\alpha$};
\node at (234:0.45) {\small $\epsilon$};
\node at (-54:0.45) {\small $\gamma$};

\node at (70:0.78) {\small $\alpha$}; 
\node at (32:0.8) {\small $\beta$};
\node at (28:1.2) {\small $\delta$};
\node at (42:1.5) {\small $\epsilon$};
\node at (57:1.35) {\small $\gamma$};

\node at (4:0.8) {\small $\gamma$}; 
\node at (-40:0.8) {\small $\alpha$};
\node at (-44:1.15) {\small $\beta$};
\node at (-18:1.45) {\small $\delta$};
\node at (10:1.15) {\small $\epsilon$};

\node[draw,shape=circle, inner sep=0.5] at (0,0) {\small $1$};
\node[draw,shape=circle, inner sep=0.5] at (45:1.15) {\small $2$};\node[draw,shape=circle, inner sep=0.5] at (-18:1) {\small $3$};

\end{scope}


\fill (0,0.7) circle (0.1);

\foreach \x in {0,1,2} 
\draw[rotate=72*\x]
	(18:0.7) -- (90:0.7);

\draw
	(90:0.7) -- (60:1.5)
	(90:0.7) -- (60:1.5) -- (40:1.7)
	(18:1.3) -- (-18:1.7) -- (-54:1.3) -- (-90:1.7) -- (234:1.3)
	(-54:0.7) -- (-54:1.3)
	(162:0.7) -- (162:1.3) -- (198:1.7);
	
\draw[line width=1.5]
	(-54:0.7) -- (18:0.7)
	(234:0.7) -- (234:1.3)
	(40:1.7) -- (18:1.3);
	
\draw[dashed]
	(-54:0.7) -- (234:0.7)
	(18:0.7) -- (18:1.3)
	(198:1.7) -- (234:1.3);

\draw[dotted]
	(90:0.7) -- (120:1.5) -- (140:1.7) -- (162:1.3);
	
\node at (90:0.45) {\small $\delta$}; 
\node at (162:0.45) {\small $\epsilon$};
\node at (18:0.45) {\small $\beta$};
\node at (234:0.45) {\small $\gamma$};
\node at (-54:0.45) {\small $\alpha$};

\node at (70:0.78) {\small $\epsilon$}; 
\node at (32:0.8) {\small $\gamma$};
\node at (28:1.2) {\small $\alpha$};
\node at (40:1.45) {\small $\beta$};
\node at (57:1.3) {\small $\delta$};

\node at (4:0.8) {\small $\alpha$}; 
\node at (-40:0.8) {\small $\beta$};
\node at (-44:1.15) {\small $\delta$};
\node at (-18:1.45) {\small $\epsilon$};
\node at (10:1.15) {\small $\gamma$};	

\node at (-68:0.8) {\small $\gamma$}; 
\node at (-112:0.8) {\small $\alpha$};
\node at (-116:1.15) {\small $\beta$};
\node at (-90:1.45) {\small $\delta$};
\node at (-62:1.15) {\small $\epsilon$};	

\node at (176:0.8) {\small $\delta$}; 
\node at (220:0.8) {\small $\beta$};
\node at (224:1.15) {\small $\alpha$};
\node at (198:1.45) {\small $\gamma$};
\node at (170:1.15) {\small $\epsilon$};

\node[draw,shape=circle, inner sep=0.5] at (0,0) {\small $1$};
\node[draw,shape=circle, inner sep=0.5] at (45:1.05) {\small $2$};
\node[draw,shape=circle, inner sep=0.5] at (-18:1.05) {\small $3$};
\node[draw,shape=circle, inner sep=0.5] at (-90:1.05) {\small $4$};
\node[draw,shape=circle, inner sep=0.5] at (198:1.05) {\small $5$};
\node[draw,shape=circle, inner sep=0.5] at (135:1.05) {\small $6$};

\end{tikzpicture}
\caption{Cases 2 and 3 for $a^3bc$.}
\label{3abcF3}
\end{figure}

\subsection{Case 3$(a^3bc)$}
\label{3abcCase3}

Let the third of Figure \ref{3abc_pentagon} be the center tile $T_1$ in the partial neighbourhood in the second of Figure \ref{3abcF3}. Since $E_{34}$ is adjacent to $E_{13}=b$ and $E_{14}=c$, we get $E_{34}=a$ and determine $T_3,T_4$. This further determines $T_2,T_5$. The angle sum of $\alpha\beta\gamma$ and the angle sum for pentagon imply
\[
\alpha+\beta+\gamma=2\pi,\;
\delta+\epsilon=(1+\tfrac{4}{f})\pi.
\]

Suppose $H=\delta\epsilon\cdots$ has $\beta\thick\beta$. Then the angle sum of $H$ implies 
\[
\beta
\le \pi-\tfrac{1}{2}(\delta+\epsilon)
=(\tfrac{1}{2}-\tfrac{2}{f})\pi,\;
\alpha+\gamma\ge 2\pi-\beta
\ge (\tfrac{3}{2}+\tfrac{2}{f})\pi.
\]
Moreover, the AAD $\thin^{\delta}\beta^{\alpha}\thick^{\alpha}\beta^{\delta}\thin$ of $\thin\beta\thick\beta\thin$ implies a vertex $\dash\alpha\thick\alpha\dash\cdots=\theta\dash\alpha\thick\alpha\dash\rho\cdots$, where $\theta,\rho=\alpha,\gamma$ are two angles. By $\alpha+\gamma>\pi$, one of $\theta,\rho$ is $\alpha$. Then $\dash\alpha\thick\alpha\dash\cdots=\thick\alpha\dash\alpha\thick\alpha\dash\alpha\cdots$ or $\thick\alpha\dash\alpha\thick\alpha\dash\gamma\cdots$. The angle sum of the vertex implies $4\alpha\le 2\pi$ or $3\alpha+\gamma\le 2\pi$. Combined with $\alpha+\gamma
\ge (\tfrac{3}{2}+\tfrac{2}{f})\pi$, we always get $\gamma>\pi$. On the other hand, the AAD $\thick^{\beta}\alpha^{\gamma}\dash^{\gamma}\alpha^{\beta}\thick$ implies $\gamma^2\cdots$ is a vertex. We get a contradiction. 

Therefore $H=\delta\thin\epsilon\cdots$ has no $\beta\thick\beta$. By the same reason, $H$ has no $\gamma\dash\gamma$. By the edge length consideration, this implies that, if $H=\delta\thin\epsilon\cdots$ has any one of $\alpha,\beta,\gamma$, then the remainder of $H$ is $\thin\beta\thick\alpha\dash\gamma\thin$. This contradicts the vertex $\alpha\beta\gamma$, and also contradicts the angle sum for pentagon. 

Therefore $H$ has only $\delta,\epsilon$. By $\delta+\epsilon>\pi$, we conclude $H=\delta\epsilon\cdots=\delta^3\epsilon,\delta\epsilon^3, \delta^4\epsilon,\delta\epsilon^4$. We will continue studying the case in Section \ref{pent_ddivision}.

\subsection{Double Pentagonal Subdivision}
\label{pent_ddivision}

After Sections \ref{3abcCase1}, \ref{3abcCase2}, \ref{3abcCase3}, the only remaining cases for the edge combination $a^3bc$ are the following.
\begin{description}
\item Case 1.3, $H=\alpha^4$, $E_{34}=b$, $E_{45}=a$: We have the second of Figure \ref{3abcF13b}. The angle sums of $\beta^2\epsilon,\gamma^2\delta,\delta^3,\alpha^4$ and the angle sum for pentagon imply
\[
\alpha=\tfrac{1}{2}\pi,\;
\beta=(\tfrac{5}{6}-\tfrac{4}{f})\pi,\;
\gamma=\delta=\tfrac{2}{3}\pi,\;
\epsilon=(\tfrac{1}{3}+\tfrac{8}{f})\pi.
\]
\item Case 3, $H=\delta^3\epsilon,\delta\epsilon^3, \delta^4\epsilon,\delta\epsilon^4$: We have the second of Figure \ref{3abcF3}. The angle sums of $\alpha\beta\gamma$ and the angle sum for pentagon imply
\[
\alpha+\beta+\gamma=2\pi,\;
\delta+\epsilon=(1+\tfrac{4}{f})\pi.
\]
\end{description}
We also recall that $f\ge 24$. 

\subsubsection*{Case 1.3. $H=\alpha^4$}

Since $f=24$ implies $\beta=\gamma=\delta=\epsilon$, contradicting Lemma \ref{geometry}, we have $f>24$. 

In Section \ref{avc}, we obtained all the possible vertex combinations. We need to consider $f=48,72,120$, and the rest.

For $f=48$, the AVC is given by \eqref{avc13}
\[
\text{AVC}=\{\alpha\beta^2,\beta^2\epsilon,\gamma^2\delta,\delta^3,\alpha^4,\epsilon^4\}.
\]
We construct the tiling based on the AVC and the fact that there is a vertex $\delta^3$, such that the tiles around the vertex are arranged as $T_1,T_4,T_5$ in the second of Figure \ref{3abcF13b}. Figure \ref{3abcF13b_tiling} shows such a vertex $\delta^3$, with three tiles $T_1,T_1',T_1''$ around the vertex as assumed. We will use $T_n',T_n''$ to denote the two rotations of $T_n$, and the subsequent conclusions remain valid after rotations. 

By the AVC, we have $\beta_1\epsilon_{1'}\cdots=\beta^2\epsilon$. This determines $T_2$ (and its other two rotations). Then $\alpha_1\alpha_2\cdots=\alpha^4$ (we will omit mentioning ``by the AVC'') determines $T_3,T_4$ (we will omit mentioning ``and other two rotations''). Then $\gamma_2\gamma_3\cdots=\gamma^2\delta$ and $\epsilon_2\epsilon_{4'}\cdots=\epsilon^4$ determine $T_5$. Then $\beta_5\epsilon_3\cdots=\beta^2\epsilon$ determines $T_6$. Then $\delta_3\delta_6\cdots=\delta_3\delta_6\delta_7$. By $\delta_7$, we know $\beta_3\beta_4\cdots=\beta_3\beta_4\epsilon_7$. Then $\delta_7,\epsilon_7$ determine $T_7$. We find that $T_3,T_6,T_7$ around a vertex $\delta^3$ just like $T_1,T_1',T_1''$ around the original vertex $\delta^3$. The argument for the tiling can therefore be repeated by starting from the new $\delta^3$.

\begin{figure}[htp]
\centering
\begin{tikzpicture}[rotate=35, >=latex]


\foreach \a in {0,1,2}
{
\begin{scope}[rotate=120*\a]

\draw
	(0,0) -- (-30:0.7) -- (5:1.1)  
	(70:1.7) -- (110:1.7) -- (125:1.1) 
	(70:1.7) -- (65:2.3) -- (45:2.4) -- (35:1.8) -- (10:2.2) -- (-10:1.7)
	(110:1.7) -- (100:2.3)
	(70:3.2) -- (45:3.1) -- (45:2.4);

\draw[line width=1.5]
	(90:0.7) -- (55:1.1) -- (35:1.8)
	(80:2.6) -- (65:2.3)
	(45:3.1) -- (30:3);

\draw[dashed]
	(5:1.1) -- (55:1.1) -- (70:1.7)
	(100:2.3) -- (80:2.6) -- (70:3.2)
	(10:2.2) -- (30:3);

\node at (30:0.2) {\small $\delta$}; 
\node at (70:0.6) {\small $\beta$};
\node at (-15:0.6) {\small $\epsilon$};
\node at (52:0.9) {\small $\alpha$};
\node at (8:0.9) {\small $\gamma$};

\node at (90:0.9) {\small $\beta$}; 
\node at (67:1.15) {\small $\alpha$};
\node at (113:1.15) {\small $\delta$};
\node at (75:1.5) {\small $\gamma$};
\node at (107:1.5) {\small $\epsilon$};

\node at (55:1.4) {\small $\alpha$}; 
\node at (64:1.75) {\small $\gamma$};
\node at (43:1.8) {\small $\beta$};
\node at (48:2.2) {\small $\delta$};
\node at (62:2.15) {\small $\epsilon$};

\node at (42:1.2) {\small $\alpha$}; 
\node at (10:1.25) {\small $\gamma$};
\node at (30:1.6) {\small $\beta$};
\node at (11:1.95) {\small $\delta$};
\node at (-4:1.65) {\small $\epsilon$};

\node at (74:1.85) {\small $\delta$}; 
\node at (103:1.8) {\small $\epsilon$};
\node at (80:2.4) {\small $\alpha$};
\node at (98:2.1) {\small $\gamma$};
\node at (71:2.15) {\small $\beta$};

\node at (64:2.5) {\small $\beta$}; 
\node at (74:2.65) {\small $\alpha$};
\node at (49:2.55) {\small $\delta$};
\node at (69:3) {\small $\gamma$};
\node at (48:2.9) {\small $\epsilon$};

\node at (41:2.4) {\small $\delta$}; 
\node at (33:1.95) {\small $\epsilon$};
\node at (42:2.85) {\small $\beta$};
\node at (17:2.2) {\small $\gamma$};
\node at (32:2.8) {\small $\alpha$};

\end{scope}
}

\node[draw,shape=circle, inner sep=0.5] at (30:0.6) {\small $1$};
\node[draw,shape=circle, inner sep=0] at (150:0.6) {\small $1'$};
\node[draw,shape=circle, inner sep=-0.2] at (-94:0.65) {\small $1''$};
\node[draw,shape=circle, inner sep=0.5] at (90:1.3) {\small $2$};
\node[draw,shape=circle, inner sep=0.5] at (53:1.8) {\small $3$};
\node[draw,shape=circle, inner sep=0.5] at (17:1.55) {\small $4$};
\node[draw,shape=circle, inner sep=0] at (135:1.55) {\small $4'$};
\node[draw,shape=circle, inner sep=0.5] at (85:2) {\small $5$};
\node[draw,shape=circle, inner sep=0.5] at (58:2.8) {\small $6$};
\node[draw,shape=circle, inner sep=0.5] at (31:2.3) {\small $7$};

\end{tikzpicture}
\caption{Construct double pentagonal subdivision.}
\label{3abcF13b_tiling}
\end{figure}

The tiles $T_1,T_2$ and their rotations form a local tiling that is exactly the same as the tiling of a triangular face of the regular octahedron in Figure \ref{double_tiling}. The repeated construction around the new $\delta^3$ creates another such tiling of another triangular face next to the existing triangular face tiling. Therefore further repetitions give the double pentagonal subdivision tiling. The tiling we get at the end is first of Figure \ref{dsubdivision_tiling}. 

The argument for $f=48$ also applies to $f=120$. For $f=120$, the AVC is given by \eqref{avc120} 
\[
\text{AVC}=\{\beta^2\epsilon,\gamma^2\delta,\delta^3,\alpha^4,\epsilon^5\}.
\]
Compared with the AVC for $f=48$, we see $\alpha\beta^2$ is missing, and $\epsilon^4$ is replaced by $\epsilon^5$. The argument still starts with the same assumption on the neighbourhood of one $\delta^3$ vertex. We need to pay attention to where $\epsilon^4$ is used for $f=48$. The place we use $\epsilon^4$ is $\epsilon_2\epsilon_{4'}\cdots=\epsilon^2\cdots=\epsilon^4$. For the new AVC, this becomes $\epsilon_2\epsilon_{4'}\cdots=\epsilon^5$. Therefore the argument is valid, and we still get the double pentagonal subdivision tiling. The tiling we get at the end is the second of Figure \ref{dsubdivision_tiling}.

For $f=72$, we have
\[
\text{AVC}=\{\beta^2\epsilon,\gamma^2\delta,\delta^3,\alpha^4,\delta\epsilon^3\}.
\]
The argument for $f=48$ is valid until $\epsilon_2\epsilon_{4'}\cdots=\delta\epsilon^3$. 
The AVC implies $\beta\gamma\cdots,\gamma\epsilon\cdots$ are not vertices. Therefore the AAD of $\delta\epsilon^3$ is $\thin^{\gamma}\epsilon^{\delta}\thin^{\beta}\delta^{\epsilon}\thin^{\delta}\epsilon^{\gamma}\thin\epsilon\thin=\thin^{\delta}\epsilon^{\gamma}\thin\epsilon\thin^{\gamma}\epsilon^{\delta}\thin^{\beta}\delta^{\epsilon}\thin$. The AAD of $\thin^{\delta}\epsilon^{\gamma}\thin\epsilon\thin^{\gamma}\epsilon^{\delta}\thin$ implies $\gamma\thin\gamma\cdots$ is a vertex. On the other hand, by the edge length consideration, the AVC implies $\gamma^2\cdots=\gamma^2\delta=\thin\gamma\dash\gamma\thin\delta\thin$. Therefore $\gamma\thin\gamma\cdots$ is not a vertex. The contradiction proves that there is no tiling for $f=72$.

For general $f$ (i.e., $f\ne 48,72,120$), we have 
\[
\text{AVC}=\{\beta^2\epsilon,\gamma^2\delta,\delta^3,\alpha^4\}.
\]
Again the argument for $f=48$ is valid until $\epsilon_2\epsilon_{4'}\cdots$. Since $\epsilon^2\cdots$ is not in the AVC, we get a contradiction.

This completes the discussion for Case 1.3, $H=\alpha^4$. The conclusion is the double pentagonal subdivision tilings for $f=48$ and $120$. 

\subsubsection*{Case 3. $H=\delta^3\epsilon,\delta\epsilon^3, \delta^4\epsilon,\delta\epsilon^4$}

If $H=\delta^3\epsilon$ in the second of Figure \ref{3abcF3}, then $A_{6,12}=\delta$. By the edge length consideration, this implies $\delta_5\epsilon_1\cdots=\delta_5\epsilon_1\epsilon_6=\delta\epsilon^2$. Then the angle sums of $\delta^3\epsilon,\delta\epsilon^2$ and $\delta+\epsilon=(1+\tfrac{4}{f})\pi$ imply $f=20$, a contradiction. We get the same contradiction for $H=\delta\epsilon^3$.

The final case we need to consider is Case 3, $H=\delta^4\epsilon,\delta\epsilon^4$. After finishing all the other cases for $a^3bc$, we may assume that there is no $3^5$-tile and no $3^44$-tile. By Lemma \ref{base_tile3}, this implies $f\ge 60$. On the other hand, for $H=\delta^4\epsilon$, the argument for $H=\delta^3\epsilon$ can be used to show that $\delta\epsilon^2$ is a vertex. Then the angle sums of $\delta^4\epsilon,\delta\epsilon^2$ and $\delta+\epsilon=(1+\tfrac{4}{f})\pi$ imply $f=28$, contradicting $f\ge 60$. We get the same contradiction for $H=\delta\epsilon^4$.

\end{document}